\documentclass[11pt,letterpaper]{article}

\usepackage{times}
\selectfont
\usepackage{microtype}

\usepackage{amsmath,amssymb,amsthm,mathtools,bm}
\usepackage{bbm}

\usepackage{enumitem}
\usepackage{booktabs}
\usepackage{graphicx}
\usepackage{xcolor}
\usepackage{array}
\usepackage{tabularx}
\usepackage{float}
\usepackage{caption} 

\usepackage{titlesec}
\usepackage{titling}
\usepackage{titletoc}
\usepackage{tikz}
\usetikzlibrary{arrows.meta,positioning,fit,calc}
\usepackage[most]{tcolorbox}

\usepackage[numbers,sort&compress]{natbib}
\usepackage{hyperref}
\usepackage[capitalize,noabbrev]{cleveref}
\usepackage{fancyhdr}

\definecolor{linkblue}{RGB}{32,79,135}
\hypersetup{
 colorlinks=true,
 hypertexnames=false,
 linkcolor=linkblue,
 citecolor=linkblue,
 urlcolor=linkblue,
 pdftitle={Exact Minimax Quickest Change Detection under Pollak's Criterion},
 pdfauthor={Ali Tajer}
}

\setlist{itemsep=0.25em,topsep=0.35em}

\titleformat{\section}
  {\normalfont\bfseries\large}
  {\thesection}{0.72em}{}
\titlespacing*{\section}{0pt}{1.35em}{0.55em}

\titleformat{\subsection}
  {\normalfont\bfseries\normalsize}
  {\thesubsection}{0.72em}{}
\titlespacing*{\subsection}{0pt}{1.15em}{0.40em}

\titleformat{\subsubsection}
  {\normalfont\bfseries\normalsize}
  {\thesubsubsection}{0.72em}{}
\titlespacing*{\subsubsection}{0pt}{1.0em}{0.35em}

\titlecontents{section}
  [0pt]{\addvspace{0.05em}}
  {\contentslabel{1.8em}}{}
  {\titlerule*[0.55pc]{.}\contentspage}
\titlecontents{subsection}
  [1.8em]{}
  {\contentslabel{2.6em}}{}
  {\titlerule*[0.55pc]{.}\contentspage}

\newtheoremstyle{mimotheorem}
  {0.70em}{0.70em}{\itshape}{0pt}{\bfseries}{.}{0.45em}
  {\thmname{#1}\thmnumber{ #2}\thmnote{ \normalfont(#3)}}
\theoremstyle{mimotheorem}
\newtheorem{theorem}{Theorem}[section]
\newtheorem{lemma}[theorem]{Lemma}
\newtheorem{proposition}[theorem]{Proposition}
\newtheorem{corollary}[theorem]{Corollary}
\newtheorem{definition}[theorem]{Definition}
\newtheorem{assumption}[theorem]{Assumption}
\newtheorem{example}[theorem]{Example}

\newtheoremstyle{mimoremark}
  {0.65em}{0.65em}{\normalfont}{0pt}{\itshape}{.}{0.45em}
  {\thmname{#1}\thmnumber{ #2}\thmnote{ \normalfont(#3)}}
\theoremstyle{mimoremark}
\newtheorem{remark}[theorem]{Remark}

\crefname{theorem}{Theorem}{Theorems}
\Crefname{theorem}{Theorem}{Theorems}
\crefname{lemma}{Lemma}{Lemmas}
\Crefname{lemma}{Lemma}{Lemmas}
\crefname{proposition}{Proposition}{Propositions}
\Crefname{proposition}{Proposition}{Propositions}
\crefname{corollary}{Corollary}{Corollaries}
\Crefname{corollary}{Corollary}{Corollaries}
\crefname{definition}{Definition}{Definitions}
\Crefname{definition}{Definition}{Definitions}
\crefname{assumption}{Assumption}{Assumptions}
\Crefname{assumption}{Assumption}{Assumptions}
\crefname{remark}{Remark}{Remarks}
\Crefname{remark}{Remark}{Remarks}
\crefname{example}{Example}{Examples}
\Crefname{example}{Example}{Examples}

\definecolor{mimoboxborder}{RGB}{91,119,132}
\definecolor{mimoboxback}{RGB}{241,246,248}
\newtcolorbox{resultbox}[1][]{
  enhanced,
  breakable,
  colback=mimoboxback,
  colframe=mimoboxborder,
  boxrule=0.55pt,
  arc=0pt,
  outer arc=0pt,
  left=8pt,right=8pt,top=7pt,bottom=7pt,
  before skip=0.65em,
  after skip=0.65em,
  #1
}

\numberwithin{equation}{section}
\makeatletter
\@removefromreset{equation}{section}

\makeatother

\newcommand{\Pinf}{\mathbb P_{\infty}}
\newcommand{\Einf}{\mathbb E_{\infty}}

\newcommand{\Ezero}{\mathbb E_{0}}
\newcommand{\one}{\mathbf 1}

\newcommand{\dd}{\,\mathrm d}

\newcommand{\Geom}{\operatorname{Geom}}
\newcommand{\Law}{\operatorname{Law}}

\title{\bfseries\LARGE Pollak's Minimax Quickest Change Detection:\\ Non-Asymptotic Optimality}
\author{Ali Tajer\\ \emph{Electrical, Computer, and Systems Engineering Department}\\
\emph{Rensselaer Polytechnic Institute}}
\date{\today}

\renewenvironment{abstract}
  {\vspace{-1.8em}\begin{center}\small\bfseries Abstract\end{center}\vspace{-0.55em}%
   \begin{list}{}{%
     \setlength{\leftmargin}{0.38in}%
     \setlength{\rightmargin}{0.38in}%
   }\item\relax\small}
  {\end{list}}

\begin{document}
\maketitle

\begin{abstract}

Although Pollak's minimax formulation is one of the central frameworks in quickest change detection (QCD), the strongest general optimality results available for it are predominantly \emph{asymptotic}, applying as the average run-length-to-false-alarm constraint, $\gamma$, tends to infinity. Exact results for finite $\gamma$ have previously been available only for special models or restricted regimes. This paper addresses the general finite-$\gamma$ problem over the complete class of randomized, history-dependent stopping rules for i.i.d. change models. The key technical development is a \emph{survival-process representation} that recasts the optimization of Pollak's minimax criterion over stopping times as an equivalent linear variational optimization over admissible survival processes. Although this formulation is infinite-dimensional, it establishes the existence of an optimizer and provides an exact characterization of the finite-$\gamma$ Pollak minimax value. This characterization, in turn, provides a principled basis for constructing computable stopping rules whose performance can be made arbitrarily close to the optimum.  The resulting rules are driven by a recursively updated weighted likelihood-ratio statistic with generally time-varying injections and boundaries. Importantly, this structure is not imposed a priori: the optimization is carried out over the full class of randomized, history-dependent stopping rules, and the Shiryaev--Roberts form emerges naturally from the solution. In particular, the classical Shiryaev--Roberts recursion arises as the time-homogeneous special case. Finally, under a likelihood-ratio floor condition, the framework yields closed-form exact minimax solutions for a nontrivial class of change models.
\end{abstract}

\section{Introduction}
\label{sec:introduction}

Quickest change detection (QCD) is concerned with detecting, as quickly as possible, a change in the statistical distribution of a sequence of observations, subject to a prescribed level of protection against false alarms. Two principal formulations have shaped QCD development. In the Bayesian formulation, the change point is assigned a prior distribution and the detection rule is designed to minimize a Bayesian risk that balances detection delay and false alarms \cite{girshick1952,shiryaev1961,shiryaev1963}. In the minimax formulation, no prior is imposed on the change point, and performance is instead guaranteed against its least favorable location. The two canonical minimax criteria are due to Lorden \cite{lorden1971} and Pollak \cite{pollak1985}. Lorden's formulation further protects against the least favorable pre-change observation history, whereas Pollak's formulation measures the conditional average delay given survival to the change and takes the worst case only over the change time. 

The exact minimax problem under Lorden's criterion was resolved by Moustakides, who established the optimality of the cumulative sum (CuSum) procedure \cite{moustakides1986}. Pollak's formulation has proved more difficult, and despite an extensive literature and increasingly sharp asymptotic results \cite{pollak1985,polunchenkoTartakovsky2012,tartakovskyEtAl2012}, its exact minimax solution at a prescribed finite false-alarm level has remained unresolved in general. The existing literature has developed various decision rules in the general class of the Shiryaev--Roberts (SR) rules. In this class, by denoting the one-observation likelihood ratio at time $n$ by $\Lambda_n$, the classical SR statistic follows the recursion
\begin{align}
R_n=(1+R_{n-1})\Lambda_n
\end{align}
with $R_0=0$, and the procedure stops when $R_n$ first crosses a constant threshold $A$, i.e.,
\begin{align}
    T_A=\inf\{n\ge 1:R_n\ge A\}\ .
\end{align}
Two important variants of this rule modify the initialization while retaining the same recursion and constant-boundary first-crossing architecture. A deterministically head-started generalized SR procedure takes $R_0=h>0$. This family is often denoted SR-$r$ in the literature. Another variation is the Shiryaev--Roberts--Pollak (SRP) procedure, which instead randomizes the initial state $R_0$ according to the 
the distribution of the SR statistic falling below $A$  under the no-change model.

distribution of the statistic $R_n$ falling below the decision threshold. We use the SR family as a recurring benchmark below, while distinguishing carefully between the classical SR rule, deterministic head starts, and the quasi-stationary SRP initialization.

\vspace{.05 in}
\noindent \textbf{Existing literature.} Pollak's original construction used the SRP initialization to make the conditional expected delay an equalizer across change times and showed, under appropriate regularity conditions, that the additive gap between the SRP delay and the minimax value is $o(1)$ as the false-alarm level $\gamma\to\infty$ \cite{pollak1985}. The quasi-stationary construction and its convergence properties were subsequently studied further in \cite{pollakSiegmund1986}. In related work, \cite{pollak1987} derived large-threshold asymptotic expressions for false-alarm average run lengths and immediate-change delays of SR-type and mixture procedures in one-parameter exponential families. Early numerical and finite-sample comparisons under Pollak-type criteria were developed in, among others, \cite{mevorachPollak1991} and \cite{srivastavaWu1993}.

A separate line of work established exact optimality properties of SR procedures under criteria related to, but distinct from, Pollak's original minimax criterion. In particular, \cite{moustakides2008} and \cite{siegmundYakir2008} developed generalized-Bayes and minimax arguments clarifying the role of SR-type procedures, and \cite{pollakTartakovsky2009} established exact optimality of the classical SR procedure for an integral average-delay criterion and for the stationary multi-cyclic regime in which a distant change is preceded by repeated false-alarm cycles. These results help explain the strong late-change behavior of the SR recursion, but they do not solve Pollak's finite-$\gamma$ minimax problem.

Finite-$\gamma$ results have also highlighted the importance of initialization. Polunchenko and Tartakovsky \cite{polunchenkoTartakovsky2010} developed a general lower bound based on an auxiliary integral criterion optimized by a deterministically head-started SR procedure and showed that, for an exponential model, a suitable head start makes this lower bound coincide with Pollak's criterion over a restricted false-alarm range. The generalized SR family was studied numerically in greater detail by \cite{moustakidesEtAl2011}, including the classical SR rule, optimized deterministic head starts, SRP, CuSum, and exponentially weighted moving-average schemes. Their results showed that a carefully chosen deterministic head start can improve upon SRP in finite-$\gamma$ regimes. Building on these developments, \cite{tartakovskyEtAl2012} established third-order asymptotic optimality for both SRP and suitably designed head-started generalized SR procedures.

These results leave a basic question open: \textbf{for a prescribed finite false-alarm level $\gamma$, what is Pollak's minimax risk over the complete class of randomized, history-dependent stopping rules?} The associated questions are equally important: is the infimum achieved; how can the unrestricted value be computed; and what sequential structure emerges when no SR form is imposed in advance?

\vspace{.05 in}
\noindent \textbf{Novel methodology by introducing survival processes: approach.} Our starting point is that Pollak's criterion depends critically on whether the detector is still active when the change occurs. We therefore describe a stopping rule via its \emph{survival process}, defined as follows. For each observation history up to time $n$, $y_n$ is the conditional probability that the detector remains active beyond time $n$. This changes the problem's geometry. All metrics used to characterize Pollak's minimax problem become related to linear functions of the survival process $\{y_n:n\in\mathbb{N}\}$. Specifically, the no-change average run-length (ARL) becomes a linear function of the survival process, and the every conditional-delay criterion becomes the ratio of two functionals. This leaves the ratio defining each conditional delay as the only nonlinearity.

\begin{figure}[t]
\centering
\begin{tikzpicture}[
    node distance=12mm and 5.5mm,
    >=Latex,
    every node/.style={font=\footnotesize},
    wideblock/.style={
        draw,
        rounded corners,
        align=center,
        minimum height=11mm,
        minimum width=24mm,
        inner sep=3pt
    },
    detblock/.style={
        draw,
        rounded corners,
        align=center,
        minimum height=12mm,
        minimum width=38mm,
        inner sep=4pt
    },
    note/.style={
        draw,
        dashed,
        rounded corners,
        align=left,
        inner sep=4pt
    },
    line/.style={->, thick}
]


\node[wideblock] (pollak) {%
Pollak minimax QCD\\
$\displaystyle \inf_T J_{\sf P}(T)$\\
s.t. $\mathbb E_\infty[T]\ge\gamma$
};

\node[wideblock, right=of pollak] (survival) {%
Survival-process\\
reformulation\\
$T\longleftrightarrow\mathbf y$
};

\node[wideblock, right=of survival] (fixedr) {%
Fix candidate\\
risk $r$\\
$\mathbf y\in\mathcal F_{\le r}$
};

\node[wideblock, right=of fixedr] (frontier) {%
Frontier value\\
$\displaystyle
\Gamma(r)=
\sup_{\mathbf y\in\mathcal F_{\le r}}
\operatorname{ARL}(\mathbf y)
$
};

\node[wideblock, right=of frontier] (crossing) {%
First crossing\\
$\displaystyle
r_\gamma=
\inf\{r:\Gamma(r)\ge\gamma\}
$
};

\coordinate (lowerrow)
    at ($(fixedr.south)+(0,-21mm)$);

\node[
    note,
    anchor=north
] (truncate)
at ($(lowerrow)+(-40mm,10mm)$) {%
\textbf{Computational route:}
truncate at depth $N$\\
$\Rightarrow$ finite-dimensional optimization\\
$\Rightarrow$ compute $\Gamma_N(r)$\\
$\Rightarrow$ search for $r_{\gamma,N}$\\
$\Rightarrow$ approximation improves as $N$ increases
};

\node[
    detblock,
    anchor=north
] (detector)
at ($(lowerrow)+(40mm,0)$) {%
\textbf{Optimal structure / detector}\\
recover $\mathbf y^\star$ and the\\
corresponding stopping rule
};

\draw[line] (pollak) -- (survival);
\draw[line] (survival) -- (fixedr);
\draw[line] (fixedr) -- (frontier);
\draw[line] (frontier) -- (crossing);


\draw[line]
    (truncate.east)
    -- ++(5mm,0)
    -- ++(0,20mm)
    -- ([xshift=-6mm]fixedr.south);


\draw[line]
    ([xshift=4mm]fixedr.south)
    -- ([xshift=4mm,yshift=-28mm]fixedr.south)
    -- ([yshift=0mm]detector.west);
    

\coordinate (detinR) at ([yshift=-0mm]detector.east);

\draw[line]
    (crossing.south)
    -- (crossing.south |- detinR)
    -- (detinR);


\node[
    draw,
    rounded corners,
    inner sep=6pt,
    fit=(pollak)(survival)(fixedr)(frontier)(crossing),
    label={[font=\footnotesize]above:Conceptual pipeline}
] {};

\end{tikzpicture}

\caption{Pollak's minimax
problem is recast in terms of survival processes. For each candidate risk level $r$, the frontier $\Gamma(r)$ gives the largest achievable false-alarm average run length subject to all conditional-delay constraints
being bounded by $r$. The first risk level for which the frontier reaches the prescribed false-alarm level $\gamma$ gives the minimax value $r_\gamma$. For computation, truncation at depth $N$ produces the
approximating frontier $\Gamma_N(r)$ and crossing $r_{\gamma,N}$.}
\label{fig:method-pipeline}
\end{figure}
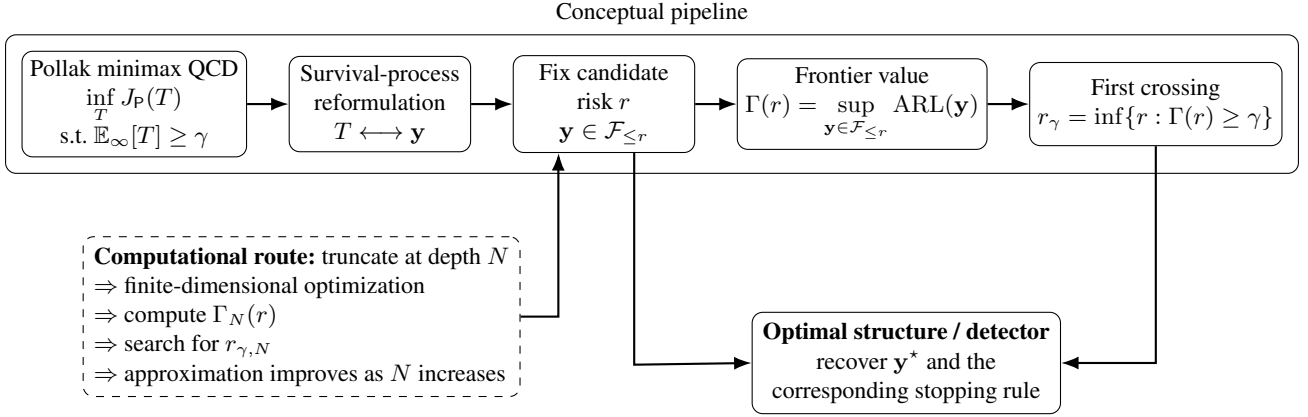

The next step is to reverse the minimax problem. Rather than directly minimizing the worst conditional-delay ratio, fix a candidate delay level $r$ and ask a simpler question: \emph{how large an ARL can be supported if every change-time conditional delay must be no larger than $r$?} For fixed $r$, every ratio constraint becomes a linear inequality, producing a monotone risk--false-alarm frontier $\Gamma(r)$. The optimal $r$ is then not chosen heuristically. It is recovered by a one-dimensional \emph{outer search} based on the observation that the Pollak minimax value is the first risk level at which this frontier reaches the required ARL $\gamma$,

Thus the method has two conceptually distinct layers: an \emph{inner} optimization that designs the best detector for a fixed candidate risk $r$, and an \emph{outer} scalar search that finds the smallest $r$ capable of supporting the prescribed false-alarm level. The computational development follows exactly the same logic after truncation: solve the truncated fixed-$r$ problem, update the bracket on $r$, and repeat until the frontier crossing is localized to the desired accuracy. A summary of this approach is depicted in Figure~\ref{fig:method-pipeline}.

\noindent \textbf{Contributions.} Besides the novel methodology, the contributions of this paper can be organized into four levels.

\emph{1. Exact finite-$\gamma$ characterization.} We solve Pollak's minimax problem for every finite $\gamma>1$ over the complete class of randomized, history-dependent stopping rules. The survival-process reformulation converts the fixed-$r$ problem into a linear variational optimization. We establish the existence of a frontier optimizer, prove that the frontier reaches its first crossing, and show that the exact minimax value is obtained by the outer scalar crossing search described above. 

\emph{2. Structural characterization of arbitrarily near-optimal rules.} Truncating the variational problem and dualizing it reveals a recursively updated weighted likelihood-ratio statistic
\begin{align}
M_n = \omega_n+\Lambda_n M_{n-1},
\end{align}
with time-dependent dual injections $\{\omega_n:n\in\mathbb{N}\}$ and generally time-dependent stopping boundaries, i.e., the optimal stopping rule will be of the form
\begin{align}
    T_{\rm optimal}=\inf\{n\ge 1:M_n > b_n\}\ .
\end{align}
This statistic generalizes the RS family of statistics by having time-varying injections $\omega_n$ at every time. It is noteworthy that \emph{this emerges as an optimal choice without imposing any structure a priori on the test}. Specifically, constant injections and constant boundaries reduce the recursion to the classical SR recursion, explaining why SR-type rules arise naturally without being imposed in advance.

\emph{3. Constructive approximation and error control.} The exact frontier is infinite-dimensional. We introduce a truncation order $N$, prove explicit tail bounds, and show that the truncated crossing values $r_{\gamma,N}$ decrease monotonically to $V_{\sf P}(\gamma)$. For finite alphabets, the truncated primal is a finite-dimensional linear program. For general observation spaces, the associated dual optimization is finite-dimensional in the $N$ non-negative dual weights, with the required expectations evaluated analytically or numerically. Each truncated optimizer can be converted into an admissible stopping rule whose Pollak risk converges to the unrestricted optimum.

\emph{4. Explicit exact solutions.} Under a likelihood-ratio floor condition, we derive an explicit information-theoretic lower bound and identify conditions under which it is attained. This gives closed-form exact minimizers for a nontrivial class of models, including a range of Bernoulli models. In one such Bernoulli model, we prove that the unrestricted optimum is strictly smaller than the best value achievable within the canonical constant-boundary generalized SR class, demonstrating that the additional degrees of freedom exposed by the variational formulation can be essential at finite $\gamma$.

The remainder of the paper is organized as follows.
Section~\ref{sec:model} introduces the statistical model, randomized stopping rules, and Pollak's minimax criterion.
Section~\ref{sec:survival} develops the survival-process representation.
Section~\ref{sec:fixed-risk} formulates the fixed-risk variational problem and establishes the exact finite-$\gamma$ minimax characterization and existence of an optimal rule, including the outer search over $r$.
Section~\ref{sec:truncated} develops truncated approximations with explicit error bounds and near-optimal admissible rules.
Section~\ref{sec:finite-dual} characterizes the off-boundary structure of truncated optimizers through a recursively updated weighted likelihood-ratio statistic, with the SR recursion arising as a time-homogeneous special case.
Section~\ref{sec:likelihood-floor} derives explicit all-rules lower bounds and exact solutions under a likelihood-ratio floor, including the Bernoulli specialization.
Section~\ref{sec:implementation} develops the Pollak variational frontier algorithm, separating the outer search in $r$ from the inner fixed-$r$ synthesis and discussing rigorous two-sided bracketing.
Finally, Section~\ref{sec:numerical} presents exact and numerical comparisons with classical SR-type and CuSum procedures. Section~\ref{sec:conclusion} provides some concluding remarks.

\section{Model and Notations}
\label{sec:model}

Let $(\mathsf X,\mathcal X)$ be a standard Borel observation space. The pre-change and post-change laws are $P_\infty$ and $P_0$, with $P_0\ll P_\infty$. We observe $\{X_n:n\in\mathbb{N}\}$ sequentially. For an unknown deterministic $\nu\in\mathbb N_0:= {\mathbb N}\cup \{0\}$, under $\mathbb P_\nu$, we have
\begin{align}
X_1,\ldots,X_\nu
\stackrel{\mathrm{i.i.d.}}{\sim}P_\infty
\qquad \mbox{and} \qquad
X_{\nu+1},X_{\nu+2},\ldots
\stackrel{\mathrm{i.i.d.}}{\sim}P_0
\; ,
\label{eq:model}
\end{align}
and all observations are independent. Throughout the paper, unless explicitly stated otherwise, we assume that $P_\infty$ and $P_0$ are mutually absolutely continuous. Randomized procedures play an important role in Pollak's formulation, including the randomized initialization used by SRP. Let $\{U_n:n\in\mathbb N_0\}$ be independent $\operatorname{Unif}(0,1)$ random variables, independent of the observations, and let $\mathbb P_U$ denote their product law. We use
\begin{align}
\mathcal F_n
:=
\sigma(U_0,X_1,U_1,\ldots,X_n,U_n)
\; .
\label{eq:behavioral-filtration}
\end{align}
A randomized detection rule $T$ is an $\mathbb N\cup\{\infty\}$-valued stopping time with respect to $\{\mathcal F_n:n\in\mathbb N_0\}$.
Thus $U_0$ may randomize the detector's initialization or select among behavioral branches, but the detector does not stop before the first observation. The class of such rules, denoted by $\mathcal T$, includes non-randomized procedures, randomized initialization, and fresh randomized actions at subsequent times. Next, we define the likelihood ratio terms 
\begin{align}
\label{eq:likelihoods1} \Lambda_n  := 
\frac{\dd P_0}{\dd P_\infty}(X_n), \quad \forall n\in\mathbb{N} \; , \qquad \mbox{and} \qquad L_{\nu,n}
:= 
\prod_{j=\nu+1}^{n}\Lambda_j,
\quad \forall n>\nu \; ,
\end{align}
with $L_{\nu,\nu}=1$. Accordingly, for a stopping time $T$, define
\begin{align}
q_n(T)
:= 
\Pinf(T>n) \; .
\label{eq:q}
\end{align}
\begin{definition}[Pollak risk]
For every finite change time $\nu$, define
\begin{align}
\label{eq:A}
A_\nu(T)
&:=
\mathbb E_\nu[(T-\nu)\one_{\{T>\nu\}}]
=
\sum_{n=\nu}^{\infty}\mathbb P_\nu(T>n)
\; .
\end{align}
Whenever $q_\nu(T)>0$, the conditional detection delay associated with a change at time $\nu$ is
\begin{align}
\label{eq:D}
D_\nu(T)
&:=
\mathbb E_\nu[T-\nu\mid T>\nu]
=
\frac{A_\nu(T)}{q_\nu(T)}
\; .
\end{align}
The \emph{Pollak risk} of the stopping rule $T$ is defined as
\begin{align}
J_{\sf P}(T)
&:=
\sup_{\nu\ge0}D_\nu(T)
\; .
\end{align}
If $q_\nu(T)=0$ for any finite $\nu$, we set $J_{\sf P}(T)=\infty$. Finally, for a prescribed ARL $\gamma>1$, the optimal Pollak minimax value is
\begin{align}
V_{\sf P}(\gamma)
&:=
\inf_{T:\,\Einf [T]\ge\gamma}J_{\sf P}(T)
\; .
\label{eq:value}
\end{align}
\end{definition}
\section{Survival-Process Representation}
\label{sec:survival}

The direct optimization of Pollak's criterion over stopping times is difficult for two related reasons. First, a stopping rule is a history-dependent sequential object: at each time, the decision to stop or continue may depend on the entire observation history and any auxiliary randomization used by the detector. Second, Pollak's criterion depends not only on when the procedure stops, but also on the probability that it \emph{survives} without stopping up to each possible change time. 

This inspires describing a detection rule in terms of its survival behavior rather than its stopping time directly. The key observation in this section is that, once the observation history is fixed, all remaining randomness in the stopping decision comes from the detector's auxiliary randomization. Averaging over this randomization yields a conditional survival probability for each observation history. As we will show below, conditional survival probabilities completely characterize every quantity that enters Pollak's criterion. To formalize this representation, corresponding to any observation history $x_{1:n}:= (x_1,\ldots,x_n)$ 
we define the \emph{survival process} induced by a stopping rule $T$ as follows.

\begin{definition}[Survival process]
Because $T$ is a stopping time, for every $n\ge1$ the event $\{T>n\}$ is $\mathcal F_n$-measurable. Hence its indicator can be represented by a measurable function
\begin{align}
s_n(X_{1:n},U_{0:n})
=
\mathbf 1_{\{T>n\}}
\qquad\text{almost surely}.
\end{align}
For an observation history $x_{1:n}$, define
\begin{align}
y_n(x_{1:n})
:=
\int
s_n(x_{1:n},u_{0:n})
\,\mathbb P_U(\mathrm d u_{0:n})
\; ,
\label{eq:survival-process}
\end{align}
and set $y_0:=1$. Thus $y_n(x_{1:n})$ is the probability, over the detector's own randomization, that the detector remains active beyond time $n$ when the observation history is held fixed.
\end{definition}

The distinction between the conditional survival probability $y_n$ and ordinary survival probability $q_n(T)$ defined in~\eqref{eq:q} is important. The quantity $y_n(x_{1:n})$ describes survival conditional on an observation history. It therefore retains the full history dependence of the detection rule. By contrast, averaging $y_n$ over the observation law produces the \emph{unconditional} survival probability. In particular, under the no-change law, we have $q_n(T)=\mathbb E_\infty[y_n(X_{1:n})]$. More generally, because $P_0\ll P_\infty$ makes the relevant finite-dimensional laws absolutely continuous with respect to $P_\infty^n$, the same versions give, under a change at time $\nu$,
\begin{align}
\mathbb P_\nu(T>n)
=
\mathbb E_\nu
\bigl[y_n(X_{1:n})\bigr]
\; .
\label{eq:survival-under-nu}
\end{align}
The survival process satisfies a simple consistency condition. The survival events are nested, $\{T>n\}\subseteq\{T>n-1\}$, and therefore their finite-history indicators can be chosen so that, after integrating over the detector randomization,
\begin{align}
0\le y_n\le  y_{n-1}\le1
\qquad
P_\infty^n\text{-almost surely},
\qquad n\ge1.
\label{eq:survival-compatibility}
\end{align}
When a pointwise version is convenient, compatible representatives satisfying these inequalities for every history can be chosen as described below. We note that when working with the $L^\infty(P_\infty^n)$ equivalence classes, $y_{n-1}$ is lifted to the $n$-observation space by ignoring its last coordinate.

\vspace{.05 in} \noindent \textbf{Alternative representation of Pollak's Problem.} The importance of the survival process is that all quantities involved in Pollak's problem can now be expressed directly in terms of $\{y_n:n\in\mathbb{N}_0\}$. Specifically, from \eqref{eq:A} and \eqref{eq:survival-under-nu} we have 
\begin{align}
A_\nu(T)
=
\sum_{n=\nu}^{\infty}
\mathbb P_\nu(T>n)
=\sum_{n=\nu}^{\infty}
\mathbb E_\nu
\bigl[
y_n(X_{1:n})
\bigr]
\; ,
\label{eq:A-survival}
\end{align}
which, in conjunction with~\eqref
{eq:D} provides that whenever $q_\nu(T)>0$ we have
\begin{align}
D_\nu(T)
=\frac{1}{\displaystyle
\mathbb E_\infty \bigl[ y_\nu(X_{1:\nu})
\bigr] }\; \displaystyle
\sum_{n=\nu}^{\infty}
\mathbb E_\nu \bigl[y_n(X_{1:n})
\bigr]\; .
\label{eq:D-survival}
\end{align}
Similarly, the ARL under no change is
\begin{align}
\Einf [T]
=
\sum_{n=0}^{\infty}
q_n(T)
=
\sum_{n=0}^{\infty}
\Einf
\bigl[
y_n(X_{1:n})
\bigr]
\; .
\label{eq:arl-survival}
\end{align}
Next, we define the class of admissible survival processes. To emphasize the dependence on $\mathbf y=\{y_n:n\in\mathbb N_0\}\in\mathcal Y$, we also write $A_\nu(\mathbf y)$ and $D_\nu(\mathbf y)$ for the corresponding delay numerator and conditional delay, and $T_{\mathbf y}$ for a stopping-time realization of $\mathbf y$.
\begin{definition}[Admissible survival processes]
Let $\mathcal Y$ denote the collection of measurable sequences $\{y_n:n\in\mathbb N_0\}$, identified coordinatewise up to $P_\infty^n$-null sets, such that
\begin{align}
y_0=1
\qquad \mbox{and} \qquad
0\le y_n\le y_{n-1}\le1
\qquad
P_\infty^n\text{-almost surely},
\qquad n\ge1\ .
\label{eq:survival-space-ae}
\end{align}
 When we implement a stopping rule, we choose compatible measurable representatives that satisfy the same inequalities pointwise. Such representatives always exist by recursively clipping each representative to the interval $[0,y_{n-1}]$. Because $P_0\ll P_\infty$, these null-set modifications do not change any ARL or delay functional used below.
For a prescribed false-alarm level $\gamma>1$, define
\begin{align}
\mathcal Y_\gamma
:=
\Big\{
y\in\mathcal Y:
\sum_{n=0}^{\infty}
\Einf\!\Big[y_n(X_{1:n})\Big]
\ge
\gamma
\Big\}
\; .
\label{eq:Y-gamma}
\end{align}
\end{definition}

With the notations above, Pollak's minimax problem can be written entirely in terms of the survival process as
\begin{align}
V_{\sf P}(\gamma)
=
\inf_{y\in\mathcal Y_\gamma}
\left\{
\sup_{\nu\ge 0} \; 
\frac{1}{ \displaystyle
\mathbb E_\infty
\bigl[ y_\nu(X_{1:\nu})
\bigr] }\; \displaystyle
\sum_{n=\nu}^{\infty}
\mathbb E_\nu
\left[ y_n(X_{1:n}) \right]
\right\}
\; ,
\label{eq:value-survival}
\end{align}
Throughout the paper, if $q_\nu(\mathbf y)=0$ for any finite $\nu$, the corresponding Pollak risk is interpreted as $+\infty$, consistently with the convention in the definition of $J_{\sf P}$. Thus, after passing from $T$ to its associated survival process, both the false-alarm constraint and the numerator and denominator of every conditional-delay criterion become sums of expectations of the functions $y_n$.  Next, we show that the survival process is a complete representation of Pollak's minimax problem. Specifically, every randomized stopping rule generates an element of $\mathcal Y$, and conversely every element of $\mathcal Y$ admits a compatible representative that can be implemented by a randomized stopping rule. The following lemma formalizes these statements.

\begin{lemma}[Completeness of the survival representation]
\label{lemma:survival-completeness}
Let $\mathbf y\in\mathcal Y$. For any compatible pointwise representative of $\mathbf y$, there exists a randomized stopping time $T_{\mathbf y}\ge1$ whose conditional survival functions agree with that representative $P_\infty^n$-almost surely for every $n$. Conversely, every randomized stopping time $T\ge1$ induces an element $\mathbf y\in\mathcal Y$. Hence the survival-process formulation and the randomized stopping-time formulation have identical ARL and Pollak-delay functionals.
\end{lemma}

\begin{proof}
See Appendix~\ref{app:survival-completeness}.
\end{proof}
This lemma implies that optimization over randomized stopping rules is equivalent to optimization over the class $\mathcal Y$ of admissible survival processes. More specifically, it establishes that replacing stopping times by survival processes entails no loss of generality. Two randomized stopping rules that induce the same collection $\{y_n:n\in\mathbb{N}_0\}$ have identical survival probabilities under every change time and therefore have the same average run length, the same conditional delays, and the same Pollak risk. Accordingly, to solve Pollak's minimax problem, the survival process provides a complete representation of the detection rule. This reformulation allows us to replace optimization over a class of randomized stopping times with optimization over measurable functions that satisfy the simple order constraints in \eqref{eq:survival-compatibility}.

\section{Fixed-Risk Reformulation}
\label{sec:fixed-risk}

The survival-process representation enables replacing the false-alarm mean as well as the numerator and denominator of every Pollak delay with a linear function of $\mathbf y$.  The problem, nevertheless, remains nonlinear because the risk for every change time $\nu$ is a ratio, and optimization is over the worst-case of these ratios:
\begin{align}
D_\nu(\mathbf y)
=
\frac{A_\nu(\mathbf y)}{q_\nu(\mathbf y)}
\; .
\label{eq:pollak-ratio-y}
\end{align}
In this section, we avoid directly minimizing over these ratios. Instead, we first fix a candidate risk level $r$ and ask how much false-alarm protection can be achieved while requiring every conditional detection delay to be at most $r$. For fixed $r$, each ratio constraint becomes a linear inequality. This produces a scalar frontier that describes the trade-off between the allowable detection delay and the largest achievable ARL. We then recover the original minimax problem by identifying the first point at which this frontier reaches the prescribed ARL level $\gamma$. Using the expansions of $q_n(T)$, $A_\nu(T)$, $D_\nu(T)$, and $\Einf[T]$, for an admissible survival process $\mathbf y\in\mathcal Y$, we introduce an \emph{infinite-dimensional} linear program as follows.

\subsection{Fixed-Risk ARL Frontier}

Fix a candidate risk level $r\ge1$. Whenever $q_\nu(\mathbf y)>0$, the requirement $D_\nu(\mathbf y) \le r$ is equivalent to $A_\nu(\mathbf y)
\le r q_\nu(\mathbf y)$, which indicates that for a fixed $r$, each conditional-delay constraint $D_\nu(\mathbf y) \le r$
becomes a linear inequality in the survival process $\{y_n:n\in\mathbb{N}_0\}$. We refer to $A_\nu(\mathbf y)
\le r q_\nu(\mathbf y)$ as the \emph{fixed-risk constraint for $\nu$}. This is a basis for specifying an infinite-dimensional linear problem. Before formalizing this problem, we note one technical issue pertinent to $q_\nu=0$ at a finite $\nu$. The original Pollak problem assigns infinite risk to a stopping rule for which $q_\nu=0$ at any finite $\nu$.  To obtain a solution to the fixed-risk optimization, however, it is helpful to \emph{temporarily} permit survival probabilities to vanish at a finite time. The role of this relaxation and its subsequent repair will be discussed in Section~\ref{sec:tail-repair}. If $q_\nu(\mathbf y)=0$, then the non-negativity of $y_\nu$ implies that $y_\nu=0$ almost surely.  The survival consistency condition in~\eqref{eq:survival-compatibility} then forces all subsequent survival functions to vanish along the corresponding histories, which in turn also implies $A_\nu(\mathbf y)=0$. Hence the corresponding fixed-risk constraint becomes the valid equality $A_\nu(\mathbf y) = r q_\nu(\mathbf y) = 0$.

Including $q_\nu(\mathbf y)=0$ therefore enlarges the original feasible class of finite-risk Pollak rules. We refer to this enlarged feasible class as the \emph{extended} fixed-risk class.
We denote this \emph{extended} feasible set by
\begin{align}
\mathcal F_{\le r}
:=
\left\{
\mathbf y\in\mathcal Y:
A_\nu(\mathbf y)
\le
r q_\nu(\mathbf y) \; ,
\quad
\forall\,\nu\in\mathbb N_0
\right\}
\; ,
\label{eq:fixed-risk-feasible-set}
\end{align}
in which $\mathbf{y}$ must satisfy an infinite number of linear constraints.  Subsequently, for a given $r\geq 1$, we define the \emph{fixed-risk} ARL frontier as follows.
\begin{definition}[Pollak variational frontier]
\label{def:Gamma}
For $r\ge1$, define the PVF as
\begin{align}
\Gamma(r)
:=
\sup_{\mathbf y\in\mathcal F_{\le r}}
\Einf[T]
=
\sup_{\mathbf y\in\mathcal F_{\le r}}
\sum_{n=0}^{\infty}q_n(\mathbf y)
\; .
\label{eq:Gamma}
\end{align}
\end{definition}
Therefore, $\Gamma(r)$ is the largest ARL that can be supported while requiring the conditional detection delay associated with every change time to be no larger than $r$. For fixed $r$, this is an infinite-dimensional linear program: the objective, the survival-process constraints defining $\mathcal Y$, and all constraints defining $\mathcal F_{\le r}$ are linear in $\mathbf y$; the ratio appearing in Pollak's conditional delay has disappeared; and each change time contributes one linear inequality in $\mathbf y$. Hence the original collection of infinitely many linear-fractional constraints is replaced by a countable collection of linear constraints.

It is important to separate two design variables that play different roles. The survival process $\mathbf y$ is the \emph{inner} design variable such that for a fixed candidate risk $r$, it is optimized to compute $\Gamma(r)$. The scalar $r$ is the \emph{outer} design variable in the sense that after the fixed-$r$ problem has been solved, $r$ is adjusted until the optimized ARL reaches the target $\gamma$. Thus the method does not assume that $r$ is known. Solving a single fixed-$r$ problem answers only one feasibility question, while solving Pollak's original minimax problem requires locating the first crossing of the monotone frontier. This distinction will be retained in both the theoretical characterization and the numerical algorithm.

\subsection{Frontier Crossing and the Minimax Lower Bound}

The frontier $\Gamma(r)$ converts the original minimax problem into a scalar tradeoff between risk and false-alarm protection. We next identify the risk level at which this frontier first supports the prescribed ARL requirement. In this subsection, we characterize the first risk level at which the required false-alarm mean $\gamma$ becomes feasible and show that this crossing point provides a universal lower bound on the Pollak minimax value. This is done by noting that the frontier $\Gamma(r)$ has a basic monotonicity property. Specifically, if $r_1\le r_2$, every process satisfying the more stringent risk bound $r_1$ also satisfies the bound $r_2$. Hence  $\mathcal F_{\le r_1}
\subseteq \mathcal F_{\le r_2}$ and consequently $\Gamma(r_1) \le \Gamma(r_2)$. Thus, increasing the allowable detection delay can only increase the maximum false-alarm protection that can be supported.

This monotonicity suggests a simple characterization of the original minimax problem. For a prescribed ARL requirement $\gamma>1$, the Pollak minimax value is recovered from the first crossing of the PVF with the prescribed ARL level $\gamma$, motivated by which we define
\begin{align}
r_\gamma
:=
\inf
\left\{
r\in[1,\gamma]:
\Gamma(r)\ge\gamma
\right\}
\; ,
\label{eq:r-gamma-preview}
\end{align}
based on which, $r_\gamma$ is the first risk level at which an ARL of at least $\gamma$ becomes feasible. Because $\Gamma(r)$ is nondecreasing, this is a one-dimensional monotone crossing problem. Conceptually, one starts with the bracket $[1,\gamma]$, evaluates or bounds the optimized frontier at trial values of $r$, and moves the lower endpoint upward when the target ARL is impossible and the upper endpoint downward when it is achievable. Section~\ref{sec:implementation} provides the computational counterpart for rigorous two-sided bounds when a trial point cannot be classified at finite truncation depth. Note that the search for $r$ can be restricted to the interval $[1,\gamma]$. The lower endpoint follows because, conditional on survival to the change time, the integer-valued residual delay is at least one. The upper endpoint is sufficient because an observation-independent geometric stopping rule with mean $\gamma$ satisfies $\Einf[T]= \gamma$ and $D_\nu(T) = \gamma$ for all $\nu\in\mathbb N_0$, based on which $\Gamma(\gamma) \ge \gamma$. This shows that the set in \eqref{eq:r-gamma-preview} is nonempty.

Noting that $r_\gamma$ is well defined, the central purpose of this reformulation is to show that this crossing point is exactly the solution of the original Pollak minimax problem, namely, $V_{\sf P}(\gamma)
= r_\gamma$. This identity reduces the original minimization over all randomized, history-dependent stopping rules to a one-dimensional search over the risk level $r$. The remainder of the analysis is devoted to establishing two inequalities: one inequality is $V_{\sf P}(\gamma)\ge r_\gamma$, which follows directly from the definition of the frontier, and the second one is its reverse inequality, i.e., $V_{\sf P}(\gamma)\le r_\gamma$, establishing which requires additional arguments concerning the construction of a valid stopping rule from a frontier optimizer. 

The key to showing these inequalities is to distinguish what the frontier $\Gamma(r)$ tells us from what is ultimately required to characterize $V_{\sf P}(\gamma)$. Recall that $V_{\sf P}(\gamma)$ asks: \emph{among genuine Pollak-admissible stopping rules whose ARL is at least $\gamma$, what is the smallest achievable Pollak risk?} By contrast, $\Gamma(r)$ asks the reverse question: \emph{if we require the Pollak risk to be no larger than $r$, what is the largest ARL that can be supported?} This distinction makes the inequality $V_{\sf P}(\gamma)\ge r_\gamma$ immediate: if  $\Gamma(r)<\gamma$ then even after optimizing over the entire fixed-risk class, risk level $r$ cannot support an ARL of $\gamma$. Therefore, no genuine stopping rule with ARL at least $\gamma$ can have Pollak risk at most $r$. Since this is true for every $r<r_\gamma$, we obtain $V_{\sf P}(\gamma) \ge r_\gamma$. This is formalized in the next lemma. 
\begin{lemma}[Frontier lower bound]
\label{lem:frontier-lower-bound}
For every $\gamma>1$, we have $V_{\sf P}(\gamma)
\ge r_\gamma$.
\end{lemma}
\begin{proof}
See Appendix~\ref{app:frontier-lower-bound}.
\end{proof}

\subsection{Existence of a Frontier Optimizer via Tail Control}
\label{sec:tail-optimizer}

To prove the reverse inequality $V_{\sf P}(\gamma)\le r_\gamma$, it is
enough to show that the frontier value $\Gamma(r)$ is achieved. Consider
therefore a feasible sequence $\{\mathbf y^{(k)}\}$ whose ARLs converge to
$\Gamma(r)$. The main difficulty is to pass to a limiting survival process
without losing ARL through increasingly late stopping times. We address this
in two steps: first, weak-* compactness yields convergence of the survival
coordinates at every fixed time; second, a uniform geometric tail bound
prevents survival probability from escaping to arbitrarily large times. Combined, these properties preserve both feasibility and ARL in the limit and yield a frontier optimizer. More extended discussions are provided in Appendix~\ref{app:additional_section_4_3}.
The first ingredient of our analysis is compactness, which indicates that an optimizing sequence contains a subsequence whose survival coordinates converge at every fixed time. Weak-* convergence is natural here because it preserves the finite-time expectations that enter both the ARL and the fixed-risk constraints. This is formalized in the next lemma.
\begin{lemma}[Sequential compactness of the survival space]
\label{lem:survival-compact}
Every sequence
$\{\mathbf y^{(k)}\}_{k\ge1}\subseteq\mathcal Y$ has a subsequence
$\{\mathbf y^{(k_j)}\}_{j\ge1}$ and a survival process
$\mathbf y^\star\in\mathcal Y$ such that, for every fixed $n$,
\begin{align}
y_n^{(k_j)}
\longrightarrow
y_n^\star
\qquad
\text{in the weak-* topology as }j\to\infty
\; .
\end{align}
\end{lemma}

\begin{proof}
See Appendix~\ref{app:survival-compact}.
\end{proof}
Lemma~\ref{lem:survival-compact} provides a candidate limiting survival
process, but coordinate-wise convergence alone does not control the infinite ARL tail.  To obtain such control, we first need to transfer a nontrivial stopping probability over a finite post-change block into a
uniformly positive stopping probability under the no-change law. To obtain uniform tail control, for an integer $M\ge1$ and $a\in(0,1)$ define 
\begin{align}
\tau_M(a)
:=
\inf
\left\{
\mathbb E_\infty^M[\Phi]:
0\le\Phi\le1 \; ,
\quad
\mathbb E_0^M[\Phi]\ge a
\right\}
\; .
\label{eq:transfer-modulus}
\end{align}

\begin{lemma}[Finite-block probability transfer]
\label{lem:transfer}
For every finite $M\ge1$ and every
$a\in(0,1)$, there exists a constant $c_{M,a}>0$ such that every $[0,1]$-valued random variable $\Phi$ satisfying $\mathbb E_0^M[\Phi]\ge a$ also satisfies
\begin{align}
\mathbb E_\infty^M[\Phi]
\ge
c_{M,a}
\; .
\end{align}
\end{lemma}

\begin{proof}
See Appendix~\ref{app:transfer}.
\end{proof}
Lemma~\ref{lem:transfer} provides exactly the probability-transfer
mechanism that is needed.  Combining it with the Pollak delay bound shows that, over
every sufficiently long block, a fixed positive fraction of the surviving
no-change probability must disappear, uniformly over the entire
fixed-risk class.
Fix a finite risk bound $R\ge1$, choose $m_R\ge1$ such that $m_R+1\ge2R$, and define
\begin{align}
\delta_R
:=
\tau_{m_R}\!\left(\frac12\right)
>
0
\; .
\label{eq:block-delta}
\end{align}
The Pollak delay bound forces a probability of at least $1/2$ of stopping over such a post-change block, while Lemma~\ref{lem:transfer} transfers a uniformly positive fraction of this probability to the no-change law. Iteration yields the following geometric tail bound.

\begin{lemma}[Uniform block contraction]
\label{lem:block-contraction}
For every finite $R\ge1$ and every
$\mathbf y\in F_{\le R}$,
\begin{align}
q_{n+m_R}(\mathbf y)
&\le
(1-\delta_R)q_n(\mathbf y) \; ,
\qquad n\ge0
\; ,
\label{eq:block-contraction}\\
\sum_{n=j m_R}^{\infty}
q_n(\mathbf y)
&\le
\frac{m_R}{\delta_R}
(1-\delta_R)^j
\; ,
\qquad j\ge0
\; .
\label{eq:uniform-tail-bound}
\end{align}
\end{lemma}

\begin{proof}
See Appendix~\ref{app:block-contraction}.
\end{proof}
We now have the two ingredients needed for achieving the frontier value $\Gamma(r)$. Lemma~\ref{lem:survival-compact} provides convergence on every finite set
of coordinates, while Lemma~\ref{lem:block-contraction} prevents ARL mass from escaping to arbitrarily late times.  Together, they allow passage to the limit while preserving both feasibility and the frontier
objective.

\begin{theorem}[Existence of a fixed-risk frontier optimizer]
\label{thm:frontier-optimizer}
For every finite $r\ge1$, there exists a survival process $\mathbf y_r^\star
\in F_{\le R}$ such that $\Gamma(r)
= \operatorname{ARL}(\mathbf y_r^\star)$.
Thus, the supremum defining $\Gamma(r)$ is achieved.
\end{theorem}

\begin{proof}
See Appendix~\ref{app:frontier-optimizer}.
\end{proof}
As discussed earlier, the existence guarantee in Theorem~\ref{thm:frontier-optimizer} is provided for the \emph{extended} feasible set, in which it is possible to have zero survival at a finite time. The next subsection repairs this feasible set extension without increasing its risk.

\subsection{Repairing a Finite Truncation Order}
\label{sec:tail-repair}

The feasible set extension temporarily allows $q_N(\mathbf y) = 0$ at a finite time $N$. This was useful for establishing the existence of an optimizer, but it is not permitted in the original Pollak problem. If $q_N=0$, then a change occurring at or after time $N$ is conditioned on an event of zero probability, and under our convention the resulting Pollak risk is infinite. Thus, before an extended frontier optimizer can be used to establish the achievability bound $V_{\sf P}(\gamma) \le r_\gamma$, we must convert it into a rule with strictly positive survival at every finite time. 

The goal of this subsection is to show that this conversion can be performed without sacrificing the performance obtained from the fixed-risk optimization. Specifically, if the extended rule terminates at a finite time, we will modify only its final surviving branch so that: 
\begin{enumerate}
\item its ARL remains unchanged;
\item all fixed-risk constraints preceding the terminal time remain unchanged;
\item survival becomes strictly positive at every finite time; and
\item every newly created late change-time fixed-risk constraint has conditional delay exactly $r$.
\end{enumerate}
Consequently, an extended rule satisfying the fixed-risk constraints at level $r$ is converted into an admissible Pollak rule whose risk remains no larger than $r$.

\vspace{.05 in} \noindent \textbf{The Finite-Terminal Defect:} Let $T$ be a stopping-time realization of an extended process in
$F_{\le R}$, and suppose that its survival probability vanishes at some finite time $N$, i.e., 
\begin{align}
q_{N-1}(T) > 0
\qquad \mbox{and} \qquad
q_N(T)=0
\; .
\label{eq:first-zero-survival}
\end{align}
Accordingly, define the terminal surviving event $E := \{T>N-1\}$. Because $N$ is the first zero-survival time, the event $E$ has positive probability under the no-change law. Furthermore, $q_N(T)=0$ means that the rule cannot survive beyond time $N$. Hence, on the event $E$, the extended rule is forced to stop exactly one step later. Thus, viewed from time $N-1$, the terminal branch of the extended rule has exactly one unit of residual lifetime. The idea behind the repair is to replace this deterministic one-step residual with a randomized residual that still has an average of 1 but has a positive tail extending indefinitely.

\vspace{.05 in} \noindent \textbf{The Bernoulli--Geometric Completion:} Fix $r>1$. Let $B \sim
\operatorname{Bernoulli}\!\left(\frac{1}{r}\right)$ and $
G \sim
\Geom\!\left(\frac{1}{r}\right)$, 
where for $G$ we have $\mathbb E[G] = r$. Take $B$ and $G$ independent of each other and of the observation process and of the randomization used by the original rule. On the terminal event $E$, we replace the forced stop at time $N$ as follows. With probability $1-1/r$, we stop one step earlier, at time $N-1$. With the remaining probability $1/r$, we attach a geometric tail of mean $r$. Thus given the rule $T$, design the new rule $\widetilde T$ as
\begin{align}
\widetilde T
:=
\begin{cases}
T, & E^c,
\\[1ex]
N-1, & E\cap\{B=0\}, \\[1ex]
N-1+G, & E\cap\{B=1\}.
\end{cases}
\label{eq:completion-rule}
\end{align}
The central identity behind this construction is $\mathbb E[BG]=1$. Therefore, conditional on reaching the terminal branch $E$,
\begin{align}
\mathbb E[\widetilde T\mid E]
&= N-1+\mathbb E[BG] = N
\; .
\label{eq:completion-conditional-mean}
\end{align}
The modified branch has exactly the same \emph{expected} terminal time as the original forced stop at $N$. This repair preserves the mean of the residual lifetime. The construction also remains within the behavioral randomized class introduced in \eqref{eq:behavioral-filtration}. Formally, the uniform auxiliary variable available at time $N-1$ can be split, through a measure-preserving relabeling, into independent uniform coordinates: one reproduces the original decision randomization, and the others generate $B$ and $G$. Thus, no extension of the admissible class of randomized stopping rules is required. The following lemma summarizes the properties of this completion.

\begin{lemma}[Geometric terminal completion]
\label{lem:tail-completion}
Let $r>1$, and let $T$ be an extended stopping rule satisfying $A_\nu(T)
\le r q_\nu(T)$ for all  $\nu\in\mathbb N_0$. Suppose there exists $N\ge2$ as the first integer such that $q_N(T)=0$.  Then there exists a randomized stopping time $\widetilde T$ such that $\Pinf(\widetilde T>n) > 0$ for all $n\in\mathbb N_0$ and 
\begin{align}
\Einf[\widetilde T]
&= \Einf[T]
\; ,
\label{eq:completion-mean}\\
D_\nu(\widetilde T)
&= D_\nu(T) \; ,
&& 0\le\nu\le N-2
\; ,
\label{eq:completion-early}\\
D_\nu(\widetilde T)
&=r \; ,
&& \nu\ge N-1
\; .
\label{eq:completion-late}
\end{align}
Consequently, $J_{\sf P}(\widetilde T)
\le r$.
\end{lemma}

\begin{proof}
See Appendix~\ref{app:tail-completion}.
\end{proof}
\subsection{Exact Calibration of the Average Run Length}
\label{sec:arl-calibration}

Lemma~\ref{lem:tail-completion} resolves the admissibility issue. The resulting ARL,
however, may exceed the prescribed level $\gamma$. The next step is therefore to reduce the ARL to exactly $\gamma$ while preserving both
admissibility and the risk bound.
\begin{lemma}[Exact ARL calibration]
\label{lem:arl-calibration}
Let $\bar T$ be an admissible randomized stopping time satisfying $1<\gamma\le a:=\Einf[\bar T]<\infty$ and $
J_{\sf P}(\bar T) \le r$ where $r\ge1$. Then there exists an admissible randomized stopping time
$T_\gamma$ such that
\begin{align}
\label{eq:calibration-exact-arl}
\Einf[T_\gamma]
= \gamma\; , \qquad \mbox{and} \qquad 
J_{\sf P}(T_\gamma) \le r \; .
\end{align}
Furthermore, if $a>\gamma$, then $D_\nu(T_\gamma)
= D_\nu(\bar T)$ for all $\nu\ge 1$.
\end{lemma}

\begin{proof}
See Appendix~\ref{app:arl-calibration}.
\end{proof}
\subsection{Exact Minimax Characterization}
\label{sec:general-main}

We now complete the achievability direction and establish the exact
minimax identity
\begin{align}
V_{\sf P}(\gamma)=r_\gamma
\; .
\end{align}
The preceding subsections have established frontier achievement, repaired finite-terminal optimizers, and provided exact ARL calibration. Two final facts are needed at the crossing risk itself. First, the frontier must
actually achieve the prescribed level $\gamma$ at $r_\gamma$; second, $r_\gamma$ must exceed one so that the terminal-completion construction
is available whenever needed.
\begin{lemma}[The frontier reaches the crossing]
\label{lem:crossing-achieved}
For every $\gamma>1$, we have $\Gamma(r_\gamma) \ge \gamma$. Consequently,
\begin{align}
r_\gamma
= \min \left\{ r\in[1,\gamma]:
\Gamma(r)\ge\gamma \right\}
\; .
\label{eq:r-gamma-minimum}
\end{align}
\end{lemma}

\begin{proof}
See Appendix~\ref{app:crossing-achieved}.
\end{proof}
Lemma~\ref{lem:crossing-achieved} shows that the frontier does not merely
approach $\gamma$ from above: the crossing risk itself supports the required
ARL. It remains only to verify that $r_\gamma>1$, which is needed if a finite-terminal optimizer must be completed.

\begin{lemma}[The crossing risk exceeds one]
\label{lem:r-gamma-positive}
For every $\gamma>1$ we have $r_\gamma > 1$.
\end{lemma}

\begin{proof}
See Appendix~\ref{app:r-gamma-positive}.
\end{proof}
The two properties needed at the crossing are now in place:
$\Gamma(r_\gamma)\ge\gamma$ and $r_\gamma>1$. A frontier optimizer at
$r_\gamma$ therefore has ARL at least $\gamma$; if necessary, Lemma
~\ref{lem:tail-completion} converts it into a genuine Pollak-admissible
rule, and Lemma~\ref{lem:arl-calibration} calibrates its ARL to exactly
$\gamma$. Together with the lower bound in Lemma~\ref{lem:frontier-lower-bound},
this yields the exact minimax characterization.

\begin{theorem}[Exact Pollak minimax characterization]
\label{thm:main}
For every $\gamma>1$,
\begin{align}
V_{\sf P}(\gamma)
=
r_\gamma
=
\min
\left\{
r\in[1,\gamma]:
\Gamma(r)\ge\gamma
\right\}
\; .
\label{eq:main-value}
\end{align}
Furthermore, there exists an admissible randomized stopping time
$T_\gamma^\star$ satisfying
\begin{align}
\label{eq:positive-survival}
\Einf[T_\gamma^\star]
= \gamma \; , \quad J_{\sf P}(T_\gamma^\star)
= V_{\sf P}(\gamma)
\; , \quad \mbox{and}  \quad \Pinf(T_\gamma^\star>n)
> 0 \; , 
\forall\,n\in\mathbb N_0 \; .
\end{align}
\end{theorem}

\begin{proof}
See Appendix~\ref{app:main}.
\end{proof}

Theorem~\ref{thm:main} closes the loop begun with the fixed-risk reformulation. The scalar frontier has an exact operational interpretation: $r_\gamma$ is not merely a lower bound obtained by relaxing the original problem; it is the exact Pollak minimax value. The corresponding frontier optimizer can be implemented as a randomized stopping rule, repaired if its extended survival terminates at a finite time, and calibrated if necessary so that the false-alarm constraint holds with equality. The resulting $T_\gamma^\star$ is an exact optimizer over the complete randomized, history-dependent class.

\section{Truncated Approximation of the PVF}
\label{sec:truncated}

The preceding sections establish the existence of an admissible stopping rule that achieves the exact Pollak minimax value. While the characterization is exact, designing an optimal rule becomes equivalent to solving an infinite-dimensional variational problem: the frontier $\Gamma(r)$ is defined over an infinite survival process and involves one constraint for every possible change time $\nu\in\mathbb{N}_0$. Consequently, the existence result does not by itself provide a feasible construction of the optimal rule. 

In this section, we bridge this gap. Our approach relies on \emph{truncating} the survival process at a finite order $N$. This gives a truncated frontier $\Gamma_N(r)$, whose distance from the exact frontier can be bounded as a function of $N$ using the uniform tail control established in Lemma~\ref{lem:block-contraction}. To proceed, we define a truncated approximation to the PVF, denoted by $r_{\gamma,N}$, and show that
\begin{align}
r_{\gamma,N}
\downarrow
V_{\sf P}(\gamma)
\qquad
\text{as }N\to\infty
\; ,
\label{eq:finite-convergence-preview}
\end{align}
i.e., the truncated PVF crossings decrease monotonically to the exact PVF crossing. We show that the minimax value can be approached arbitrarily closely by increasing the truncation order $N$. This approximation applies to stopping rules as well. Such truncation yields a finite-order optimizer that can be converted into an admissible stopping rule $T_{\gamma,N}$. For finite observation alphabets, this optimizer is obtained from a finite-dimensional primal linear program; for general observation spaces, the truncated primal remains function-valued, while Section~\ref{sec:finite-dual} provides a finite-dimensional dual characterization. The resulting admissible rule satisfies the following limiting behavior, in which all constants are fixed and only $N$ grows to infinity:
\begin{align}
0 \le J_{\sf P}(T_{\gamma,N})-V_{\sf P}(\gamma)
\le r_{\gamma,N}-V_{\sf P}(\gamma) \overset{N\rightarrow\infty}{\xrightarrow{\hspace{.8cm}}} 0
\; .
\label{eq:finite-rule-gap-preview}
\end{align}
Hence the truncated problem provides a sequence of admissible rules whose performance can be arbitrarily close to that of the exact Pollak optimizer. We note that throughout this section, the false-alarm level $\gamma>1$ and the observation laws $P_\infty$ and $P_0$ are fixed. The approximation
parameter is the truncation order $N$. Accordingly, all convergence statements in this section are with respect to $N\to\infty$, rather
than the classical asymptotic regime $\gamma\to\infty$.

The same inner--outer decomposition introduced for the exact frontier now becomes computational. For each fixed $N$ and trial $r$, the inner problem computes $\Gamma_N(r)$. The outer problem then finds the smallest $r$ for which $\Gamma_N(r)\ge\gamma$, namely $r_{\gamma,N}$. Increasing $N$ refines this entire crossing problem. It does not merely improve the solution of a fixed-$r$ design, but produces a decreasing sequence of outer-search solutions $r_{\gamma,N}\downarrow V_{\sf P}(\gamma)$.

\subsection{Truncated Frontier and Error Bound}

Fix a truncation order $N\ge1$. Let $\mathcal Y_N$ denote the projection of the survival-process space onto the coordinates $\{0,\ldots,N-1\}$. Equivalently, an element of $\mathcal Y_N$ may be viewed as an extended survival process obtained by setting $y_n=0$ for all $n\ge N$. Define the truncated ARL as
\begin{align}
\operatorname{ARL}_N(\mathbf y)
:=
\sum_{n=0}^{N-1}q_n(\mathbf y)
\; ,
\label{eq:finite-arl}
\end{align}
and, for $0\le\nu<N$, define the truncated delay numerator
\begin{align}
A_\nu^{[N]}(\mathbf y)
:=
\sum_{n=\nu}^{N-1}
\Einf
\left[
L_{\nu,n}y_n
\right]
\; .
\label{eq:finite-delay-numerator}
\end{align}
For a fixed risk level $r\ge1$, the truncated frontier is
\begin{align}
\Gamma_N(r)
:=
\max_{\mathbf y\in\mathcal Y_N}
\operatorname{ARL}_N(\mathbf y) \quad \mbox{subject to} \quad 
A_\nu^{[N]}(\mathbf y)
\le
r q_\nu(\mathbf y) \; ,
\qquad
0\le\nu<N
\; .
\label{eq:finite-risk-constraints}
\end{align}
Hence, $\Gamma_N(r)$ is the largest truncated ARL that can be achieved while satisfying every change-time constraint arising for truncation order $N$. For a finite observation alphabet, this is a finite-dimensional linear
program. For a general observation space, truncation makes the number
of time coordinates finite, although each survival function may remain infinite-dimensional. Section~\ref{sec:finite-dual} shows that, nevertheless, the dual of the truncated problem can be expressed as a finite-dimensional convex optimization over $N$ nonnegative dual
weights.

The connection between $\Gamma_N(r)$ and the exact frontier is particularly simple. A truncated feasible process can be extended by zero beyond $N-1$, and therefore it is also feasible for the extended infinite-dimensional problem. Conversely, truncating any feasible infinite process preserves every constraint before $N$, because truncation leaves $q_\nu$ unchanged and can only decrease $A_\nu$. The only quantity lost is the contribution of the survival tail to the ARL. 
Recall from Lemma~\ref{lem:block-contraction} that, for the fixed risk level $r$, there exist constants $m_r$ and $\delta_r>0$ such that the survival tail is uniformly bounded. Define
\begin{align}
\varepsilon_N(r)
:=
\frac{m_r}{\delta_r}
(1-\delta_r)^{\lfloor N/m_r\rfloor}
\; .
\label{eq:finite-error}
\end{align}
Then $\varepsilon_N(r)\to0$ as $N\to\infty$. The error bound is formalized in the next theorem.

\begin{theorem}[truncated frontier error bound]
\label{thm:finite-frontier-bounds}
For every finite $r\ge1$ and  $N\ge1$,
\begin{align}
\Gamma_N(r)
\le
\Gamma(r)
\le
\Gamma_N(r)+\varepsilon_N(r)
\; .
\label{eq:frontier-error-bound}
\end{align}
\end{theorem}

\begin{proof}
See Section~\ref{app:finite-frontier-bound}.
\end{proof}
The two inequalities in \eqref{eq:frontier-error-bound} have direct interpretations. The first follows because zero extension embeds every truncated feasible process into the infinite-length feasible class. For the second, truncating an infinite-length feasible process loses at most $\varepsilon_N(r)$ of its ARL. Hence, the entire difference between the two optimization problems is governed by the survival probability that persists beyond the truncation order. 
Theorem~\ref{thm:finite-frontier-bounds} also gives useful truncated tests for a candidate risk level $r$. If $\Gamma_N(r)\ge\gamma$, then the truncated problem already supports the required ARL, and the terminal-completion result from Lemma~\ref{lem:tail-completion} yields an admissible rule with Pollak risk no larger than $r$. Hence $V_{\sf P}(\gamma)\le r$. Conversely, if
\begin{align}
\Gamma_N(r)+\varepsilon_N(r)
<
\gamma
\; ,
\label{eq:finite-infeasible-test}
\end{align}
then $\Gamma(r)<\gamma$, and therefore $V_{\sf P}(\gamma)>r$. Increasing $N$ makes the gap between these two conclusions vanish.

\subsection{Truncated Crossing and Convergence}

We next use the truncated frontier in the same way that the exact frontier was used in Section~\ref{sec:fixed-risk}. For $N\ge\lceil\gamma\rceil$, define
\begin{align}
r_{\gamma,N}
:=
\min
\left\{
r\in[1,\gamma]:
\Gamma_N(r)\ge\gamma
\right\}
\; .
\label{eq:finite-crossing}
\end{align}

\begin{lemma}[Existence of the truncated crossing]
\label{lem:finite-crossing}
For every $\gamma>1$ and $N\ge\lceil\gamma\rceil$, the crossing value in \eqref{eq:finite-crossing} is well defined and satisfies $1
< r_{\gamma,N} \le \gamma $.
\end{lemma}
\begin{proof}
See Appendix~\ref{app:finite-crossing}.
\end{proof}
The upper endpoint follows from an observation-independent randomized deadline with mean $\gamma$, while $r=1$ can support only ARL one under the equivalence assumption. The finite crossing is directly related to the exact minimax value. Since $\Gamma_N(r)\le\Gamma(r)$, reaching ARL $\gamma$ is at least as difficult in the truncated problem as in the problem with an infinite survival process. Therefore,
\begin{align}
r_{\gamma,N}
\ge
r_\gamma
=
V_{\sf P}(\gamma)
\; .
\label{eq:finite-upper-bound}
\end{align}
Hence every truncated crossing provides an upper bound on the exact Pollak value. Furthermore, increasing the truncation order enlarges the finite feasible problem. In particular, any process of length $N$ can be embedded in the problem of length $(N+1)$ by setting its additional survival coordinate to zero. Thus $\Gamma_N(r)\le\Gamma_{N+1}(r)$, and consequently
\begin{align}
r_{\gamma,N+1}
\le
r_{\gamma,N}
\; .
\label{eq:crossing-monotonicity}
\end{align}
Hence,  $\{r_{\gamma,N}: N\geq \lceil\gamma\rceil\}$ is a decreasing sequence of upper bounds on $V_{\sf P}(\gamma)$.

The remaining question is whether these bounds actually converge to the exact value. Addressing this requires slightly more than the frontier error bound alone. Truncating an exact optimizer decreases its ARL by a small amount, so it may fail to reach $\gamma$ at the exact crossing $r_\gamma$. The proof resolves this by restoring the lost ARL through a mixture of the truncated optimizer and a one-step shifted copy. The increase in the corresponding risk bound is no larger than the ARL lost through truncation, which vanishes as $N\to\infty$. 
\begin{theorem}[Convergence of truncated PVF approximations]
\label{thm:finite-convergence}
For every $\gamma>1$,
\begin{align}
r_{\gamma,N}
\downarrow
V_{\sf P}(\gamma)
\qquad\text{as }N\to\infty
\; .
\end{align}
More specifically, for every integer $N$ such that $N+1\ge\lceil\gamma\rceil$ and $\varepsilon_N\!\left(V_{\sf P}(\gamma)\right)\le 1$, we have
\begin{align}
0
\le
r_{\gamma,N+1}-V_{\sf P}(\gamma)
\le
\varepsilon_N\!\left(V_{\sf P}(\gamma)\right)
\; .
\label{eq:finite-crossing-explicit-condition}
\end{align}
Furthermore, for every $N\ge\lceil\gamma\rceil$, a truncated optimizer can
be converted into an admissible randomized stopping rule $T_{\gamma,N}$
satisfying
\begin{align}
\Einf[T_{\gamma,N}]
&=
\gamma
\; ,
&
J_{\sf P}(T_{\gamma,N})
&\le
r_{\gamma,N}
\; .
\end{align}
Consequently,
\begin{align}
0
\le
J_{\sf P}(T_{\gamma,N})-V_{\sf P}(\gamma)
\le
r_{\gamma,N}-V_{\sf P}(\gamma)
\longrightarrow
0
\; .
\end{align}
\end{theorem}

\begin{proof}
See Appendix~\ref{app:finite-convergence}.
\end{proof}
This is the main approximation result. At every finite truncation order,
$r_{\gamma,N}$ is an upper bound on the exact Pollak value, and these
bounds improve monotonically to $V_{\sf P}(\gamma)$. Importantly, the
approximation applies not only to the optimal values: each truncated
optimizer can be converted, using Lemmas~\ref{lem:tail-completion} and
\ref{lem:arl-calibration}, into an admissible rule whose Pollak risk
approaches the unrestricted optimum. Thus, for every $\epsilon>0$, there exists a sufficiently large $N$ such that
\begin{align}
0
\le
r_{\gamma,N}-V_{\sf P}(\gamma)
<
\epsilon
\; .
\end{align}
We emphasize that this conclusion does not depend on any particular parametric family of stopping rules. The truncated problems retain the complete survival-process formulation, so the convergence is with respect to the same unrestricted class of randomized, history-dependent procedures used in the exact minimax problem. 

\begin{example}[Gaussian mean shift]
Consider
\begin{align}
P_\infty=\mathcal N(0,1)
\qquad \mbox{and} \qquad
P_0=\mathcal N(\theta,1)
\; .
\end{align}
\end{example}
The truncation guarantee in Theorem~\ref{thm:finite-convergence} can be made explicit in this model. For fixed $\gamma>1$ and $\theta\ne0$, define
\begin{align}
m_\gamma
&:=
\left\lceil2\gamma-1\right\rceil,
&
\delta_{\gamma,\theta}
&:=
\Phi\!\left(-|\theta|\sqrt{m_\gamma}\right),
\label{eq:gaussian-convergence-constants}
\end{align}
where $\Phi$ is the standard Gaussian distribution function. Appendix~\ref{app:gaussian-convergence} gives a sharper geometric tail bound and shows that, once that tail bound is at most one,
\begin{align}
0
\le
r_{\gamma,N+1}-V_{\sf P}(\gamma)
\le
\frac{2m_\gamma}{\delta_{\gamma,\theta}}
\exp\!\left\{
-\frac{\delta_{\gamma,\theta}}{m_\gamma}N
\right\}.
\label{eq:gaussian-exponential-convergence}
\end{align}
For $\gamma=5$ and $\theta=0.5$, $m_\gamma=9$ and $\delta_{\gamma,\theta}=\Phi(-1.5)\approx0.0668072$. The sharper geometric tail bound used in the proof is already below one for $N\ge639$, so the finite-$N$ inequality is applicable from that point onward. The exponential bound in \eqref{eq:gaussian-exponential-convergence} is below $10^{-2}$ for $N\ge1375$. Consequently,
\begin{align}
0\le r_{5,N+1}-V_{\sf P}(5)<10^{-2},
\qquad N\ge1375.
\end{align}
Equivalently, a truncation order $K\ge1376$ guarantees $0\le r_{5,K}-V_{\sf P}(5)<10^{-2}$ under this bound. The $N+1$ indexing is intrinsic to Theorem~\ref{thm:finite-convergence}, which controls the next truncated crossing using the survival tail beyond $N$.

\section{Structure of Truncated Optimal Rules}
\label{sec:finite-dual}

The truncated formulation in Section~\ref{sec:truncated} yields a sequence of optimization problems whose values converge to the exact Pollak minimax value. It does not, however, immediately reveal the structure of the corresponding optimal stopping rules. Even for a fixed truncation order $N$, the primal variables consist of survival functions defined over all possible observation histories.  The purpose of this section is to obtain an explicit sequential
representation of these truncated optimizers. Fix a truncation order
$N$ and a candidate risk level $r$. Associate with each change-time
constraint
\begin{align}
A_\nu^{[N]}(\mathbf y)
\le
r q_\nu(\mathbf y)
\; ,
\qquad
0\le \nu<N
\; ,
\end{align}
a nonnegative weight $\omega_\nu\ge0$. Recall that $\Lambda_n$ is the one-observation likelihood ratio. The weighted contributions of
all possible change times up to time $n$ can then be summarized by
\begin{align}
\label{eq:Mn}
M_n := \sum_{\nu=0}^{n}
\omega_\nu L_{\nu,n}
\; , \qquad L_{\nu,n} :=
\prod_{j=\nu+1}^{n}\Lambda_j
\; , 
\end{align}
with $L_{n,n}=1$. Consequently, $M_0=\omega_0$ and
\begin{align}
M_n
=
\omega_n+\Lambda_n M_{n-1}
\; ,
\qquad
1\le n<N
\; .
\label{eq:dual-statistic-preview}
\end{align}
Hence, although the truncated optimization contains one constraint for each possible change time, their weighted cumulative contribution to the off-boundary decision can be represented by the scalar state $M_n$. A backward dynamic program shows that the sign of a scalar continuation advantage determines the optimal action whenever $M_n$ is strictly below or above a \textbf{time-dependent boundary}. On a positive-probability boundary tie, however, the saddle-compatible continuation probability may depend on the complete observation history

\begin{remark}
    We highlight that this representation is derived from the unrestricted truncated optimization. No SR structure is assumed in advance, and the SR recursion emerges as a special case in which the dual weights and stopping boundaries become constant.
\end{remark}

\subsection{Dual Representation and the Weighted Likelihood Statistic}

Fix $N\ge2$ and $r>1$. Recall that the truncated problem maximizes $\operatorname{ARL}_N(\mathbf y)$ subject to
\begin{align}
A_\nu^{[N]}(\mathbf y)
\le
r q_\nu(\mathbf y) \; ,
\qquad
0\le\nu<N
\; .
\label{eq:dual-primal-constraints}
\end{align}
Assign a nonnegative weight $\omega_\nu$ with each of these constraints and let
$\boldsymbol\omega=(\omega_0,\ldots,\omega_{N-1})$. The corresponding Lagrangian is
\begin{align}
\mathcal L_N(\mathbf y,\boldsymbol\omega)
:=
\operatorname{ARL}_N(\mathbf y)
+
\sum_{\nu=0}^{N-1}
\omega_\nu
\left[
r q_\nu(\mathbf y)
-
A_\nu^{[N]}(\mathbf y)
\right]
\; .
\label{eq:dual-lagrangian}
\end{align}
Recall that
$\operatorname{ARL}_N(\mathbf y)=\sum_{n=0}^{N-1}q_n(\mathbf y)$, based on which from \eqref{eq:dual-lagrangian} we get
\begin{align}
\mathcal L_N(\mathbf y,\boldsymbol\omega)
&=
\sum_{n=0}^{N-1}\Einf[y_n]
+
r\sum_{\nu=0}^{N-1}\omega_\nu\Einf[y_\nu]
-
\sum_{\nu=0}^{N-1}
\omega_\nu
\sum_{n=\nu}^{N-1}
\Einf
\left[
L_{\nu,n}y_n
\right]
\; .
\label{eq:dual-lagrangian-expanded-1}
\end{align}
The first two terms already separate according to the observation time.
For the last term, we interchange the order of summation. At a fixed
time $n$, every possible change time $\nu\le n$ contributes the
quantity $\omega_\nu L_{\nu,n}$. Hence
\begin{align}
\sum_{\nu=0}^{N-1}
\omega_\nu
\sum_{n=\nu}^{N-1}
\Einf
\left[
L_{\nu,n}y_n
\right]
&=
\sum_{n=0}^{N-1}
\Einf
\left[
\left(
\sum_{\nu=0}^{n}
\omega_\nu L_{\nu,n}
\right)y_n
\right]
\; .
\label{eq:dual-sum-rearrangement}
\end{align}
This rearrangement identifies the quantity that summarizes, at time
$n$, the contributions of all possible change times up to $n$.
By recalling the definition of $M_n$ in \eqref{eq:Mn}, the Lagrangian becomes
\begin{align}
\mathcal L_N(\mathbf y,\boldsymbol\omega)
&=
1+(r-1)\omega_0
+
\sum_{n=1}^{N-1}
\Einf
\left[
\bigl(
1+r\omega_n-M_n
\bigr)y_n
\right]
\; ,
\label{eq:dual-lagrangian-expanded}
\end{align}
where the equality uses $y_0=1$, $L_{0,0}=1$, and hence
$M_0=\omega_0$. 
The importance of $M_n$ is that it can be updated recursively rather
than recomputed from the entire collection of past likelihood ratios.
For every $\nu<n$,
$L_{\nu,n}=L_{\nu,n-1}\Lambda_n$, while $L_{n,n}=1$. Therefore,
\begin{align}
M_n
&=
\omega_n
+
\sum_{\nu=0}^{n-1}
\omega_\nu L_{\nu,n-1}\Lambda_n
=
\omega_n+\Lambda_nM_{n-1}
\; .
\label{eq:dual-statistic}
\end{align}
Thus, the statistic $M_n$  arises directly by regrouping the weighted change-time
constraints according to the current observation time. Once the dual weights are fixed, all contributions from past change times are summarized by a single, recursively updated scalar state. 

\subsection{Duality and Threshold Structure}

For fixed $\boldsymbol\omega$, maximizing
\eqref{eq:dual-lagrangian-expanded} over the survival process is a truncated optimal stopping problem. Set $W_N(m)=0$ and, working backward for $n=N-1,\ldots,1$, define
\begin{align}
W_n(m)
=
\left[
1+r\omega_n-m
+
\Einf
W_{n+1}
\bigl(
\omega_{n+1}+\Lambda m
\bigr)
\right]_+
\; ,
\label{eq:dual-bellman}
\end{align}
where we set $\omega_N=0$, and $\Lambda$ denotes a generic one-observation likelihood ratio under $P_\infty$. For convenience, define the continuation advantage
\begin{align}
H_n(m)
:=
1+r\omega_n-m
+
\Einf
W_{n+1}
\bigl(
\omega_{n+1}+\Lambda m
\bigr)
\; .
\label{eq:dual-continuation}
\end{align}
Thus, continuation is preferable when $H_n(M_n)>0$, whereas stopping is preferable when $H_n(M_n)<0$.
The corresponding dual objective is
\begin{align}
\mathfrak D_N(\boldsymbol\omega)
:=
1+(r-1)\omega_0
+
\Einf
W_1
\bigl(
\omega_1+\Lambda_1\omega_0
\bigr)
\; .
\label{eq:dual-objective}
\end{align}
The following theorem collects the principal consequences of the dual formulation. 

\begin{theorem}[Truncated dual characterization]
\label{thm:finite-dual-structure}

For every $N\ge2$ and $r>1$,
\begin{align}
\Gamma_N(r)
&=
\min_{\boldsymbol\omega\ge0}
\mathfrak D_N(\boldsymbol\omega)
\; .
\label{eq:finite-strong-duality}
\end{align}
The primal maximum and the dual minimum are both achieved. For fixed $\boldsymbol\omega\ge0$ and $1\le n<N$, let
$b_{n,N}(\boldsymbol\omega)$ denote the unique zero of the
continuation-advantage function $H_n$ generated by
$\boldsymbol\omega$. Let $\boldsymbol\omega^\star$ be a minimizing
dual vector, define $M_n^\star$ according to
\eqref{eq:dual-statistic}, and set
\begin{align*}
b_{n,N}
:=
b_{n,N}(\boldsymbol\omega^\star)
\; .
\end{align*}
Then there exists a single primal optimizer $\mathbf y^\star$
forming a saddle pair with $\boldsymbol\omega^\star$ such that, for
every $1\le n<N$, $P_\infty^n$-almost surely on histories satisfying
\begin{align*}
y_{n-1}^\star(x_{1:n-1})
>
0
\; ,
\end{align*}
its stopping rule satisfies
\begin{align}
M_n^\star
< b_{n,N} \quad\Longrightarrow\quad
\text{continue} \qquad \mbox{vs.} \qquad 
M_n^\star
> b_{n,N} \quad\Longrightarrow\quad
\text{stop}
\; .
\label{eq:dual-threshold-rule}
\end{align}
In the tie set
$M_n^\star=b_{n,N}$, the rule uses the measurable, possibly
history-dependent continuation kernel inherited from
$\mathbf y^\star$, which preserves primal feasibility and
optimality. An arbitrary tie probability is not asserted to preserve these properties. 
Furthermore,
\begin{align}
\omega_\nu^\star
\left[
r q_\nu(\mathbf y^\star)
-
A_\nu^{[N]}(\mathbf y^\star)
\right]
&=
0
\; ,
\qquad
0\le\nu<N
\; .
\label{eq:dual-complementarity}
\end{align}

\end{theorem}
\begin{proof}
See Appendix~\ref{app:finite-dual-threshold}.
\end{proof}
The scalar threshold representation is exact off boundary ties. If
$M_n^\star=b_{n,N}$ has positive probability, the optimizer retains the
saddle-compatible, possibly history-dependent continuation kernel specified
in Theorem~\ref{thm:finite-dual-structure}; when
$P_\infty(M_n^\star=b_{n,N})=0$, the rule is therefore a scalar threshold rule almost surely. 
Two further consequences will be useful. First, for $\boldsymbol\omega\ge0$,
\begin{align}
\Gamma_N(r) \le \mathfrak D_N(\boldsymbol\omega)
\; ,
\label{eq:dual-any-vector-upper-bound}
\end{align}
so any dual vector provides an upper bound on the truncated frontier.
Second, complementary slackness implies that
\begin{align}
\omega_\nu^\star>0
\quad\Longrightarrow\quad
A_\nu^{[N]}(\mathbf y^\star)
=
r q_\nu(\mathbf y^\star)
\; .
\end{align}
Hence, whenever $q_\nu(\mathbf y^\star)>0$, a positive dual weight identifies a change time whose conditional delay achieves the prescribed risk level $r$. 
We note that because the dynamic program is solved backward from the terminal order $N$, the boundaries $b_{n,N}$ are generally time-dependent. Thus, for each fixed $N$, there exists a saddle-compatible truncated optimizer whose off-boundary decisions have a time-inhomogeneous
weighted likelihood-ratio threshold form.

\vspace{.05 in} \noindent \textbf{Relation to the SR Recursion.}
The recursion in \eqref{eq:dual-statistic} is closely related to the classical SR statistic. To see the connection, suppose that the dual injections after time zero are constant,
\begin{align}
\omega_n
=
\omega>0 \; ,
\qquad
1\le n<N \; ,
\label{eq:constant-dual-weights}
\end{align}
and that the corresponding stopping boundaries are also constant, $b_{n,N} = b$. Define
\begin{align}
R_n
:=
\frac{M_n}{\omega}-1
\; \quad \Rightarrow \quad R_n
&=
\frac{\omega+\Lambda_nM_{n-1}}{\omega}-1
=
(1+R_{n-1})\Lambda_n
\; ,
\label{eq:dual-sr-rescaling}
\end{align}
which is exactly the SR recursion. 
In particular, when $\omega\le\omega_0<b$, the rescaled initial state and boundary are
\begin{align}
R_0=\frac{\omega_0}{\omega}-1\ge0,
\qquad
A=\frac{b}{\omega}-1>R_0,
\end{align}
so the classical SR initialization $R_0=0$ corresponds to $\omega_0=\omega$.

The significance of
this connection is the direction in which it is obtained. We did not
restrict the truncated optimization to the SR family. Instead, the
weighted recursion in \eqref{eq:dual-statistic}
emerged from dualizing the unrestricted survival-process problem. When the dual injections and stopping boundaries are constant, its
off-boundary recursion reduces to the classical SR recursion, with
deterministic initialization yielding the ordinary SR procedure and its
head-started SR-$r$ variants.

More generally, the randomized stopping-rule class considered in this
paper also contains the SRP procedure. Indeed, the auxiliary
randomization available at initialization permits the initial SR
statistic to be drawn from its quasi-stationary distribution below the
detection threshold, after which the same constant-boundary SR recursion
is applied. Thus, SRP is contained in the unrestricted randomized
framework, although its quasi-stationary initialization is not itself
derived from the constant-dual-parameter specialization above.

In general, the optimal truncated dual injections and boundaries need not be constant. The truncated formulation therefore permits
time-varying dual injections and time-varying boundaries, yielding a broader class of sequential rules than the classical constant-boundary
SR family.

\section{Likelihood-Ratio Floors: Explicit Lower Bounds and Exact Solutions}
\label{sec:likelihood-floor}

The results developed so far apply to general equivalent pre- and post-change laws. In this section, we impose additional
structure on the observation model and ask whether it can lead directly to an explicit lower bound on the Pollak risk. The relevant structure is
a pointwise lower bound on the one-observation likelihood ratio. If the likelihood ratio can never fall below a constant $c$, then the
likelihood ratio accumulated over any $k$ post-change observations can never fall below $c^k$. This allows every post-change survival probability to be bounded from below using only the corresponding
pre-change survival probability.

The important consequence is that the observation histories can be eliminated from the lower-bound argument. The stochastic detection problem reduces to a deterministic inequality involving the scalar
no-change survival sequence
$\{q_n(T):n\in\mathbb N_0\}$. This yields an explicit lower bound valid for the entire class of admissible stopping rules. 
Even more importantly, tracing the conditions for equality reveals when this bound can be achieved. Equality determines both how much survival
probability the detector must retain and on which observations that survival probability can be retained. When the likelihood-ratio floor is achieved on a sufficiently large set, these conditions lead to
an explicit stationary randomized optimal rule.

Unlike the general existence theory developed earlier, the results in
this section do not require mutual absolute continuity unless stated
otherwise. It is sufficient that $P_0\ll P_\infty$ so that the likelihood ratio is defined. We begin with the additional distributional assumption.

\begin{assumption}[Likelihood-ratio floor]
\label{ass:likelihood-floor}
There exists $c\in[0,1)$ such that
\begin{align}
\Lambda(X)
=
\frac{\mathrm dP_0}{\mathrm dP_\infty}(X)
\ge
c
\qquad
P_\infty\text{-almost surely}
\; .
\label{eq:likelihood-floor}
\end{align}
\end{assumption}
A strictly positive likelihood-ratio floor arises naturally in
finite-alphabet models with common support, contamination models, many bounded-support families, and several one-sided parameter changes such as
Poisson rate increases, exponential rate decreases, and Gaussian variance increases. A particularly transparent example is the mixture model when $Q\ll P_\infty$:
\begin{align}
P_0 &= cP_\infty+(1-c)Q
\; , \qquad 0<c<1
\; .
\end{align}
\begin{theorem}[Likelihood-floor lower bound]
\label{thm:floor-bound}
Suppose Assumption~\ref{ass:likelihood-floor} holds. Then every
admissible stopping rule satisfying $\Einf[T]\ge\gamma>1$ obeys
\begin{align}
J_{\sf P}(T)
\ge
R_c(\gamma)
:=
\frac{\gamma}
{c+(1-c)\gamma}
\; .
\label{eq:floor-lower-bound}
\end{align}
Consequently,
\begin{align}
V_{\sf P}(\gamma)
\ge
\frac{\gamma}
{c+(1-c)\gamma}
\; .
\label{eq:floor-value-lower-bound}
\end{align}
\end{theorem}

\begin{proof}
See Appendix~\ref{app:floor-bound}.
\end{proof}
The endpoint $c=0$ gives the trivial bound $J_{\sf P}(T)\ge1$, since the
residual delay after any surviving change is a positive integer. Theorem~\ref{thm:floor-bound} differs from the frontier lower bound
developed earlier. The frontier characterizes the exact value but requires solving an optimization problem. In contrast, \eqref{eq:floor-lower-bound} is explicit and depends on the observation
model only through the scalar likelihood-ratio floor $c$. It is also
an all-rules bound: no statistic, threshold structure, Markov representation, or specific family of detection procedures has
been imposed. 
We next identify conditions under which the lower bound in
Theorem~\ref{thm:floor-bound} is achieved. Define 
\begin{align}
    \rho:=1-\frac{1}{\gamma}\ , \qquad \mbox{and} \qquad B := \left\{ x\in\mathsf X: \Lambda(x)=c \right\}\ , \;\; \mbox{where} \;\; \beta:=P_\infty(B)\ .
\end{align}
Thus, $P_0(B)=c\beta$. We show that achieving the lower bound requires geometric pre-change survival with one-step continuation probability $\rho$, with continuation occurring only on the likelihood-ratio floor. Hence, if $\beta\ge\rho$, the floor set contains sufficient probability mass to implement this survival behavior. This is formalized in the following theorem.

\begin{theorem}[Sharpness under an achieved likelihood-ratio floor]
\label{thm:floor-sharpness}
Suppose Assumption~\ref{ass:likelihood-floor} holds, the floor set
$B$ satisfies $\beta=P_\infty(B)>0$, and $\rho = 1-\frac1\gamma \le \beta$. Consider the stationary randomized stopping rule $T_c^\star$ that,
at every time $n$, continues according to
\begin{align}
\text{continue at time }n
\quad\Longleftrightarrow\quad
X_n\in B
\ \text{ and }\
U_n\le\frac{\rho}{\beta}
\; .
\label{eq:floor-optimal-rule}
\end{align}
\begin{align}
\Einf[T_c^\star]
=
\gamma \; , \quad \mbox{and} \quad V_{\sf P}(\gamma) & = D_\nu(T_c^\star)
= \frac{1}{1-c\rho}
= \frac{\gamma}
{c+(1-c)\gamma} \; ,
\qquad \nu\ge 0
\; .
\label{eq:floor-rule-equalizer}
\end{align}
\end{theorem}

\begin{proof}
See Appendix~\ref{app:floor-sharpness}.
\end{proof}
The theorem therefore identifies a sufficient condition under which the explicit all-rules lower bound in Theorem~\ref{thm:floor-bound} is achieved by a stationary randomized equalizer rule.  We next specialize this
result to the Bernoulli model, where both the likelihood-ratio floor and the probability mass of its achieving set are explicit. In this case, the sharpness condition reduces to a simple range of false-alarm levels, and the corresponding optimal rule takes an elementary randomized form.

\begin{theorem}[Exact Bernoulli minimax solution]
\label{thm:bernoulli-exact}
Consider the Bernoulli model
\begin{align}
 P_\infty(X_n=1)=p \ , \qquad \mbox{and} \qquad    P_0(X_n=1)=q\ , \qquad 0<p<q<1\ .
\end{align}
For every $1<\gamma\le1/p$, define
\begin{align}
\eta_\gamma=\frac{1-\gamma p}{\gamma(1-p)}\ .    
\end{align}
The stationary randomized rule that stops whenever $X_n=1$ and, when $X_n=0$, stops with probability $\eta_\gamma$ and continues with probability $1-\eta_\gamma$, is admissible, has $\Einf[T]=\gamma$, and is an exact Pollak minimizer with
\begin{align}
V_{\sf P}(\gamma)
=
\frac{\gamma(1-p)}{1-q+\gamma(q-p)}
\; .
\label{eq:bernoulli-exact-value}
\end{align}
\end{theorem}

\begin{proof}
See Appendix~\ref{app:bernoulli-exact}.
\end{proof}
\section{The Pollak Variational Frontier Algorithm}
\label{sec:implementation}

The theory above leads to a two-level computational procedure. The outer level searches over the candidate Pollak risk $r$ and the inner level solves the truncated fixed-$r$ design problem. Keeping these two levels separate matters because the dual recursion finds the best truncated detector for a \emph{given} $r$, whereas the Pollak minimax design is obtained only after locating the smallest $r$ that supports ARL $\gamma$. 
For a prescribed false-alarm level $\gamma$ and truncation order $N\ge\lceil\gamma\rceil$, define
\begin{align}
r_{\gamma,N}
=
\min\{r\in[1,\gamma]:\Gamma_N(r)\ge\gamma\}.
\end{align}
Theorem~\ref{thm:finite-convergence} shows that $r_{\gamma,N}\downarrow V_{\sf P}(\gamma)$. Thus $N$ controls the approximation of the infinite problem, while the outer search over $r$ locates the crossing for the chosen $N$.

\vspace{.05 in}
\noindent\textbf{Inner problem: synthesize the best detector at a fixed $r$.}
For a trial value $r$, solve
\begin{align}
\Gamma_N(r)
=
\min_{\boldsymbol\omega\ge0}
\mathfrak D_N(\boldsymbol\omega).
\end{align}
For each candidate vector $\boldsymbol\omega=(\omega_0,\ldots,\omega_{N-1})$, the dual objective is evaluated by the backward recursion
\begin{align}
W_n(m)
=
\left[
1+r\omega_n-m
+
\Einf W_{n+1}(\omega_{n+1}+\Lambda m)
\right]_+,
\qquad
n=N-1,\ldots,1,
\label{eq:pvfa-backward}
\end{align}
with $W_N\equiv0$ and $\omega_N=0$. The optimization is convex in the $N$ nonnegative dual weights. For finite-alphabet models the expectations reduce to finite sums. For continuous models they require numerical integration; rigorous one-sided conclusions require the validated enclosures described in Appendix~\ref{app:certified-pvfa}. At a dual minimizer $\boldsymbol\omega^\star$, the sequential statistic is
\begin{align}
M_0=\omega_0^\star,
\qquad
M_n=\omega_n^\star+\Lambda_n M_{n-1},
\qquad 1\le n<N\ .
\label{eq:pvfa-online-statistic}
\end{align}
For each $n$, the unique root of $H_n(b_{n,N})=0$ gives the off-boundary threshold $b_{n,N}$. A corresponding primal optimizer supplies both a feasible survival process and, when a boundary tie has positive probability, the saddle-compatible continuation kernel required on that tie set. Complementary slackness provides a useful diagnostic: a positive dual weight $\omega_\nu^\star$ identifies an active change-time constraint.

A numerical implementation should then perform a forward primal check. Using the candidate policy, evaluate the truncated survival probabilities $q_n$, the delay numerators $A_\nu^{[N]}$, and the resulting $\operatorname{ARL}_N$. This forward pass verifies the inequalities $A_\nu^{[N]}\le r q_\nu$ and determines whether the target $\gamma$ has actually been reached. This distinction matters when the inner optimization is solved numerically: a dual candidate by itself provides an \emph{upper} bound on $\Gamma_N(r)$, whereas an upper-crossing certificate requires an explicitly primal-feasible rule with $\operatorname{ARL}_N\ge\gamma$. When ties are null, the off-boundary threshold policy determines this forward candidate almost surely; when ties have positive mass, the forward check must use the saddle-compatible tie kernel.

\vspace{.05 in}
\noindent\textbf{Outer problem: locate the smallest feasible risk $r$.}
The monotonicity of $\Gamma_N(r)$ makes the outer problem one-dimensional. With an exact evaluation of $\Gamma_N(r)$, ordinary bisection on the sign of $(\Gamma_N(r)-\gamma)$ locates $r_{\gamma,N}$ to any desired precision.  For rigorous bounds on the \emph{untruncated} minimax value, however, the available tests are one-sided. A primal-feasible truncated process with $\operatorname{ARL}_N\ge\gamma$ proves $V_{\sf P}(\gamma)\le r$, while a dual upper bound satisfying
\begin{align}
\mathfrak D_N(\widehat{\boldsymbol\omega})+\varepsilon_N(r)<\gamma
\end{align}
proves $V_{\sf P}(\gamma)>r$. At a trial value exactly equal to $V_{\sf P}(\gamma)$, no strict conclusion needs to become available at any finite $N$. Therefore, a rigorous search must not rely on a rule that repeatedly tests a single undecided midpoint. Appendix~\ref{app:certified-pvfa} uses a safeguarded two-probe bracket. Starting with $r_L=1$ and $r_U=\gamma$, choose two interior probes, for example
\begin{align}
r_1=r_L+\frac{r_U-r_L}{3},
\qquad
r_2=r_U-\frac{r_U-r_L}{3}.
\end{align}
Apply the primal and dual tests at both probes. A certified lower point raises $r_L$; a certified upper point lowers $r_U$. At the exact mathematical level, if one probe happens to coincide with the minimax value and remains undecided, the other probe is distinct from the value and is classifiable for sufficiently large $N$. In a validated numerical implementation, the same safeguard removes the structural midpoint degeneracy, although finite termination additionally requires the one-sided numerical enclosures to become decisive. Each successful stage contracts the bracket by at least a factor $2/3$. The search stops when $r_U-r_L\le\eta$, yielding
\begin{align}
r_L<V_{\sf P}(\gamma)\le r_U,
\end{align}
and an admissible rule with risk at most $r_U$. This safeguard removes the possible nontermination of a strict midpoint test at an exact crossing.

\vspace{.05 in}
\noindent\textbf{From the truncated optimizer to an admissible detector.}
The truncated primal optimizer is initially an extended rule. It may terminate at the truncation order. Lemma~\ref{lem:tail-completion} appends the Bernoulli--geometric terminal completion without increasing the risk bound, and Lemma~\ref{lem:arl-calibration} then reduces the ARL to exactly $\gamma$ if the frontier solution has ARL larger than required. These operations are part of the detector construction, not merely proof devices.

\vspace{.05 in}
\noindent\textbf{Online implementation.}
After the offline synthesis stores $\{\omega_n^\star,b_{n,N}\}$, every strict off-boundary decision uses the scalar recursion in \eqref{eq:pvfa-online-statistic}:
\begin{align}
M_n<b_{n,N}
&\;\Longrightarrow\;
\text{continue},
&
M_n>b_{n,N}
&\;\Longrightarrow\;
\text{stop}.
\label{eq:algorithm-decision}
\end{align}
If $P_\infty(M_n=b_{n,N})=0$ for all $n$, this gives an almost-surely scalar online detector requiring one likelihood-ratio evaluation, one scalar update, and one comparison per observation. If a boundary tie has positive probability, the saddle-compatible tie kernel from Theorem~\ref{thm:finite-dual-structure} must also be implemented; in discrete models this kernel can retain dependence on the tied observation history and therefore need not be representable by $M_n$ alone.

\vspace{.05 in}
\noindent\textbf{Implementation summary.}
For prescribed $(P_\infty,P_0,\gamma)$:
\begin{enumerate}
\item Choose a truncation order $N\ge\lceil\gamma\rceil$ and a desired outer-search tolerance.
\item Maintain a bracket in $r$. For an exact truncated crossing, bisect using $\Gamma_N(r)-\gamma$; for rigorous bounds on $V_{\sf P}(\gamma)$, use the safeguarded two-probe procedure of Appendix~\ref{app:certified-pvfa}.
\item At each trial $r$, minimize $\mathfrak D_N(\boldsymbol\omega)$ using the backward recursion \eqref{eq:pvfa-backward}, and recover/check a corresponding primal candidate.
\item At the selected upper crossing, retain $\boldsymbol\omega^\star$, the boundaries $\{b_{n,N}\}$, and any saddle-compatible boundary-tie kernel required by the primal optimizer.
\item Implement the recursion \eqref{eq:pvfa-online-statistic} and the threshold decisions \eqref{eq:algorithm-decision}.
\item Apply terminal completion and, when needed, exact ARL calibration. The resulting admissible rule satisfies $J_{\sf P}(T_{\gamma,N})\le r_{\gamma,N}$ for the exact truncated crossing, with $r_{\gamma,N}\downarrow V_{\sf P}(\gamma)$ as $N\to\infty$.
\end{enumerate}

\section{Numerical Evaluation}
\label{sec:numerical}

The numerical section has three complementary purposes. A Bernoulli example is fully analytic and establishes an exact structural separation between the unrestricted optimum and the canonical constant-boundary GSR class. An exponential example contains a previously known exact-optimality regime and therefore provides a useful consistency check for the truncated computation. The Gaussian example has no closed-form finite-$\gamma$ optimum available here, so we use it to study truncation behavior and compare the resulting numerical benchmark with optimized classical procedures.
For the continuous-model experiments below, the displayed PVFA quantities are numerical evaluations of the truncated primal--dual construction. Appendix~\ref{app:certified-pvfa} gives conditions under which the same computations become rigorous one-sided bounds: the primal feasibility and ARL inequalities and the continuous expectations must be verified with one-sided numerical enclosures. Unless such validated enclosures are explicitly reported, the plotted continuous-model values should be interpreted as numerical approximations to the corresponding truncated bounds rather than as machine-certified inequalities. When validated endpoints $r_L$ and $r_U$ satisfy Corollary~\ref{cor:certified-bracket}, one has
\begin{align}
r_L<V_{\sf P}(\gamma)\le r_U,
\end{align}
and then $J_{\sf P}(T)-r_U>0$ is a rigorous lower bound on the suboptimality of a benchmark procedure $T$.

\subsection{Exact Separation from the Canonical Generalized SR Class}

Consider the Bernoulli change model $P_\infty(X_n=1) = \frac14$ and $P_0(X_n=1)
= \frac12$. This example isolates a structural question: does restricting the detector to a canonical SR first-crossing architecture lose anything
relative to optimization over the complete class of randomized, history-dependent stopping rules? To this end, consider the generalized
SR recursion
\begin{align}
R_n
=
(1+R_{n-1})\Lambda_n\ , \quad \mbox{and} \quad T_A^Q = \inf\{n\ge1:R_n\ge A\}
\; ,
\end{align}
where $A\geq 0$ is a constant threshold, and the initial state $R_0$ is independent of the observations and may
have an arbitrary distribution $Q$ supported strictly below $A$. Denote
the collection of all such procedures, over all thresholds $A$ and all
such initial distributions $Q$, by $\mathcal G_{\rm can}$. This class contains the
classical SR procedure, every deterministically head-started SR-$r$ procedure, arbitrary randomized head starts below the threshold, and SRP whenever its quasi-stationary initialization is
defined. The comparison therefore optimizes over a broad classical procedure class rather than against a particular choice of threshold or initialization. For the Bernoulli setting, Theorem~\ref{thm:bernoulli-exact} gives the unrestricted minimax value 
\begin{align}
V_{\sf P}(\gamma)
=
\frac{3\gamma}{\gamma+2},
\qquad
1<\gamma\le4
\; .
\label{eq:num-bernoulli-value}
\end{align}
By contrast, optimization over the canonical generalized Shiryaev–Roberts (GSR) class gives (proof in Appendix~\ref{app:canonical-gsr-separation}):
\begin{align}
\label{eq:num-gsr-value}
\inf_{\substack{T\in\mathcal G_{\rm can}\\
\Einf[T]\ge\gamma}}
J_{\sf P}(T) = 2\ , \qquad 1<\gamma\le4
\; .
\end{align}
Figure~\ref{fig:bernoulli-exact-separation} displays this separation. The
important point is not that SR procedures perform poorly:
rather, the example establishes that the canonical constant-boundary first-crossing GSR architecture is not, in general, rich enough to achieve the finite-$\gamma$ Pollak optimum. The shaded region is an exact
optimality gap, with neither curve obtained from simulation or numerical
approximation. This example therefore serves primarily as a finite-$\gamma$ structural
separation result. It demonstrates that enlarging the decision class
beyond canonical GSR first-crossing rules can matter, while also showing
that the benefit depends on the model and false-alarm regime. The next
example examines the complementary situation in which a highly optimized
SR rule already lies very close to the unrestricted
finite-$\gamma$ benchmark.

{\subsection{Exponential Model with a Known Exact-Optimality Regime}
We next revisit the exponential model
$f_\infty(x)=e^{-x}\mathbf{1}_{\{x\geq 0\}}$ and
$f_0(x)=2e^{-2x}\mathbf{1}_{\{x\geq 0\}}$, which provides a useful
benchmark because its finite-$\gamma$ behavior has been investigated in~\cite{polunchenkoTartakovsky2010}, which showed that a suitably head-started SR procedure is exactly
Pollak-minimax over the restricted false-alarm range $1<\gamma<\gamma_\star 
:= \left(1-\frac{1}{2}\log 3\right)^{-1}
\approx 2.219$. 
Figure~\ref{fig:exponential-performance-gaps} compares the Pollak risks of optimized SR-$r$, SRP, and CuSum with the numerically computed PVFA upper-candidate at truncation order $N=200$. The plotted quantity is the corresponding signed margin. For an exactly solved or rigorously validated upper endpoint $r_U$, optimized SR-$r$ in the known exact-optimality region satisfies $J_{\sf P}(T_{\mathrm{SR-r}})=V_{\sf P}(\gamma)\le r_U$, so the signed margin must be nonpositive. Outside that region, a \emph{validated} positive margin implies
\begin{align}
J_{\sf P}(T)
>
r_U
\ge
V_{\sf P}(\gamma)
\; ,
\end{align}
and therefore proves that the corresponding procedure is not exactly
Pollak-minimax at that finite false-alarm level.

\begin{figure}[t]
    \centering
    \begin{minipage}[t]{0.48\textwidth}
        \centering
        \includegraphics[width=\linewidth]{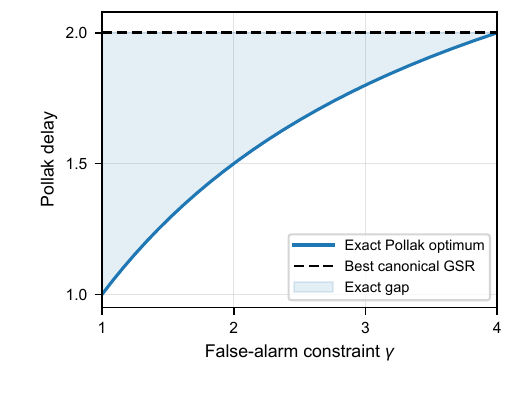}
        \captionof{figure}{\footnotesize 
        The solid curve is the exact unrestricted Pollak minimax value
        $V_{\sf P}(\gamma)=3\gamma/(\gamma+2)$, and the dashed line is the
        optimal value over the canonical constant-boundary first-crossing
        GSR class. The shaded region is the exact optimality gap.
        }
        \label{fig:bernoulli-exact-separation}
    \end{minipage}
    \hfill
    \begin{minipage}[t]{0.48\textwidth}
        \centering
        \includegraphics[width=\linewidth]{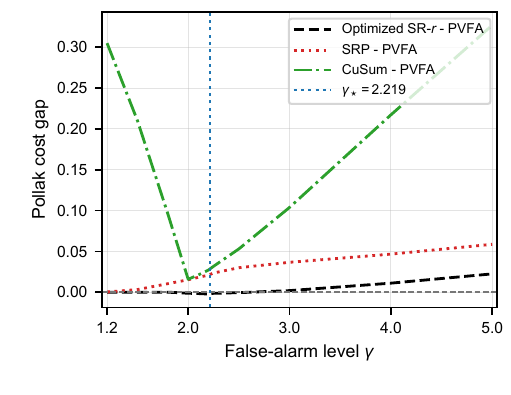}
        \captionof{figure}{\footnotesize  
        Signed numerical margins to the truncated PVFA benchmark in the exponential model as a
        function of ARL $\gamma$, with PVFA computed at truncation order $N=200$.
        The vertical dotted line marks $\gamma_\star=(1-\tfrac{1}{2}\log 3)^{-1}\approx2.219$, the upper endpoint of the known exact Pollak-optimality regime.}
        \label{fig:exponential-performance-gaps}
    \end{minipage}

\end{figure}

\subsection{Gaussian Mean-shift Models}
We next consider the Gaussian mean-shift family $P_\infty=\mathcal{N}(0,1)$ and
$P_0=\mathcal{N}(\theta,1)$.
Unlike the Bernoulli example above, the likelihood ratio in this model is
not bounded away from zero, and the explicit exact characterization of
Section~\ref{sec:likelihood-floor} does not apply.  We instead use the truncated
PVFA construction as an increasingly sharp upper bound on the unrestricted
Pollak minimax value. These experiments address two complementary questions.  First, we examine how the truncated benchmark improves as the truncation order increases for a fixed weak change.  Second, fixing the computational truncation order, we examine how the advantage relative to classical procedures depends on the statistical strength of the change.

Figure~\ref{fig:gaussian-horizon-convergence} considers the weak mean-shift setting $\theta=0.05$ with $\gamma=10$ and varies the truncation order $N$.
Theorem~\ref{thm:finite-convergence} guarantees monotone decrease of the \emph{exact} truncated crossing $r_{\gamma,N}$: increasing $N$ enlarges the truncated feasible class and produces a progressively sharper upper bound on $V_{\sf P}(\gamma)$. The plotted PVFA values are numerical approximations to these crossings. Numerically, the truncated benchmark lies above the classical procedures at short orders and then falls below CuSum and SRP and, at larger $N$, below the optimized deterministically head-started SR procedure. When a truncated value satisfies the inequality
\begin{align}
r_{\gamma,N}<J_{\sf P}(T_{\mathrm{SR-r}})
\end{align}
it implies
\begin{align}
V_{\sf P}(\gamma)\le r_{\gamma,N}<J_{\sf P}(T_{\mathrm{SR-r}}),
\end{align}
and therefore ensures that the optimized head-started SR rule is not exactly minimax for that finite-$\gamma$ instance. 

\begin{figure}[t]
    \centering
    \begin{minipage}[t]{0.48\textwidth}
        \centering
        \includegraphics[width=\linewidth]{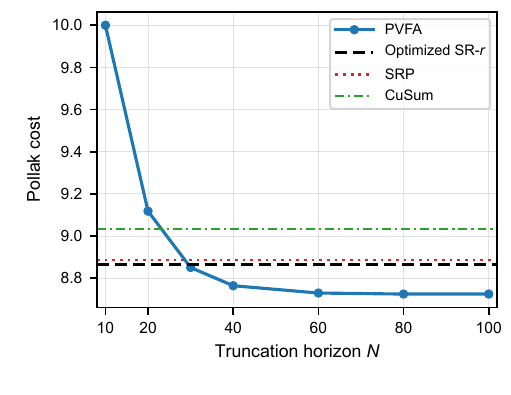}
        \captionof{figure}{\footnotesize
        Truncated convergence of PVFA with increasing truncation
        order $N$ for the Gaussian mean-shift model. The horizontal
        lines show the Pollak risks of optimized SR-$r$, SRP, and CuSum.
        }
        \label{fig:gaussian-horizon-convergence}
    \end{minipage}
    \hfill
    \begin{minipage}[t]{0.48\textwidth}
        \centering
        \includegraphics[width=\linewidth]{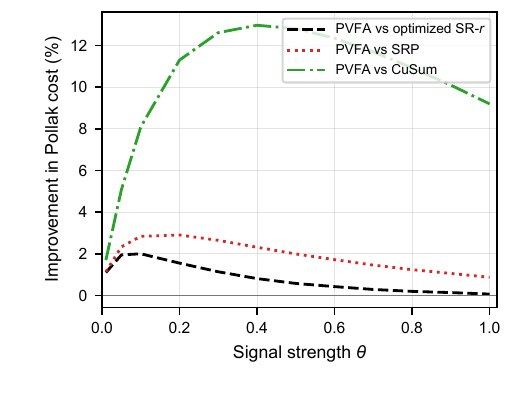}
        \captionof{figure}{\footnotesize
        Relative finite-$\gamma$ margin to the truncated PVFA benchmark as a function of the post-change mean $\theta$, with
        $P_\infty=\mathcal{N}(0,1)$,
        $P_0=\mathcal{N}(\theta,1)$, and PVFA computed at truncation
        order $N=100$.
        }
        \label{fig:gaussian-signal-strength}
    \end{minipage}
\end{figure}

Figure~\ref{fig:gaussian-signal-strength} examines whether this
finite-$\gamma$ separation depends systematically on the difficulty of the
change.  We fix $\gamma=10$ and $N=100$ and vary the mean shift
$\theta$.  
\begin{align}
\Delta_T(\theta)
:=
100\,
\frac{J_{\sf P}(T)-r_{\gamma,N}}
     {J_{\sf P}(T)}
\; .
\end{align}
Numerically, the relative margin clearly depends on signal strength. For great changes, the optimized SR-$r$ curve lies extremely close to the truncated PVFA benchmark. As the change weakens, the computed margin relative to optimized SR-$r$ rises to about two percent before decreasing again in the nearly indistinguishable limit; the corresponding margins relative to SRP and CuSum are larger over portions of the displayed range. These calculations suggest that the additional time-inhomogeneous degrees of freedom can matter more in weak-signal regimes. When the PVFA values are validated with the one-sided bounds of Appendix~\ref{app:certified-pvfa}, the same positive margins become rigorous nonoptimality statements for the benchmark procedures.

Figures~\ref{fig:gaussian-horizon-convergence} and \ref{fig:gaussian-signal-strength} provide two complementary numerical views of the finite-$\gamma$ regime. The first is consistent with the theorem that the exact truncated upper bounds sharpen monotonically with $N$. The second illustrates that the numerical separation from classical procedures is model- and signal-strength-dependent. These observations are consistent with the strong asymptotic performance of SR-type procedures while illustrating why asymptotic near-optimality need not imply exact finite-$\gamma$ optimality.

\section{Conclusions}
\label{sec:conclusion}

In this paper, we have considered the minimax quickest change-detection problem under Pollak's criterion, where the objective is to minimize the worst-case conditional post-change detection delay subject to a constraint $\gamma$ on the average run length to false alarm.  While the asymptotic regime in which $\gamma\to\infty$ has been
extensively studied and several procedures are known to enjoy strong asymptotic optimality properties, the finite-$\gamma$ problem remains substantially less understood. In particular, a general characterization of the unrestricted minimax solution at fixed false-alarm levels has been lacking.  In this paper, we have addressed this gap by developing a finite-$\gamma$ framework that recasts the optimization of Pollak's minimax criterion over the space of stopping times as an equivalent optimization over a class of processes, which we refer to as \emph{survival processes}.  This reformulation also reveals the structure of an optimal test and establishes the existence of an optimizer within the survival-process representation. The resulting optimization, however, remains infinite-dimensional and is therefore not directly tractable. We address this difficulty by truncating the survival process at a finite number of coordinates,  reducing the problem to a finite-dimensional linear program. Although this truncation introduces a suboptimality gap relative to the original infinite-dimensional problem, we have shown that the gap can be explicitly controlled through the truncation order and made arbitrarily small as the truncation level increases. The resulting algorithm also recovers the conventional Shiryaev--Roberts (SR) and Shiryaev--Roberts--Pollak (SRP) procedures as special cases, while allowing optimization over a substantially broader class of stopping rules. Both analytically and empirically, we show that the resulting rules can strictly outperform these classical procedures at finite $\gamma$. Finally, under an appropriate bounded-likelihood-ratio condition, we derive an exact universal lower bound on Pollak's minimax cost and identify settings in which this bound is achieved, providing an exact characterization of the minimax optimum.

\bibliographystyle{plain}
\bibliography{QCD_Pollak}

\newpage 

\appendix

\section{Proof of Lemma~\ref{lemma:survival-completeness}}
\label{app:survival-completeness}
\subsection{Forward direction: Realizability of an admissible survival process} 

Let $\mathbf y=\{y_n:n\in\mathbb N_0\}\in\mathcal Y$. By the definition of $\mathcal Y$ in Section~\ref{sec:survival}, the coordinates are equivalence classes modulo $P_\infty^n$-null sets. Choose compatible measurable representatives satisfying, pointwise,
\begin{align}
y_0=1,
\qquad
0\le y_n(x_{1:n})\le y_{n-1}(x_{1:n-1})\le1,
\qquad n\ge1.
\label{eq:proof-survival-consistency}
\end{align}
Such representatives may be obtained recursively by clipping, and the modification does not change any criterion because $P_0\ll P_\infty$.
We construct a randomized stopping rule that realizes these survival probabilities. For each $n\ge1$, define the conditional probability of continuing at time $n$, given that the detector has survived through time $n-1$, by
\begin{align}
c_n(x_{1:n})
:=
\begin{cases}
\displaystyle
\frac{y_n(x_{1:n})}
     {y_{n-1}(x_{1:n-1})},
&
y_{n-1}(x_{1:n-1})>0,
\\[2ex]
0,
&
y_{n-1}(x_{1:n-1})=0.
\end{cases}
\label{eq:continuation-probability}
\end{align}
Based on \eqref{eq:proof-survival-consistency} we necessarily have
\begin{align}
0\le c_n(x_{1:n})\le1
\; .
\end{align}
Let $\{U_n:n\in\mathbb{N}\}$ be the independent auxiliary random variables introduced earlier. Starting with the detector active at time zero, after observing $X_n$, and provided that it has not stopped before time $n$, let it continue beyond time $n$ if
\begin{align}
U_n
\le
c_n(X_{1:n}) \; ,
\end{align}
and stop at time $n$ otherwise. Denote the resulting stopping time by $T_{\mathbf y}$. This rule uses only the observations and auxiliary randomness available by time $n$, and therefore $T_{\mathbf y}$ is a randomized stopping time. We now verify that its survival process is exactly $\mathbf y$. For $n=0$,
\begin{align}
\mathbb P_U(T_{\mathbf y}>0)
=
1
=
y_0
\; .
\end{align}
For $n\ge1$, conditional on the observation history $X_{1:n}=x_{1:n}$, survival through time $n$ requires continuation at each of the times $\{1,\ldots,n\}$. Independence of the auxiliary random variables yields
\begin{align}
\mathbb P_U
\bigl(
T_{\mathbf y}>n
\mid
X_{1:n}=x_{1:n}
\bigr)
&=
\prod_{j=1}^{n}
c_j(x_{1:j}) =
y_n(x_{1:n})
\; .
\label{eq:survival-telescope}
\end{align}
Specifically, whenever the denominators are positive, the product in
\eqref{eq:survival-telescope} telescopes as
\begin{align}
\prod_{j=1}^{n}
\frac{y_j(x_{1:j})}
     {y_{j-1}(x_{1:j-1})}
=
\frac{y_n(x_{1:n})}{y_0}
=
y_n(x_{1:n})
\; .
\end{align}
On the other hand, when $y_j(x_{1:j})=0$ for some $j$, admissibility implies that
$y_k(x_{1:k})=0$ for every $k\ge j$ along the same history, and the
zero convention in \eqref{eq:continuation-probability} gives the same
conclusion. Hence $T_{\mathbf y}$ realizes the prescribed survival
process.

\subsection{Converse direction: Survival process induced by a stopping time}

For the converse direction, let $T\ge1$ be any randomized stopping time. Since $\{T>n\}\in\mathcal F_n$, choose a measurable finite-history representative $s_n$ such that
\begin{align}
s_n(X_{1:n},U_{0:n})
=
\mathbf 1_{\{T>n\}}
\qquad\text{almost surely}\ .
\end{align}
Define
\begin{align}
y_n(x_{1:n})
:=
\int s_n(x_{1:n},u_{0:n})\,\mathbb P_U(\mathrm d u_{0:n})\ ,
\label{eq:induced-survival-proof}
\end{align}
with $y_0=1$. Measurability follows from Tonelli's theorem. Because the survival events are nested, the induced equivalence classes admit versions satisfying
\begin{align}
0\le y_n(x_{1:n})\le y_{n-1}(x_{1:n-1})\le1
\qquad
P_\infty^n\text{-almost surely},
\qquad n\ge1 \ .
\end{align}
If pointwise compatibility is desired, recursively clipping the representatives as in the definition of $\mathcal Y$ changes them only on null sets and yields $0\le y_n(x_{1:n})\le y_{n-1}(x_{1:n-1})\le1$ for every history.
Thus the induced sequence
\begin{align}
\mathbf y=\{y_n:n\in\mathbb N_0\}
\end{align}
belongs to $\mathcal Y$. Hence every admissible survival process can be implemented by a
randomized stopping time, and every randomized stopping time induces an admissible survival process.

\section{Proof of Lemma~\ref{lem:frontier-lower-bound}}
\label{app:frontier-lower-bound}

By contradiction, assume that there exists an admissible stopping time $T$ satisfying $\Einf[T]
\ge \gamma$ and $J_{\sf P}(T) < r_\gamma$.
Set
\begin{align}
r :=
J_{\sf P}(T)
\; .
\end{align}
Since the residual detection delay is at least one and
$r<r_\gamma\le\gamma$, we have $r\in[1,\gamma]$. By the completeness of the survival representation, $T$ induces an
admissible survival process $\mathbf y\in\mathcal Y$. Since
\begin{align}
D_\nu(T)
\le
r
\; ,
\qquad
\nu\ge0
\; ,
\end{align}
we have
\begin{align}
A_\nu(\mathbf y)
\le
r q_\nu(\mathbf y)
\; ,
\qquad
\nu\ge0
\; .
\end{align}
Therefore, $\mathbf y \in \mathcal F_{\le r}$. By the definition of the frontier,
\begin{align}
\Gamma(r)
\ge
\operatorname{ARL}(\mathbf y)
=
\Einf[T]
\ge
\gamma
\; .
\end{align}
Hence $r$ cannot lie below the first crossing $r_\gamma$, which
contradicts $r=J_{\sf P}(T)<r_\gamma$. Therefore, every admissible
stopping time satisfying the false-alarm constraint obeys
\begin{align}
J_{\sf P}(T)
\ge
r_\gamma
\; .
\end{align}
Taking the infimum over all such stopping times yields $V_{\sf P}(\gamma) \ge r_\gamma$.

\section{Compactness, Tail Control, and Frontier Achievement}
\label{app:additional_section_4_3}
This appendix proves the compactness and tail-control results used in
Theorem~\ref{thm:frontier-optimizer}. We first establish coordinatewise
weak-* compactness of the survival process, then prove the finite-block
probability-transfer lemma and the resulting uniform block contraction,
and finally combine these ingredients to prove frontier achievement. We use
the transfer modulus $\tau_M(a)$ and the constants $m_R$ and $\delta_R$
defined in \eqref{eq:transfer-modulus} and \eqref{eq:block-delta}; the
truncated quantities $\operatorname{ARL}_N$ and $A_\nu^{[N]}$ are those
defined in \eqref{eq:finite-arl} and \eqref{eq:finite-delay-numerator}.

\subsection{Proof of Lemma~\ref{lem:survival-compact}}
\label{app:survival-compact}

The proof has two main steps. We first show that, at every fixed time
$n$, the possible survival functions form a sequentially compact
set under weak-* convergence. We then show that a subsequence can be
chosen to converge simultaneously at every finite time and that the
survival-consistency constraints are preserved in the limit. For a fixed $n\ge1$, define
\begin{align}
K_n
=
\left\{
f\in L^\infty(P_\infty^n):
0\le f\le1
\right\}
\; .
\end{align}
Since $L^\infty(P_\infty^n)$ is the dual of
$L^1(P_\infty^n)$, the Banach--Alaoglu theorem implies that its
closed unit ball is compact in the weak-* topology. The set $K_n$
is weak-* closed within this unit ball. Indeed, nonnegativity is
preserved under weak-* limits because
\begin{align}
f\ge0
\quad\Longleftrightarrow\quad
\int f g\,\mathrm dP_\infty^n
\ge0
\quad
\text{for every }g\in L_+^1(P_\infty^n)
\; ,
\end{align}
and the constraint $f\le1$ follows in the same way by applying this
argument to $1-f$. Hence $K_n$ is weak-* compact. The sigma-field is countably
generated, and therefore $L^1(P_\infty^n)$ is separable. As a consequence, the weak-* topology on the bounded set $K_n$ can be generated by a metric. Thus, compactness of $K_n$ implies the sequential property that we need: every sequence in $K_n$ has a weak-* convergent subsequence.

We now apply this property successively to the coordinates of a
sequence
\begin{align}
\{\mathbf y^{(k)}\}_{k\ge1}
\subseteq
\mathcal Y
\; .
\end{align}
First choose a subsequence along which $y_1^{(k)}$ converges weak-* in $K_1$. From that subsequence, choose a further
subsequence along which $y_2^{(k)}$ also converges. Continuing in
this way, and taking the diagonal subsequence, we obtain indices
\begin{align}
k_1<k_2<\cdots
\end{align}
and functions $y_n^\star\in K_n$ such that, for every fixed
$n\ge1$,
\begin{align}
y_n^{(k_j)}
\longrightarrow
y_n^\star
\qquad
\text{weak-* as }j\to\infty
\; .
\label{eq:coordinate-weakstar-limit}
\end{align}
Set $y_0^\star=1$. It remains to verify that the limiting sequence
\begin{align}
\mathbf y^\star
=
\{y_n^\star:n\in\mathbb N_0\}
\end{align}
still satisfies the survival-consistency constraints. Here the two adjacent coordinates live on different product spaces, so we introduce the lift only for this functional-analytic argument. For $n\ge2$, define
\begin{align}
(\iota_n f)(x_{1:n})
:=
f(x_{1:n-1}),
\end{align}
and let $\iota_1y_0\equiv1$.
The admissibility of every
$\mathbf y^{(k_j)}$ gives
\begin{align}
0
\le
y_n^{(k_j)}
\le
\iota_n y_{n-1}^{(k_j)}
\qquad
P_\infty^n\text{-a.s.}
\label{eq:prelimit-consistency}
\end{align}
The lift operation is continuous under weak-* convergence. To see
this, let $g\in L^1(P_\infty^n)$. Then
\begin{align}
\int g\,\iota_n f\,\mathrm dP_\infty^n
=
\int
f(x_{1:n-1})
\left[
\int
g(x_{1:n-1},x_n)
P_\infty(\mathrm dx_n)
\right]
P_\infty^{n-1}(\mathrm dx_{1:n-1})
\; .
\label{eq:lift-weakstar-continuity}
\end{align}
The expression in brackets is an $L^1(P_\infty^{n-1})$ function.
Therefore, weak-* convergence of $f$ implies weak-* convergence of
$\iota_n f$.
Consequently, from
\eqref{eq:coordinate-weakstar-limit},
\begin{align}
\iota_n y_{n-1}^{(k_j)}
-
y_n^{(k_j)}
\longrightarrow
\iota_n y_{n-1}^\star
-
y_n^\star
\end{align}
weak-*. Since every function on the left-hand side is nonnegative by
\eqref{eq:prelimit-consistency}, and the nonnegative cone is weak-*
closed, the limit is also nonnegative. Hence
\begin{align}
0
\le
y_n^\star
\le
\iota_n y_{n-1}^\star
\qquad
P_\infty^n\text{-a.s.} \; ,
\qquad
n\ge1
\; .
\end{align}
The weak-* limit is an equivalence class and the preceding inequalities hold $P_\infty^n$-almost surely, exactly as required in the optimization-space definition of $\mathcal Y$. If a pointwise representative is needed for implementation, define recursively $\bar y_0=1$ and
\begin{align}
\bar y_n
:=
\min\bigl\{\max\{y_n^\star,0\},\,\iota_n\bar y_{n-1}\bigr\}.
\end{align}
Then $\bar y_n=y_n^\star$ $P_\infty^n$-almost surely and the survival-consistency inequalities hold pointwise. Relabeling this compatible representative as $\mathbf y^\star$, we obtain $\mathbf y^\star\in\mathcal Y$. Hence every sequence in $\mathcal Y$ has a subsequence converging coordinatewise weak-* to another element of $\mathcal Y$.

\subsection{Proof of Lemma~\ref{lem:transfer}}
\label{app:transfer}

Let
\begin{align}
Z_M
:=
\frac{\mathrm d P_0^M}{\mathrm d P_\infty^M}
\end{align}
denote the likelihood ratio over a block of $M$ observations. Since
\begin{align}
\mathbb E_\infty^M[Z_M]
=
1
\; ,
\end{align}
there exists a finite $C>0$ such that
\begin{align}
\mathbb E_\infty^M
\left[
Z_M\one_{\{Z_M>C\}}
\right]
<
\frac{a}{2}
\; .
\label{eq:transfer-tail-choice}
\end{align}
For any measurable $0\le\phi\le1$,
\begin{align}
\mathbb E_0^M[\phi]
&=
\mathbb E_\infty^M[Z_M\phi]
\le
C\mathbb E_\infty^M[\phi]
+
\mathbb E_\infty^M
\left[
Z_M\one_{\{Z_M>C\}}
\right]
<
C\mathbb E_\infty^M[\phi]
+
\frac{a}{2}
\; .
\end{align}
Therefore, if
\begin{align}
\mathbb E_0^M[\phi]\ge a \; ,
\end{align}
then necessarily
\begin{align}
\mathbb E_\infty^M[\phi]
\ge
\frac{a}{2C}
>
0
\; .
\end{align}
Taking the infimum over all such $\phi$ proves the claim.

\subsection{Proof of Lemma~\ref{lem:block-contraction}}
\label{app:block-contraction}

Recall $m_R$ and $\delta_R$ from \eqref{eq:block-delta}.
If $q_n(\mathbf y)=0$, monotonicity of the survival process gives
$q_{n+m_R}(\mathbf y)=0$, and
\eqref{eq:block-contraction} holds immediately. Hence suppose
$q_n(\mathbf y)>0$. Because $\mathbf y\in\mathcal F_{\le r}$, corresponding to a
change at time $n$ we have
\begin{align}
D_n(\mathbf y)
=
\frac{A_n(\mathbf y)}{q_n(\mathbf y)}
\le
R
\; .
\label{eq:block-row-bound}
\end{align}
Conditional on survival through time $n$, the positive
integer-valued residual delay $T-n$ therefore has mean at most $R$.
On the event $T>n+m_R$, this residual delay is at least $m_R+1$.
Hence
\begin{align}
\mathbb P_n
\left(
T>n+m_R
\mid
T>n
\right)
&\le
\frac{R}{m_R+1}
 \le
\frac12
\; .
\label{eq:post-change-half}
\end{align}
Hence, after a change at time $n$, conditional on having
survived through time $n$, the detector stops during the next
$m_R$ observations with probability at least $1/2$. It remains to be transferred to the no-change law. The
detector may depend on the entire history that has survived through
time $n$, so this transfer must account for arbitrary history
dependence. Let $\mathsf H_n$ denote the complete detector history
through time $n$, including its auxiliary randomization, conditional
on $T>n$. The conditional distribution of $\mathsf H_n$ is the
same under $\mathbb P_n$ and $\mathbb P_\infty$, because the two
models coincide through time $n$. For a surviving history $h$ and a future observation block
$x_{1:m_R}$, let
\begin{align}
\psi(h,x_{1:m_R})
\in[0,1]
\end{align}
denote the conditional probability, over the detector's future
randomization, that it stops by time $n+m_R$. Averaging over the
common conditional distribution of the surviving history gives
\begin{align}
\bar\psi(x_{1:m_R})
:=
\mathbb E
\left[
\psi(\mathsf H_n,x_{1:m_R})
\mid
T>n
\right]
\; .
\label{eq:averaged-block-event}
\end{align}
The inequality in~\eqref{eq:post-change-half} implies
\begin{align}
\mathbb E_0^{m_R}[\bar\psi]
\ge
\frac12
\; .
\end{align}
By Lemma~\ref{lem:transfer} and the definition of $\delta_R$,
\begin{align}
\mathbb E_\infty^{m_R}[\bar\psi]
\ge
\delta_R
\; .
\end{align}
Thus, under the no-change law,
\begin{align}
\mathbb P_\infty
\left(
T\le n+m_R
\mid
T>n
\right)
\ge
\delta_R
\; .
\label{eq:prechange-stop-block}
\end{align}
Since
\begin{align}
\mathbb P_\infty
\left(
T>n+m_R
\mid
T>n
\right)
=
\frac{q_{n+m_R}(\mathbf y)}
     {q_n(\mathbf y)}
\; ,
\end{align}
we obtain \eqref{eq:block-contraction}. Iterating the contraction at block endpoints gives
\begin{align}
q_{j m_R}(\mathbf y)
\le
(1-\delta_R)^j
\; .
\label{eq:block-endpoints}
\end{align}
Since $q_n(\mathbf y)$ is nonincreasing in $n$, for every
$k\ge j$,
\begin{align}
\sum_{n=k m_R}^{(k+1)m_R-1}
q_n(\mathbf y)
\le
m_R(1-\delta_R)^k
\; .
\end{align}
Summing these bounds over $k\ge j$ yields
\begin{align}
\sum_{n=j m_R}^{\infty}
q_n(\mathbf y)
&\le
m_R
\sum_{k=j}^{\infty}
(1-\delta_R)^k
=
\frac{m_R}{\delta_R}
(1-\delta_R)^j
\; ,
\end{align}
which proves \eqref{eq:uniform-tail-bound}.

\subsection{Proof of Theorem~\ref{thm:frontier-optimizer}}
\label{app:frontier-optimizer}

We use $A_\nu^{[N]}$ and $\operatorname{ARL}_N$ as defined in \eqref{eq:finite-delay-numerator} and \eqref{eq:finite-arl}.
We proceed in three steps. First, we show that limits of fixed-risk feasible survival processes remain fixed-risk feasible. Second, we show that the ARL is preserved under such limits. Finally, we apply these facts to an optimizing sequence.

\subsubsection{Preservation of the fixed-risk constraints}

Suppose that, along the subsequence provided by Lemma~\ref{lem:survival-compact} 
\begin{align}
y_n^{(k_j)}
\longrightarrow
y_n^\star
\end{align}
for every fixed $n$, where $\mathbf{y}^\star\in\cal Y$. We must additionally verify that it satisfies the fixed-risk inequalities. Fix a change time $\nu$ and a finite $N\ge\nu+1$. Because $A_\nu^{[N]}$ depends on only finitely many coordinates and each coefficient $L_{\nu,n}$ belongs to $L^1(P_\infty^n)$, weak-* convergence gives
\begin{align}
A_\nu^{[N]}\bigl(\mathbf y^{(k)}\bigr)
\longrightarrow
A_\nu^{[N]}(\mathbf y^\star)
\; .
\label{eq:partial-A-convergence}
\end{align}
Similarly,
\begin{align}
q_\nu\bigl(\mathbf y^{(k)}\bigr)
\longrightarrow
q_\nu(\mathbf y^\star)
\; .
\label{eq:q-convergence}
\end{align}
Since every $\mathbf y^{(k)}\in\mathcal F_{\le r}$,
\begin{align}
A_\nu\bigl(\mathbf y^{(k)}\bigr)
\le
r q_\nu\bigl(\mathbf y^{(k)}\bigr)
\; .
\end{align}
The truncated numerator is no larger than the full numerator, so
\begin{align}
A_\nu^{[N]}\bigl(\mathbf y^{(k)}\bigr)
\le
r q_\nu\bigl(\mathbf y^{(k)}\bigr)
\; .
\end{align}
Passing to the limit in $k$ yields
\begin{align}
A_\nu^{[N]}(\mathbf y^\star)
\le
r q_\nu(\mathbf y^\star)
\; .
\label{eq:partial-row-limit}
\end{align}
This holds for every finite $N$. Since the summands in
$A_\nu^{[N]}(\mathbf y^\star)$ are nonnegative,
\begin{align}
A_\nu(\mathbf y^\star)
=
\lim_{N\to\infty}
A_\nu^{[N]}(\mathbf y^\star) \; ,
\end{align}
and therefore
\begin{align}
A_\nu(\mathbf y^\star)
\le
r q_\nu(\mathbf y^\star)
\; .
\end{align}
Since $\nu$ was arbitrary,
\begin{align}
\mathbf y^\star
\in
\mathcal F_{\le r}
\; .
\label{eq:limit-fixed-risk-feasible}
\end{align}
Thus, the fixed-risk feasible class is preserved under the convergence used here.

\subsubsection{Preservation of the ARL}

For every fixed $N$, weak-* convergence of the first $N$ coordinates gives
\begin{align}
\operatorname{ARL}_N\bigl(\mathbf y^{(k)}\bigr)
&\longrightarrow
\operatorname{ARL}_N\bigl(\mathbf y^\star\bigr)
\; .
\end{align}
Indeed, for every fixed $n$,
\begin{align}
q_n\bigl(\mathbf y^{(k)}\bigr)
=
\Einf\!\left[y_n^{(k)}(X_{1:n})\right]
&\longrightarrow
\Einf\!\left[y_n^\star(X_{1:n})\right]
=
q_n\bigl(\mathbf y^\star\bigr)
\; ,
\end{align}
and $\operatorname{ARL}_N$ is a finite sum of these quantities. The remaining issue is to pass from convergence of the truncated ARLs to
convergence of the full infinite-horizon ARLs. This is precisely where the
uniform tail bound of Lemma~\ref{lem:block-contraction} is used. By
Step~C.4.1, the limiting process satisfies
\begin{align}
\mathbf y^\star
\in
\mathcal F_{\le r}
\; .
\end{align}
Hence Lemma~\ref{lem:block-contraction} applies both to every
$\mathbf y^{(k)}$ and to $\mathbf y^\star$. In particular, for every
$\mathbf y\in\mathcal F_{\le r}$,
\begin{align}
0
&\le
\operatorname{ARL}(\mathbf y)
-
\operatorname{ARL}_N(\mathbf y)
=
\sum_{n=N}^{\infty}
q_n(\mathbf y)
\le
\frac{m_r}{\delta_r}
(1-\delta_r)^{\lfloor N/m_r\rfloor}
\; .
\end{align}
Define
\begin{align}
b_N(r)
:=
\frac{m_r}{\delta_r}
(1-\delta_r)^{\lfloor N/m_r\rfloor}
\; .
\end{align}
Since $\delta_r>0$, we have
\begin{align}
b_N(r)
&\longrightarrow
0
\qquad
\text{as }N\to\infty
\; .
\end{align}
Furthermore, the bound is uniform over the entire fixed-risk feasible class
$\mathcal F_{\le r}$. For every fixed $N$, we therefore have
\begin{align}
\left|
\operatorname{ARL}\bigl(\mathbf y^{(k)}\bigr)
-
\operatorname{ARL}\bigl(\mathbf y^\star\bigr)
\right|
&\le
\left|
\operatorname{ARL}_N\bigl(\mathbf y^{(k)}\bigr)
-
\operatorname{ARL}_N\bigl(\mathbf y^\star\bigr)
\right|
\nonumber\\
&\quad
+
\left|
\operatorname{ARL}\bigl(\mathbf y^{(k)}\bigr)
-
\operatorname{ARL}_N\bigl(\mathbf y^{(k)}\bigr)
\right|
\nonumber\\
&\quad
+
\left|
\operatorname{ARL}\bigl(\mathbf y^\star\bigr)
-
\operatorname{ARL}_N\bigl(\mathbf y^\star\bigr)
\right|
\nonumber\\
&\le
\left|
\operatorname{ARL}_N\bigl(\mathbf y^{(k)}\bigr)
-
\operatorname{ARL}_N\bigl(\mathbf y^\star\bigr)
\right|
+
2b_N(r)
\; .
\end{align}
Taking $k\to\infty$ for fixed $N$ and using the convergence of the
truncated ARLs gives
\begin{align}
\limsup_{k\to\infty}
\left|
\operatorname{ARL}\bigl(\mathbf y^{(k)}\bigr)
-
\operatorname{ARL}\bigl(\mathbf y^\star\bigr)
\right|
&\le
2b_N(r)
\; .
\end{align}
Finally, letting $N\to\infty$ and using $b_N(r)\to0$, we obtain
\begin{align}
\operatorname{ARL}\bigl(\mathbf y^{(k)}\bigr)
&\longrightarrow
\operatorname{ARL}\bigl(\mathbf y^\star\bigr)
\; .
\end{align}
Thus, the ARL is preserved when passing to the limit within the fixed-risk
feasible class.

\subsubsection{Apply the result to an optimizing sequence}

Choose an optimizing sequence satisfying
\begin{align}
\mathbf y^{(k)}
&\in
\mathcal F_{\le r} \; ,
\\
\operatorname{ARL}\bigl(\mathbf y^{(k)}\bigr)
&\ge
\Gamma(r)-\frac{1}{k}
\; .
\end{align}
By Lemma~\ref{lem:survival-compact}, this sequence has a subsequence converging coordinatewise weak-* to some
\begin{align}
\mathbf y_r^\star
\in
\mathcal Y
\; .
\end{align}
Step~1 shows that
\begin{align}
\mathbf y_r^\star
\in
\mathcal F_{\le r} \; ,
\end{align}
and Step~2 shows that
\begin{align}
\operatorname{ARL}(\mathbf y_r^\star)
=
\lim_{k\to\infty}
\operatorname{ARL}\bigl(\mathbf y^{(k)}\bigr)
=
\Gamma(r)
\; .
\end{align}
Therefore $\mathbf y_r^\star$ achieves the supremum defining
$\Gamma(r)$, proving the theorem.

\section{Proof of Lemma~\ref{lem:tail-completion}}
\label{app:tail-completion}

\subsection{Preservation of the ARL} 
The original and modified rules coincide on $E^c$. On $E$, the original rule stops at time $N$, while \eqref{eq:completion-conditional-mean} shows that the modified rule also has conditional mean stopping time $N$. Therefore,
\begin{align}
\Einf[\widetilde T]
=
\Einf[T]
\; .
\end{align}
Thus, the repair does not sacrifice any of the false-alarm performance achieved by the frontier optimizer.

\subsection{Preservation of the earlier Pollak fixed-risk constraints}

Consider a change time
\begin{align}
0\le\nu\le N-2
\; .
\end{align}
On the terminal event $E$, both the original and modified rules necessarily survive beyond $\nu$. Therefore, the survival denominator at time $\nu$ is unchanged. Furthermore, on $E$, the expected terminal time remains $N$. Hence,
\begin{align}
\mathbb E_\nu
\left[
(\widetilde T-\nu)\one_E
\right]
&=
\mathbb P_\nu(E)(N-\nu) =
\mathbb E_\nu
\left[
(T-\nu)\one_E
\right]
\; .
\label{eq:completion-early-numerator}
\end{align}
Since the two rules coincide on $E^c$, both the numerator and denominator of every earlier Pollak fixed-risk constraint are preserved:
\begin{align}
A_\nu(\widetilde T)
=
A_\nu(T)
\qquad \mbox{and} \qquad
q_\nu(\widetilde T)
=
q_\nu(T),
\qquad
0\le\nu\le N-2
\; .
\label{eq:completion-early-components}
\end{align}
Consequently,
\begin{align}
D_\nu(\widetilde T)
=
D_\nu(T) \; ,
\qquad
0\le\nu\le N-2
\; .
\end{align}

\subsection{Control of the newly created late fixed-risk constraints}

The behavior is different beginning at time $N-1$, because these change times did not have positive survival under the original extended rule. Under the repaired rule,
\begin{align}
\{\widetilde T>N-1\}
=
E\cap\{B=1\}
\; .
\end{align}
Conditional on this event, the residual detection delay is $G$. Therefore,
\begin{align}
D_{N-1}(\widetilde T)
=
\mathbb E[G]
=
r
\; .
\end{align}
Now consider
\begin{align}
\nu
=
N-1+k \; ,
\qquad
k\ge1
\; .
\end{align}
Survival to time $\nu$ requires
\begin{align}
E,\qquad B=1,\qquad G>k \; .
\end{align}
Conditional on this survival event, the remaining delay is $G-k$. By the memoryless property of the geometric distribution,
\begin{align}
\Law(G-k\mid G>k)
=
\Law(G)
\; ,
\end{align}
and hence
\begin{align}
D_\nu(\widetilde T)
=
r \; ,
\qquad
\nu\ge N-1
\; .
\end{align}
Thus, the repair does not create any new Pollak fixed-risk constraint with a delay exceeding the candidate risk level $r$.

\subsection{Positive survival at every finite time}

Finally, the geometric branch has a strictly positive tail because $r>1$. For every finite $k\ge0$,
\begin{align}
\Pinf(\widetilde T>N-1+k)
&=
\Pinf(E)
\mathbb P(B=1)
\mathbb P(G>k) >
0
\; .
\label{eq:completion-positive-tail}
\end{align}
Together with $q_{N-1}(T)>0$, this shows that
\begin{align}
\Pinf(\widetilde T>n)>0
\end{align}
for every finite $n$. The repaired rule is therefore admissible under Pollak's original criterion.

The role of the Bernoulli and geometric variables can now be seen separately. The Bernoulli variable controls how much probability is assigned to the new tail. Choosing
\begin{align}
\mathbb P(B=1)=\frac1r
\end{align}
allows a geometric residual of mean $r$ to be introduced while preserving the original one-step mean through
\begin{align}
\mathbb E[BG]=1 \; .
\end{align}
The geometric distribution then supplies the second property we need: memorylessness makes the conditional residual mean equal to $r$ at every subsequent change time. Together, these two properties preserve the old Pollak fixed-riskconstraints and control all newly created ones.

It is also useful to note the scope of the assumptions in Lemma~\ref{lem:tail-completion}. The condition $r>1$ is essential for creating a genuinely positive infinite tail: when $r=1$, a geometric random variable with mean one is degenerate at one and provides no extension beyond the original finite terminal time. Similarly, the condition $N\ge2$ is exactly the case relevant to the minimax crossing with $\gamma>1$. Any extended rule with ARL greater than one must have positive survival beyond time one, so its first zero-survival time cannot be $N=1$. In the next section, we will also establish that the crossing risk $r_\gamma$ is strictly larger than one.

Lemma~\ref{lem:tail-completion} therefore resolves the second obstacle in the achievability argument. Once an extended frontier optimizer at risk $r$ has been found, either it already has positive survival at every finite time, in which case no repair is needed, or its first zero-survival time is finite, and the geometric completion converts it into a genuine Pollak-admissible stopping rule with
\begin{align}
\Einf[\widetilde T]
=
\Gamma(r) \; ,
\qquad
J_{\sf P}(\widetilde T)
\le
r
\; .
\end{align}
At the crossing $r=r_\gamma$, this is precisely the construction needed to establish $V_{\sf P}(\gamma) \le r_\gamma$
once we verify that the frontier actually reaches $\gamma$ at $r_\gamma$. The next section completes that argument and, when the resulting ARL exceeds $\gamma$, provides an exact final calibration to the prescribed false-alarm level.

\section{Proof of Lemma~\ref{lem:arl-calibration}}
\label{app:arl-calibration}

The terminal repair in the previous section solves the admissibility problem: starting from an extended optimizer with a finite terminal time, it produces a genuine stopping rule with positive survival at every finite time, the same ARL, and Pollak risk no larger than the prescribed risk level $r$. In this subsection, we address the final issue of ARL calibration. The frontier problem maximizes the ARL subject to a fixed risk bound. Consequently, a frontier optimizer at a risk level $r$ may have $\Einf[\bar T]
= a > \gamma$. Such a rule already satisfies the original constraint $\Einf[\bar T]
\ge \gamma$. While admissible,  we also would like the final minimax rule to satisfy the prescribed false-alarm level \emph{exactly}, i.e., $\Einf[T]
= \gamma$. This requires reducing the ARL without increasing the Pollak risk.

When ARL exceeds $\gamma$, a natural way to reduce the ARL is to mix the rule $\bar T$ with a rule that stops immediately. What makes this particularly convenient under Pollak's criterion is the conditioning on survival. If the change occurs at any time $\nu\ge1$, the immediate-stop branch cannot survive to $\nu$, and therefore disappears completely once we condition on $T>\nu$. Hence all Pollak fixed-riskconstraints with $\nu\ge1$ remain unchanged. Only the fixed-riskconstraint corresponding to an immediate change, $\nu=0$, must be checked separately. The following lemma formalizes this observation. If
\begin{align}
a
=
\gamma \; ,
\end{align}
there is nothing to modify, and we simply take
\begin{align}
T_\gamma
=
\bar T
\; .
\end{align}
Suppose therefore that
\begin{align}
a
>
\gamma
>
1
\; .
\end{align}
Define
\begin{align}
\alpha
:=
\frac{\gamma-1}{a-1}
\; .
\label{eq:calibration-alpha}
\end{align}
Since $1<\gamma<a$,
\begin{align}
0
<
\alpha
<
1
\; .
\end{align}
Using auxiliary randomization independent of the observations, choose at the outset between two branches. With probability $\alpha$, run a law-equivalent copy of $\bar T$; with probability $1-\alpha$, stop immediately at time $1$. Denote the resulting stopping rule by $T_\gamma$. Equivalently,
\begin{align}
T_\gamma
:=
\begin{cases}
\bar T,
&
\text{with probability }\alpha,
\\[1ex]
1,
&
\text{with probability }1-\alpha.
\end{cases}
\label{eq:calibration-mixture}
\end{align}
As in the preceding terminal-repair construction, the required branch randomization can be obtained from the atomless auxiliary detector randomization by a measure-preserving splitting, so this construction remains within the behavioral randomized class. We first verify the ARL. By construction,
\begin{align}
\Einf[T_\gamma]
&=
(1-\alpha)\cdot 1
+
\alpha a =
1+\alpha(a-1) =
1+
\frac{\gamma-1}{a-1}(a-1) =
\gamma
\; ,
\end{align}
which proves \eqref{eq:calibration-exact-arl}. We next examine the Pollak fixed-riskconstraints. Consider any
\begin{align}
\nu
\ge
1
\; .
\end{align}
The immediate-stop branch has $T_\gamma=1$ and therefore can never contribute to the event
\begin{align}
\{T_\gamma>\nu\} \; .
\end{align}
Consequently,
\begin{align}
q_\nu(T_\gamma)
&=
\alpha q_\nu(\bar T) \; ,
\label{eq:calibration-q}\\
A_\nu(T_\gamma)
&=
\alpha A_\nu(\bar T)
\; .
\label{eq:calibration-A}
\end{align}
Since $\bar T$ is admissible and $\alpha>0$, the denominator remains strictly positive. Hence
\begin{align}
D_\nu(T_\gamma)
&=
\frac{A_\nu(T_\gamma)}
     {q_\nu(T_\gamma)} =
\frac{\alpha A_\nu(\bar T)}
     {\alpha q_\nu(\bar T)} =
D_\nu(\bar T) \; ,
\qquad
\nu\ge1
\; .
\end{align}
Thus, every Pollak fixed-riskconstraint after time zero is preserved exactly. It remains only to check the fixed-riskconstraint at $\nu=0$. Since every stopping rule considered here satisfies $T\ge1$,
\begin{align}
q_0(T_\gamma)
=
1 \; ,
\end{align}
and therefore
\begin{align}
D_0(T_\gamma)
=
\mathbb E_0[T_\gamma]
\; .
\end{align}
Using the mixture,
\begin{align}
D_0(T_\gamma)
&=
(1-\alpha)\cdot1
+
\alpha D_0(\bar T) \le
(1-\alpha)
+
\alpha r \le
r \; ,
\end{align}
where the final inequality follows from $r\ge1$. Thus,
\begin{align}
D_\nu(T_\gamma)
\le
r \; ,
\qquad
\forall\,\nu\ge0
\; ,
\end{align}
and hence
\begin{align}
J_{\sf P}(T_\gamma)
\le
r
\; .
\end{align}
Finally, admissibility is preserved. For every finite $\nu\ge1$,
\begin{align}
q_\nu(T_\gamma)
=
\alpha q_\nu(\bar T)
>
0 \; ,
\end{align}
while $q_0(T_\gamma)=1$. Therefore $T_\gamma$ has positive survival at every finite change time and is admissible under Pollak's original criterion.

The calibration mechanism reduces the no-change mean from $a$ to any prescribed value between $1$ and $a$. At the same time, Pollak's conditioning protects all fixed-riskconstraints with $\nu\ge1$: once we condition on survival to the change time, the immediate-stop branch is absent, so both the numerator and denominator are multiplied by the same factor $\alpha$, and their ratio is unchanged. The only affected fixed-riskconstraint is $\nu=0$, where the new delay is a convex combination of $1$ and the original delay and therefore cannot exceed the original risk bound. This distinction is also useful conceptually. We note that the exact calibration is not needed merely to prove $V_{\sf P}(\gamma) \le r$. Rather, it shows the stronger statement that whenever a feasible rule at risk $r$ has ARL at least $\gamma$, it can be replaced by another rule with the same risk bound and with the false-alarm constraint active, i.e., $\Einf[T_\gamma] = \gamma$. Once we establish that the frontier reaches the prescribed ARL at the crossing risk $r_\gamma$, the existence and repair results from the preceding sections provide an admissible rule $\bar T$ satisfying
\begin{align}
\Einf[\bar T]
\ge
\gamma
\qquad \mbox{and} \qquad
J_{\sf P}(\bar T)
\le
r_\gamma
\; .
\end{align}
Lemma~\ref{lem:arl-calibration} then produces an admissible rule
$T_\gamma^\star$ satisfying
\begin{align}
\Einf[T_\gamma^\star]
=
\gamma
\qquad \mbox{and} \qquad
J_{\sf P}(T_\gamma^\star)
\le
r_\gamma
\; .
\label{eq:calibrated-crossing-rule}
\end{align}
This is precisely the construction required for the achievability direction $V_{\sf P}(\gamma) \le r_\gamma$.
The remaining task is therefore to verify that the infimum defining $r_\gamma$ is actually achieved, so that $\Gamma(r_\gamma) \ge \gamma$. Once this is established, the lower bound from Lemma~\ref{lem:frontier-lower-bound} and the calibrated construction above will together yield the exact minimax identity $V_{\sf P}(\gamma) = r_\gamma$.

\section{Proof of Lemma~\ref{lem:crossing-achieved}}
\label{app:crossing-achieved}

By the definition of $r_\gamma$, we can choose a decreasing sequence
\begin{align}
r_k
\downarrow
r_\gamma \; ,
\qquad
\Gamma(r_k)
\ge
\gamma
\; .
\label{eq:crossing-sequence}
\end{align}
For each $k$, Theorem~\ref{thm:frontier-optimizer} provides a frontier optimizer
\begin{align}
\mathbf y^{(k)}
\in
\mathcal F_{\le r_k}
\end{align}
satisfying
\begin{align}
\operatorname{ARL}\bigl(\mathbf y^{(k)}\bigr)
=
\Gamma(r_k)
\ge
\gamma
\; .
\label{eq:crossing-optimizers}
\end{align}
Since $r_k\le\gamma$, we have
\begin{align}
\mathcal F_{\le r_k}
\subseteq
\mathcal F_{\le\gamma}\; ,
\end{align}
and therefore every $\mathbf y^{(k)}$ belongs to the common fixed-risk class
$\mathcal F_{\le\gamma}$. By the compactness argument developed in the preceding section, there is a subsequence, which we continue to index by $k$, and a survival process $\mathbf y^\star$ such that, for every fixed $n$,
\begin{align}
y_n^{(k)}
\longrightarrow
y_n^\star
\qquad
\text{weak-* as }k\to\infty
\; .
\label{eq:crossing-weakstar}
\end{align}
We first show that the limiting process satisfies the fixed-risk constraints at the limiting risk $r_\gamma$. Fix a change time $\nu$ and a finite $N\ge\nu+1$. Since the truncated numerator depends on only finitely many survival coordinates,
\begin{align}
A_\nu^{[N]}(\mathbf y^\star)
&=
\lim_{k\to\infty}
A_\nu^{[N]}\bigl(\mathbf y^{(k)}\bigr)
\le
\lim_{k\to\infty}
r_k q_\nu\bigl(\mathbf y^{(k)}\bigr)
=
r_\gamma q_\nu(\mathbf y^\star)
\; .
\label{eq:crossing-partial-row}
\end{align}
The inequality follows because
$\mathbf y^{(k)}\in\mathcal F_{\le r_k}$.
Since \eqref{eq:crossing-partial-row} holds for every finite $N$, letting $N\to\infty$ gives
\begin{align}
A_\nu(\mathbf y^\star)
\le
r_\gamma q_\nu(\mathbf y^\star)
\; .
\end{align}
Since $\nu$ was arbitrary,
\begin{align}
\mathbf y^\star
\in
\mathcal F_{\le r_\gamma}
\; .
\label{eq:crossing-limit-feasible}
\end{align}
It remains to verify that the limiting process retains an ARL of at least $\gamma$. This is precisely the role of the uniform tail control established earlier. On the common class
$\mathcal F_{\le\gamma}$, the infinite ARL is continuous under the convergence in \eqref{eq:crossing-weakstar}. Therefore,
\begin{align}
\operatorname{ARL}(\mathbf y^\star)
&=
\lim_{k\to\infty}
\operatorname{ARL}\bigl(\mathbf y^{(k)}\bigr) \ge
\gamma
\; .
\end{align}
Together with \eqref{eq:crossing-limit-feasible}, this shows that
\begin{align}
\Gamma(r_\gamma)
\ge
\operatorname{ARL}(\mathbf y^\star)
\ge
\gamma
\; .
\end{align}
Hence $r_\gamma$ itself belongs to the crossing set, and the infimum in its definition is a minimum.

\section{Proof of Lemma~\ref{lem:r-gamma-positive}}
\label{app:r-gamma-positive}

Consider the fixed-risk class at $r=1$. If
\begin{align}
\mathbf y
\in
\mathcal F_{\le 1}
\; ,
\end{align}
then the change-time constraint corresponding to $\nu=0$ gives
\begin{align}
A_0(\mathbf y)
&\le
q_0(\mathbf y)
=
1
\; .
\end{align}
Since
\begin{align}
A_0(\mathbf y)
&=
1
+
\sum_{n=1}^{\infty}
\Ezero\!\left[
y_n(X_{1:n})
\right]
\; ,
\end{align}
and every term in the sum is nonnegative, we must have
\begin{align}
\Ezero\!\left[
y_n(X_{1:n})
\right]
&=
0
\; ,
\qquad
n\ge1
\; .
\end{align}
Hence
\begin{align}
y_n
&=
0
\qquad
P_0^n\text{-almost surely}
\; ,
\qquad
n\ge1
\; .
\end{align}
By the mutual absolute continuity of $P_0$ and $P_\infty$, the finite
product laws $P_0^n$ and $P_\infty^n$ are equivalent. Therefore,
\begin{align}
q_n(\mathbf y)
=
\Einf\!\left[
y_n(X_{1:n})
\right]
&=
0
\; ,
\qquad
n\ge1
\; ,
\end{align}
and consequently
\begin{align}
\operatorname{ARL}(\mathbf y)
&=
q_0(\mathbf y)
=
1
\; .
\end{align}
Conversely, the survival process defined by
\begin{align}
y_0^\circ
&=
1
\; ,
&
y_n^\circ
&=
0
\; ,
\qquad
n\ge1
\; ,
\end{align}
belongs to $\mathcal F_{\le1}$ and has ARL equal to one. Hence $\Gamma(1)= 1$. By Lemma~\ref{lem:crossing-achieved}, $\Gamma(r_\gamma)\ge \gamma$.
If $r_\gamma=1$, then
\begin{align}
\gamma
&\le
\Gamma(r_\gamma)
=
\Gamma(1)
=
1
\; ,
\end{align}
which contradicts $\gamma>1$. Therefore, $r_\gamma > 1$.

\section{Proof of Theorem~\ref{thm:main}}
\label{app:main}

Lemma~\ref{lem:frontier-lower-bound} has already established $V_{\sf P}(\gamma)
\ge r_\gamma$. It remains to construct an admissible stopping rule achieving the reverse inequality.
By Lemma~\ref{lem:crossing-achieved},
\begin{align}
\Gamma(r_\gamma)
\ge
\gamma
\; .
\end{align}
Theorem~\ref{thm:frontier-optimizer} therefore provides a survival process
\begin{align}
\mathbf y_\gamma^\star
\in
\mathcal F_{\le r_\gamma}
\end{align}
such that
\begin{align}
\operatorname{ARL}(\mathbf y_\gamma^\star)
=
\Gamma(r_\gamma)
\ge
\gamma
\; .
\label{eq:crossing-optimizer-arl}
\end{align}
By Lemma~\ref{lemma:survival-completeness}, this survival process can be implemented by a randomized stopping time, which we denote by $T^\circ$. There are now two possibilities.
If
\begin{align}
\Pinf(T^\circ>n)
>
0 \; ,
\qquad
\forall\,n<\infty \; ,
\end{align}
then $T^\circ$ is already admissible under Pollak's original criterion. Since
$\mathbf y_\gamma^\star\in\mathcal F_{\le r_\gamma}$,
\begin{align}
J_{\sf P}(T^\circ)
\le
r_\gamma
\; .
\end{align}
Otherwise, let $N$ be the first finite time at which
\begin{align}
\Pinf(T^\circ>N)
=
0
\; .
\end{align}
Because
\begin{align}
\Einf[T^\circ]
=
\Gamma(r_\gamma)
\ge
\gamma
>
1 \; ,
\end{align}
the first zero cannot occur at $N=1$; hence $N\ge2$. Furthermore, Lemma~\ref{lem:r-gamma-positive} gives
\begin{align}
r_\gamma>1 \; .
\end{align}
Thus the hypotheses of Lemma~\ref{lem:tail-completion} are satisfied. Applying the geometric terminal completion produces an admissible stopping rule $\bar T$ such that
\begin{align}
\Einf[\bar T]
&=
\Gamma(r_\gamma)
\ge
\gamma \; ,
\label{eq:completed-crossing-arl}\\
J_{\sf P}(\bar T)
&\le
r_\gamma \; ,
\label{eq:completed-crossing-risk}\\
\Pinf(\bar T>n)
&>
0 \; ,
\qquad
\forall\,n<\infty
\; .
\label{eq:completed-crossing-positive}
\end{align}
If no terminal repair was needed, we simply take
\begin{align}
\bar T
=
T^\circ
\end{align}
and the same conclusions hold. At this point we already have an admissible rule whose ARL is at least $\gamma$ and whose Pollak risk is at most $r_\gamma$. Therefore,
\begin{align}
V_{\sf P}(\gamma)
\le
r_\gamma
\; .
\label{eq:main-upper}
\end{align}
Combining \eqref{eq:main-upper} with the lower bound $V_{\sf P}(\gamma)
\ge r_\gamma$ gives
\begin{align}
V_{\sf P}(\gamma)
=
r_\gamma
\; .
\label{eq:main-equality}
\end{align}
It remains only to obtain exact ARL calibration. If $\Einf[\bar T]
= \gamma$, set
\begin{align}
T_\gamma^\star
=
\bar T
\; .
\end{align}
If instead
\begin{align}
\Einf[\bar T]
>
\gamma \; ,
\end{align}
apply Lemma~\ref{lem:arl-calibration}. It produces an admissible stopping time $T_\gamma^\star$ satisfying
\begin{align}
\Einf[T_\gamma^\star]
&=
\gamma \; ,
\\
J_{\sf P}(T_\gamma^\star)
&\le
r_\gamma \; ,
\\
\Pinf(T_\gamma^\star>n)
&>
0 \; ,
\qquad
\forall\,n<\infty
\; .
\end{align}
Since $T_\gamma^\star$ is feasible for the original minimax problem,
\begin{align}
J_{\sf P}(T_\gamma^\star)
\ge
V_{\sf P}(\gamma)
=
r_\gamma \; .
\end{align}
Therefore,
\begin{align}
J_{\sf P}(T_\gamma^\star)
=
V_{\sf P}(\gamma)
=
r_\gamma
\; ,
\end{align}
which proves the theorem.

\section{Proof of Theorem~\ref{thm:finite-frontier-bounds}}
\label{app:finite-frontier-bound}

We prove the two inequalities in \eqref{eq:frontier-error-bound} separately. First, consider any survival process $\mathbf y$ feasible for the
$N$-truncated problem at risk level $r$. Extend it by setting
$y_n=0$ for every $n\ge N$. For every $\nu<N$, the corresponding
untruncated numerator is then identical to its truncated
counterpart,
\begin{align}
A_\nu(\mathbf y)
=
A_\nu^{[N]}(\mathbf y)
\le
r q_\nu(\mathbf y)
\; .
\end{align}
For every $\nu\ge N$, both $q_\nu(\mathbf y)$ and
$A_\nu(\mathbf y)$ are zero. Hence the zero-extended process is
feasible for the extended fixed-risk problem defining $\Gamma(r)$.
Furthermore, its untruncated ARL equals its truncated ARL.
Therefore,
\begin{align}
\Gamma_N(r)
\le
\Gamma(r)
\; .
\label{eq:app-finite-lower}
\end{align}
For the reverse inequality, consider any untruncated process
$\mathbf y$ feasible at risk level $r$, and truncate it after time
$N-1$. For every $\nu<N$, truncation leaves $q_\nu(\mathbf y)$
unchanged and can only decrease the corresponding delay numerator.
Thus,
\begin{align}
A_\nu^{[N]}(\mathbf y)
\le
A_\nu(\mathbf y)
\le
r q_\nu(\mathbf y)
\; ,
\qquad
0\le\nu<N
\; .
\end{align}
Hence the truncated process is feasible for the $N$-truncated problem,
and therefore
\begin{align}
\operatorname{ARL}_N(\mathbf y)
\le
\Gamma_N(r)
\; .
\end{align}
The full and truncated ARLs differ only by the survival tail:
\begin{align}
\operatorname{ARL}(\mathbf y)
=
\operatorname{ARL}_N(\mathbf y)
+
\sum_{n=N}^{\infty}q_n(\mathbf y)
\; .
\label{eq:app-arl-tail}
\end{align}
Let $j=\lfloor N/m_r\rfloor$. Since $jm_r\le N$, the uniform
tail bound in Lemma~\ref{lem:block-contraction} gives
\begin{align}
\sum_{n=N}^{\infty}q_n(\mathbf y)
&\le
\sum_{n=jm_r}^{\infty}q_n(\mathbf y)
\le
\frac{m_r}{\delta_r}
(1-\delta_r)^j
=
\varepsilon_N(r)
\; .
\end{align}
Combining this inequality with \eqref{eq:app-arl-tail} yields
\begin{align}
\operatorname{ARL}(\mathbf y)
\le
\Gamma_N(r)+\varepsilon_N(r)
\; .
\end{align}
Taking the supremum over all untruncated processes feasible at
risk level $r$ gives
\begin{align}
\Gamma(r)
\le
\Gamma_N(r)+\varepsilon_N(r)
\; .
\end{align}
Together with \eqref{eq:app-finite-lower}, this proves~\eqref{eq:frontier-error-bound}.

\section{Proof of Lemma~\ref{lem:finite-crossing}}
\label{app:finite-crossing}

Define the truncation crossing set
\begin{align}
\mathcal C_{\gamma,N}
:=
\left\{
r\in[1,\gamma]:
\Gamma_N(r)\ge\gamma
\right\}
\; .
\end{align}
We first show that this set is nonempty. Let $m=\lfloor\gamma\rfloor$ and
$\theta=\gamma-m\in[0,1)$. Consider the observation-independent
randomized deadline
\begin{align}
T^{\rm d}
=
\begin{cases}
m,
&
\text{with probability }1-\theta,
\\
m+1,
&
\text{with probability }\theta,
\end{cases}
\label{eq:app-random-deadline}
\end{align}
where the second branch is absent when $\theta=0$. Since
$N\ge\lceil\gamma\rceil$, this stopping rule terminates within truncation depth $N$, and
\begin{align}
\Einf[T^{\rm d}]
=
m+\theta
=
\gamma
\; .
\end{align}
Because the rule is independent of the observations, its conditional
residual delay is no larger than $\gamma$ at every change time for
which survival is positive. Indeed, for $\nu<m$,
\begin{align}
D_\nu(T^{\rm d})
=
\Einf[T^{\rm d}]-\nu
=
\gamma-\nu
\le
\gamma
\; ,
\end{align}
while, if $\theta>0$, the only remaining positive-survival fixed-riskconstraint is
$\nu=m$, for which $D_m(T^{\rm d})=1\le\gamma$. All later fixed-riskconstraints
disappear in the truncated formulation.
Consequently,
\begin{align}
\Gamma_N(\gamma)
\ge
\gamma
\; ,
\end{align}
and therefore $\mathcal C_{\gamma,N}\neq\varnothing$. 
Before proving closedness of the crossing set, we verify that the maximum
defining $\Gamma_N(r)$ is achieved for every fixed $r\in[1,\gamma]$.
For each $1\le n<N$, define
\begin{align}
K_n
:=
\left\{
f\in L^\infty(P_\infty^n):
0\le f\le 1
\quad
P_\infty^n\text{-almost surely}
\right\}
\; ,
\end{align}
and equip $K_n$ with the weak-* topology of
$L^\infty(P_\infty^n)$. By the Banach--Alaoglu theorem, each $K_n$
is weak-* compact. Hence the finite product
\begin{align}
\prod_{n=1}^{N-1}K_n
\end{align}
is compact in the product weak-* topology. The truncated survival space $\mathcal Y_N$ is a closed subset of this
product. Indeed, the survival-consistency constraints
\begin{align}
0
\le
y_n
\le
\iota_n y_{n-1}
\qquad
P_\infty^n\text{-almost surely}
\; ,
\qquad
1\le n<N
\; ,
\end{align}
are preserved under coordinatewise weak-* convergence because the lift
maps $\iota_n$ are weak-* continuous and the nonnegative cone is
weak-* closed. Therefore, $\mathcal Y_N$ is compact. For fixed $r\in[1,\gamma]$, define the truncated feasible set
\begin{align}
\mathcal F_{N,\le r}
:=
\left\{
\mathbf y\in\mathcal Y_N:
A_\nu^{[N]}(\mathbf y)
\le
r q_\nu(\mathbf y),
\quad
0\le\nu<N
\right\}
\; .
\end{align}
For every $\nu<N$, the mappings
\begin{align}
\mathbf y
&\longmapsto
q_\nu(\mathbf y)
\; ,
&
\mathbf y
&\longmapsto
A_\nu^{[N]}(\mathbf y)
\end{align}
are continuous under the product weak-* topology, since
$A_\nu^{[N]}$ is a finite sum of integrations against the
$L^1(P_\infty^n)$ functions $L_{\nu,n}$. Hence
$\mathcal F_{N,\le r}$ is closed in $\mathcal Y_N$ and is therefore
compact. It is also nonempty, since the process
\begin{align}
y_0
&=
1
\; ,
&
y_n
&=
0
\; ,
\qquad
1\le n<N
\; ,
\end{align}
belongs to $\mathcal F_{N,\le r}$.
Finally, the truncated objective
\begin{align}
\operatorname{ARL}_N(\mathbf y)
&=
\sum_{n=0}^{N-1}
q_n(\mathbf y)
\end{align}
is continuous. It therefore achieves its maximum on
$\mathcal F_{N,\le r}$. Consequently, for every
$r\in[1,\gamma]$, there exists
$\mathbf y_{r,N}^\star\in\mathcal F_{N,\le r}$ such that
\begin{align}
\Gamma_N(r)
&=
\operatorname{ARL}_N\bigl(\mathbf y_{r,N}^\star\bigr)
\; .
\end{align}
We next show that the crossing set is closed. Let
$\{r_k\}\subseteq\mathcal C_{\gamma,N}$ satisfy $r_k\to r$.
By the achievability established above, for every $k$ choose an optimizer
\begin{align}
\mathbf y^{(k)}
&\in
\mathcal F_{N,\le r_k}
\end{align}
such that
\begin{align}
\operatorname{ARL}_N\bigl(\mathbf y^{(k)}\bigr)
&=
\Gamma_N(r_k)
\ge
\gamma
\; .
\end{align}
The truncated survival space is compact under the finite product
weak-* topology. Hence, along a subsequence,
$\mathbf y^{(k)}$ converges coordinate by coordinate to some
$\mathbf y^\star\in\mathcal Y_N$.
All quantities appearing in the truncated problem depend on only
finitely many coordinates and are continuous under this convergence.
Thus, for every $0\le\nu<N$,
\begin{align}
A_\nu^{[N]}(\mathbf y^\star)
&=
\lim_{k\to\infty}
A_\nu^{[N]}(\mathbf y^{(k)})
\le
\lim_{k\to\infty}
r_kq_\nu(\mathbf y^{(k)})
=
r q_\nu(\mathbf y^\star)
\; ,
\end{align}
and similarly
\begin{align}
\operatorname{ARL}_N(\mathbf y^\star)
=
\lim_{k\to\infty}
\operatorname{ARL}_N(\mathbf y^{(k)})
\ge
\gamma
\; .
\end{align}
Hence $r\in\mathcal C_{\gamma,N}$. The set
$\mathcal C_{\gamma,N}$ is therefore a nonempty closed subset of
the compact interval $[1,\gamma]$, so its minimum exists. It remains to be shown that this minimum is strictly greater than 1.
Suppose $r=1$. The change-time constraint corresponding to
$\nu=0$ gives

\begin{align}
A_0^{[N]}(\mathbf y)
=
1+
\sum_{n=1}^{N-1}
\mathbb E_0[y_n]
\le
q_0(\mathbf y)
=
1
\; .
\end{align}
Since every term in the sum is nonnegative, we must have
$\mathbb E_0[y_n]=0$ for every $1\le n<N$. Hence $y_n=0$,
$P_0^n$-almost surely. Due to the equivalence of the pre- and post-change distributions, the finite product laws are equivalent, so the same holds
$P_\infty^n$-almost surely. Consequently,
$q_n(\mathbf y)=0$ for every $n\ge1$, and therefore
\begin{align}
\operatorname{ARL}_N(\mathbf y)
=
1
\; .
\end{align}
Thus $\Gamma_N(1)=1<\gamma$, implying
$r_{\gamma,N}>1$. Since the crossing set is contained in
$[1,\gamma]$, we also have $r_{\gamma,N}\le\gamma$, completing
the proof.

\section{Proof of Theorem~\ref{thm:finite-convergence}}
\label{app:finite-convergence}

We first establish convergence of the truncated crossing values. Recall
that
\begin{align}
r_{\gamma,N}
\ge
r_\gamma
=
V_{\sf P}(\gamma)
\; ,
\label{eq:app-crossing-lower}
\end{align}
because $\Gamma_N(r)\le\Gamma(r)$. Furthermore,
$\Gamma_N(r)\le\Gamma_{N+1}(r)$, since every $N$-truncated process
can be embedded into the $(N+1)$-truncated problem by setting its
additional survival coordinate equal to zero. Hence
\begin{align}
r_{\gamma,N+1}
\le
r_{\gamma,N}
\; .
\label{eq:app-crossing-monotone}
\end{align}
Thus, $\{r_{\gamma,N}\}$ is nonincreasing and bounded below by
$r_\gamma$. It remains to show that its limit cannot exceed
$r_\gamma$.

Let $T_\gamma^\star$ be the exact minimizer from
Theorem~\ref{thm:main}, so that
\begin{align}
\Einf[T_\gamma^\star]
&=
\gamma
\; ,
&
J_{\sf P}(T_\gamma^\star)
&=
r_\gamma
\; .
\end{align}
For each $N$, truncate this rule according to
\begin{align}
T_\gamma^{[N]}
:=
T_\gamma^\star\wedge N
\; .
\end{align}
Define the ARL lost through truncation by
\begin{align}
e_N
&:=
\gamma-\Einf[T_\gamma^{[N]}]
=
\sum_{n=N}^{\infty}
q_n(T_\gamma^\star)
\; .
\label{eq:app-eN}
\end{align}
Since $T_\gamma^\star$ has finite mean, $e_N\downarrow0$.
Furthermore, because $J_{\sf P}(T_\gamma^\star)=r_\gamma$, the uniform
tail bound gives
\begin{align}
0
\le
e_N
\le
\varepsilon_N(r_\gamma)
\; .
\label{eq:app-eN-bound}
\end{align}
Truncation leaves every survival probability $q_\nu$ unchanged for
$\nu<N$ and can only decrease the corresponding delay numerator.
Therefore, the zero-extended rule $T_\gamma^{[N]}$ satisfies
\begin{align}
A_\nu(T_\gamma^{[N]})
\le
r_\gamma q_\nu(T_\gamma^{[N]})
\; ,
\qquad
\nu\ge0
\; ,
\label{eq:app-truncated-risk}
\end{align}
where both sides are zero after the truncation point. We now restore the ARL lost through truncation. Construct a shifted copy $S_\gamma^{[N]}$ that ignores the first observation and then
runs a law-equivalent copy of $T_\gamma^{[N]}$ on
$X_2,X_3,\ldots$, using fresh auxiliary randomization. Under the
no-change law,
\begin{align}
\Einf[S_\gamma^{[N]}]
=
\Einf[T_\gamma^{[N]}]+1
\; .
\label{eq:app-shift-mean}
\end{align}
For every $\nu\ge1$,
\begin{align}
q_\nu(S_\gamma^{[N]})
&=
q_{\nu-1}(T_\gamma^{[N]})
\; ,
\\
A_\nu(S_\gamma^{[N]})
&=
A_{\nu-1}(T_\gamma^{[N]})
\; .
\label{eq:app-shift-rows}
\end{align}
Hence all such fixed-riskconstraints satisfy the same risk bound $r_\gamma$. For
$\nu=0$, since $q_0=1$,
\begin{align}
A_0(S_\gamma^{[N]})
&=
1+A_0(T_\gamma^{[N]}) \le
r_\gamma+1
\; .
\label{eq:app-shift-row-zero}
\end{align}
Fix any $N$ for which $N+1\ge\lceil\gamma\rceil$ and
\begin{align}
\varepsilon_N(r_\gamma)\le1.
\label{eq:app-epsilon-at-most-one}
\end{align}
Then \eqref{eq:app-eN-bound} gives $0\le e_N\le1$, so $e_N$ is a valid mixing probability. Independently of the observations, select $S_\gamma^{[N]}$ with probability $e_N$ and $T_\gamma^{[N]}$ with probability $1-e_N$. Denote the resulting zero-extended rule by $\widehat T_\gamma^{[N]}$. Using \eqref{eq:app-shift-mean},
\begin{align}
\Einf[\widehat T_\gamma^{[N]}]
=
\Einf[T_\gamma^{[N]}]+e_N
=
\gamma.
\label{eq:app-restored-arl}
\end{align}
For every $\nu\ge1$, both component rules satisfy the risk bound $r_\gamma$. At $\nu=0$, \eqref{eq:app-truncated-risk} and \eqref{eq:app-shift-row-zero} give
\begin{align}
A_0(\widehat T_\gamma^{[N]})
\le
(1-e_N)r_\gamma+e_N(r_\gamma+1)
=
r_\gamma+e_N.
\end{align}
Thus, whenever $r_\gamma+e_N\le\gamma$, the rule $\widehat T_\gamma^{[N]}$ is feasible for the $(N+1)$-truncated problem at risk level $r_\gamma+e_N$ and has ARL exactly $\gamma$, implying
\begin{align}
r_{\gamma,N+1}\le r_\gamma+e_N.
\label{eq:app-case-one-bound}
\end{align}

If instead $r_\gamma+e_N>\gamma$, Lemma~\ref{lem:finite-crossing} gives $r_{\gamma,N+1}\le\gamma$, and therefore
\begin{align}
r_{\gamma,N+1}-r_\gamma
\le
\gamma-r_\gamma
<
e_N.
\label{eq:app-case-two-bound}
\end{align}
If $r_\gamma=\gamma$, the conclusion is immediate because \eqref{eq:app-crossing-lower} and $r_{\gamma,N+1}\le\gamma$ force $r_{\gamma,N+1}=r_\gamma$. Thus, in all cases, for every $N$ satisfying \eqref{eq:app-epsilon-at-most-one},
\begin{align}
0
\le
r_{\gamma,N+1}-r_\gamma
\le
e_N
\le
\varepsilon_N(r_\gamma).
\end{align}
Using $r_\gamma=V_{\sf P}(\gamma)$ gives
\begin{align}
0
\le
r_{\gamma,N+1}-V_{\sf P}(\gamma)
\le
e_N
\le
\varepsilon_N\!\left(V_{\sf P}(\gamma)\right).
\label{eq:app-crossing-error}
\end{align}
Since the right-hand side converges to zero and
$\{r_{\gamma,N}\}$ is nonincreasing,
\begin{align}
r_{\gamma,N}
\downarrow
V_{\sf P}(\gamma)
\; .
\label{eq:app-crossing-convergence}
\end{align}
It remains to establish the corresponding statement for actual
admissible stopping rules. Fix $N\ge\lceil\gamma\rceil$, and let
$\mathbf y_{\gamma,N}^\star$ be an optimizer of
$\Gamma_N(r_{\gamma,N})$. By the definition of $r_{\gamma,N}$,
\begin{align}
\operatorname{ARL}_N(\mathbf y_{\gamma,N}^\star)
=
\Gamma_N(r_{\gamma,N})
\ge
\gamma
\; .
\label{eq:app-finite-opt-arl}
\end{align}
Extend $\mathbf y_{\gamma,N}^\star$ by zero beyond the truncation
depth and use Lemma~\ref{lemma:survival-completeness} to realize it
as a randomized stopping rule. By
Lemma~\ref{lem:finite-crossing},
\begin{align}
r_{\gamma,N}
>
1
\; .
\end{align}
Moreover, \eqref{eq:app-finite-opt-arl} and $\gamma>1$ imply that its
first zero-survival time is strictly larger than one. Hence
Lemma~\ref{lem:tail-completion} converts it into an admissible
randomized stopping rule $\overline T_{\gamma,N}$ satisfying
\begin{align}
\Einf[\overline T_{\gamma,N}]
&=
\Gamma_N(r_{\gamma,N})
\ge
\gamma
\; ,
\\
J_{\sf P}(\overline T_{\gamma,N})
&\le
r_{\gamma,N}
\; .
\label{eq:app-completed-rule}
\end{align}
If its ARL equals $\gamma$, set
$T_{\gamma,N}=\overline T_{\gamma,N}$. Otherwise,
Lemma~\ref{lem:arl-calibration} produces an admissible randomized
stopping rule $T_{\gamma,N}$ such that
\begin{align}
\Einf[T_{\gamma,N}]
&=
\gamma
\; ,
&
J_{\sf P}(T_{\gamma,N})
&\le
r_{\gamma,N}
\; .
\label{eq:app-near-optimal-rule}
\end{align}
Since $T_{\gamma,N}$ is feasible for Pollak's original minimax
problem,
\begin{align}
V_{\sf P}(\gamma)
\le
J_{\sf P}(T_{\gamma,N})
\le
r_{\gamma,N}
\; .
\end{align}
Therefore,
\begin{align}
0
\le
J_{\sf P}(T_{\gamma,N})-V_{\sf P}(\gamma)
\le
r_{\gamma,N}-V_{\sf P}(\gamma)
\longrightarrow
0
\; ,
\end{align}
where the convergence follows from
\eqref{eq:app-crossing-convergence}. This proves all claims of the
theorem.

\section{Gaussian Specialization of the Truncation Bound}
\label{app:gaussian-convergence}

We derive the constants in the Gaussian mean-shift example. Let
\begin{align}
P_\infty=\mathcal N(0,1),
\qquad
P_0=\mathcal N(\theta,1),
\qquad
\theta\ne0.
\end{align}
For a block of $M$ observations define
\begin{align}
S_M:=\operatorname{sgn}(\theta)\sum_{j=1}^{M}X_j.
\end{align}
Then $S_M\sim\mathcal N(0,M)$ under $P_\infty^M$ and $S_M\sim\mathcal N(M|\theta|,M)$ under $P_0^M$. The block likelihood ratio is increasing in $S_M$. By the Neyman--Pearson lemma, among all measurable $0\le\varphi\le1$ satisfying $\mathbb E_0^M[\varphi]\ge1/2$, the minimum of $\mathbb E_\infty^M[\varphi]$ is achieved by the likelihood-ratio upper level set having post-change probability $1/2$. Since the median of $S_M$ under $P_0^M$ is $M|\theta|$,
\begin{align}
\tau_M\!\left(\frac12\right)
=
P_\infty^M(S_M\ge M|\theta|)
=
\Phi\!\left(-|\theta|\sqrt M\right).
\label{eq:app-gaussian-transfer-exact}
\end{align}

Fix $\gamma>1$ and define
\begin{align}
m_\gamma:=\lceil2\gamma-1\rceil,
\qquad
\delta_{\gamma,\theta}:=\Phi\!\left(-|\theta|\sqrt{m_\gamma}\right).
\end{align}
Since $V_{\sf P}(\gamma)\le\gamma$, the choice $m_\gamma+1\ge2\gamma$ is valid for the exact optimizer. Lemma~\ref{lem:block-contraction}, applied with the common bound $R=\gamma$, gives
\begin{align}
e_N
=
\sum_{n=N}^{\infty}q_n(T_\gamma^\star)
\le
\bar\varepsilon_N
:=
\frac{m_\gamma}{\delta_{\gamma,\theta}}
(1-\delta_{\gamma,\theta})^{\lfloor N/m_\gamma\rfloor}.
\label{eq:app-gaussian-geometric-tail}
\end{align}
Whenever $\bar\varepsilon_N\le1$, the two-case argument in the proof of Theorem~\ref{thm:finite-convergence} applies with $e_N\le\bar\varepsilon_N$, and therefore
\begin{align}
0
\le
r_{\gamma,N+1}-V_{\sf P}(\gamma)
\le
\bar\varepsilon_N.
\label{eq:app-gaussian-crossing-tail}
\end{align}

Because $\delta_{\gamma,\theta}<1/2$ and $\lfloor N/m_\gamma\rfloor\ge N/m_\gamma-1$,
\begin{align}
(1-\delta_{\gamma,\theta})^{\lfloor N/m_\gamma\rfloor}
&\le
(1-\delta_{\gamma,\theta})^{-1}
\exp\!\left\{
\frac{N}{m_\gamma}\log(1-\delta_{\gamma,\theta})
\right\}
\\
&\le
2\exp\!\left\{
-\frac{\delta_{\gamma,\theta}}{m_\gamma}N
\right\}.
\end{align}
Combining this inequality with \eqref{eq:app-gaussian-crossing-tail} yields
\begin{align}
0
\le
r_{\gamma,N+1}-V_{\sf P}(\gamma)
\le
\frac{2m_\gamma}{\delta_{\gamma,\theta}}
\exp\!\left\{
-\frac{\delta_{\gamma,\theta}}{m_\gamma}N
\right\},
\label{eq:app-gaussian-exponential}
\end{align}
whenever the displayed right-hand side (or, more sharply, $\bar\varepsilon_N$) is at most one.

For $\gamma=5$ and $\theta=0.5$, $m_\gamma=9$ and $\delta_{\gamma,\theta}=\Phi(-1.5)\approx0.0668072$. Direct evaluation of \eqref{eq:app-gaussian-geometric-tail} gives $\bar\varepsilon_N\le1$ for $N\ge639$. Solving \eqref{eq:app-gaussian-exponential} for a target $10^{-2}$ gives $N\ge1375$. Thus the bound applies to $r_{5,N+1}$; equivalently, truncation order $K\ge1376$ guarantees an error below $10^{-2}$ under the exponential bound.

\section{Proof of Theorem \ref{thm:finite-dual-structure}}

\label{app:finite-dual-threshold}

Fix $N\ge2$ and $r>1$. For $\mathbf y\in\mathcal Y_N$, define
\begin{align}
J_N(\mathbf y)
&:=
\operatorname{ARL}_N(\mathbf y)
=
\sum_{n=0}^{N-1}q_n(\mathbf y)
\; ,
\label{eq:app-dual-J}
\\
g_\nu(\mathbf y)
&:=
A_\nu^{[N]}(\mathbf y)-r q_\nu(\mathbf y)
\; ,
\qquad
0\le\nu<N
\; .
\label{eq:app-dual-g}
\end{align}
Thus,
\begin{align}
\Gamma_N(r)
=
\max_{\mathbf y\in\mathcal Y_N}
\left\{
J_N(\mathbf y):
g_\nu(\mathbf y)\le0,\
0\le\nu<N
\right\}
\; .
\label{eq:app-dual-primal}
\end{align}
We prove the result in five steps.

\subsection{Step 1: Primal achievement and strict feasibility}

For each $1\le n<N$, regard $y_n$ as an element of
$L^\infty(P_\infty^n)$ equipped with the weak-* topology. As in the
proof of Lemma~\ref{lem:survival-compact}, the set
\begin{align}
\mathcal K_n
:=
\left\{
f\in L^\infty(P_\infty^n):
0\le f\le1
\right\}
\end{align}
is weak-* compact. The truncated survival space $\mathcal Y_N$ is
a closed convex subset of the finite product
\begin{align}
\prod_{n=1}^{N-1}\mathcal K_n
\; .
\end{align}
Using the lift $\iota_n$ defined in the proof of Lemma~\ref{lem:survival-compact}, the survival constraints are
\begin{align}
0
\le
y_n
\le
\iota_n y_{n-1}
\qquad
P_\infty^n\text{-almost surely}
\; ,
\qquad
1\le n<N
\; ,
\label{eq:app-dual-survival-constraints}
\end{align}
and the proof of Lemma~\ref{lem:survival-compact} shows that these
constraints are preserved under coordinatewise weak-* limits. Hence $\mathcal Y_N$ is compact. The quantities $q_\nu(\mathbf y)$ and
$A_\nu^{[N]}(\mathbf y)$ are continuous on $\mathcal Y_N$ under this
topology, because they are finite sums of integrals of the coordinates
$y_n$ against the $L^1(P_\infty^n)$ functions $1$ and
$L_{\nu,n}$. Therefore, the feasible set in
\eqref{eq:app-dual-primal} is compact, and the primal maximum is
achieved.

We next exhibit a strictly feasible point. Choose
\begin{align}
0
<
a
<
1-\frac{1}{r}
\end{align}
and define the observation-independent survival process
\begin{align}
y_n^{\rm s}(x_{1:n})
:=
a^n
\; ,
\qquad
0\le n<N
\; .
\label{eq:app-dual-slater-process}
\end{align}
Since
\begin{align}
\Einf[L_{\nu,n}]
=
1
\; ,
\qquad
\nu\le n
\; ,
\end{align}
we have
\begin{align}
q_\nu(\mathbf y^{\rm s})
&=
a^\nu
\; ,
\\
A_\nu^{[N]}(\mathbf y^{\rm s})
&=
\sum_{n=\nu}^{N-1}
\Einf[L_{\nu,n}a^n]
\\
&=
a^\nu
\sum_{k=0}^{N-1-\nu}a^k
<
\frac{a^\nu}{1-a}
<
r a^\nu
=
r q_\nu(\mathbf y^{\rm s})
\; .
\label{eq:app-dual-strict-feasibility}
\end{align}
Thus every truncated change-time constraint is satisfied
strictly.

\subsection{Step 2: Strong duality and dual achievement}

Let
\begin{align}
v_N
:=
\Gamma_N(r)
\end{align}
and write
\begin{align}
g(\mathbf y)
:=
\bigl(
g_0(\mathbf y),\ldots,g_{N-1}(\mathbf y)
\bigr)
\; .
\end{align}
Consider the lower image
\begin{align}
\mathcal C
:=
\left\{
\left(
g(\mathbf y)+s,\
J_N(\mathbf y)-t
\right):
\mathbf y\in\mathcal Y_N,\
s\in\mathbb R_+^N,\
t\in\mathbb R_+
\right\}
\subseteq
\mathbb R^{N+1}
\; .
\label{eq:app-dual-lower-image}
\end{align}
The set $\mathcal C$ is convex. It is also closed: the image
\begin{align}
\left\{
\left(
g(\mathbf y),J_N(\mathbf y)
\right):
\mathbf y\in\mathcal Y_N
\right\}
\end{align}
is compact, and $\mathcal C$ is the sum of this compact set and the
closed cone
\begin{align}
\mathbb R_+^N
\times
\mathbb R_-
\; .
\end{align}
Because the primal maximum is achieved, there exists a feasible
$\mathbf y^\star$ satisfying $J_N(\mathbf y^\star)=v_N$. Hence
\begin{align}
(0,v_N)
\in
\mathcal C
\; .
\end{align}
On the other hand,
\begin{align}
(0,v_N+\eta)
\notin
\mathcal C
\; ,
\qquad
\eta>0
\; ,
\end{align}
because the first coordinate equal to zero requires
$g(\mathbf y)\le0$, while the second coordinate cannot then exceed
the optimal feasible value $v_N$. Therefore, $(0,v_N)$ is a boundary
point of $\mathcal C$. The strict feasible point in \eqref{eq:app-dual-strict-feasibility}
also implies that the projection of $\mathcal C$ onto its first $N$
coordinates contains a neighborhood of the origin. Indeed,
$g_\nu(\mathbf y^{\rm s})<0$ for every $\nu$, and all sufficiently
small perturbations of the origin can be written as
$g(\mathbf y^{\rm s})+s$ with $s\in\mathbb R_+^N$. By the supporting-hyperplane theorem, there exists a nonzero vector
$(c,b)\in\mathbb R^N\times\mathbb R$ such that
\begin{align}
c^\top u+bz
\le
b v_N
\; ,
\qquad
(u,z)\in\mathcal C
\; .
\label{eq:app-dual-support}
\end{align}
Since $\mathcal C$ is closed under increasing any coordinate of $u$,
we must have
\begin{align}
c_\nu
\le
0
\; ,
\qquad
0\le\nu<N
\; .
\end{align}
Since $\mathcal C$ is closed under decreasing $z$, we must have
$b\ge0$. In fact $b>0$. Otherwise, \eqref{eq:app-dual-support} would
give $c^\top u\le0$ on a neighborhood of the origin, which forces
$c=0$ and contradicts the fact that $(c,b)$ is nonzero. Normalize $b=1$ and define
\begin{align}
\omega_\nu^\star
:=
-c_\nu
\ge
0
\; ,
\qquad
0\le\nu<N
\; .
\label{eq:app-dual-support-weights}
\end{align}
Applying \eqref{eq:app-dual-support} to
$(u,z)=(g(\mathbf y),J_N(\mathbf y))$ gives
\begin{align}
J_N(\mathbf y)
+
\sum_{\nu=0}^{N-1}
\omega_\nu^\star
\left[
r q_\nu(\mathbf y)
-
A_\nu^{[N]}(\mathbf y)
\right]
\le
v_N
\; ,
\qquad
\mathbf y\in\mathcal Y_N
\; .
\label{eq:app-dual-support-lagrangian}
\end{align}
For $\omega=(\omega_0,\ldots,\omega_{N-1})\ge0$, define
\begin{align}
\mathcal L_N(\mathbf y,\omega)
&:=
J_N(\mathbf y)
+
\sum_{\nu=0}^{N-1}
\omega_\nu
\left[
r q_\nu(\mathbf y)
-
A_\nu^{[N]}(\mathbf y)
\right]
\; ,
\label{eq:app-dual-Lagrangian}
\\
d_N(\omega)
&:=
\sup_{\mathbf y\in\mathcal Y_N}
\mathcal L_N(\mathbf y,\omega)
\; .
\label{eq:app-dual-function}
\end{align}
For every primal-feasible $\mathbf y$ and every $\omega\ge0$,
\begin{align}
\mathcal L_N(\mathbf y,\omega)
\ge
J_N(\mathbf y)
\; ,
\end{align}
so weak duality gives
\begin{align}
d_N(\omega)
\ge
v_N
\; ,
\qquad
\omega\ge0
\; .
\label{eq:app-dual-weak}
\end{align}
By \eqref{eq:app-dual-support-lagrangian},
\begin{align}
d_N(\omega^\star)
\le
v_N
\; .
\end{align}
Combining this with \eqref{eq:app-dual-weak} yields
\begin{align}
\Gamma_N(r)
=
v_N
=
\min_{\omega\ge0}
d_N(\omega)
=
d_N(\omega^\star)
\; .
\label{eq:app-dual-strong}
\end{align}
Thus, strong duality holds and the dual minimum is achieved. For completeness, strict feasibility also gives a direct coercivity
bound. Since there are finitely many constraints, there exists
$\epsilon>0$ such that
\begin{align}
r q_\nu(\mathbf y^{\rm s})
-
A_\nu^{[N]}(\mathbf y^{\rm s})
\ge
\epsilon
\; ,
\qquad
0\le\nu<N
\; .
\end{align}
Consequently,
\begin{align}
d_N(\omega)
\ge
J_N(\mathbf y^{\rm s})
+
\epsilon
\|\omega\|_1
\; ,
\qquad
\omega\ge0
\; .
\label{eq:app-dual-coercive}
\end{align}
Hence every dual sublevel set is bounded. Since $d_N$ is lower
semicontinuous as the supremum of continuous affine functions of
$\omega$, its finite sublevel sets are compact.

\subsection{Step 3: Bellman representation of the dual objective}

Fix $\omega\ge0$, set $\omega_N=0$, and define
\begin{align}
M_0
&=
\omega_0
\; ,
\\
M_n
&=
\omega_n+\Lambda_n M_{n-1}
\; ,
\qquad
1\le n<N
\; .
\label{eq:app-dual-M-recursion}
\end{align}
Equivalently,
\begin{align}
M_n
=
\sum_{\nu=0}^{n}
\omega_\nu L_{\nu,n}
\; ,
\qquad
0\le n<N
\; .
\label{eq:app-dual-M-sum}
\end{align}
Regrouping the Lagrangian exactly as in Section~\ref{sec:finite-dual} gives
\begin{align}
\mathcal L_N(\mathbf y,\omega)
=
1+(r-1)\omega_0
+
\sum_{n=1}^{N-1}
\Einf
\left[
\left(
1+r\omega_n-M_n
\right)y_n
\right]
\; .
\label{eq:app-dual-L-regrouped}
\end{align}
Set $W_N\equiv0$ and, working backward for
$n=N-1,\ldots,1$, define
\begin{align}
H_n(m)
&:=
1+r\omega_n-m
+
\Einf
\left[
W_{n+1}
\left(
\omega_{n+1}+\Lambda m
\right)
\right]
\; ,
\label{eq:app-dual-H}
\\
W_n(m)
&:=
\left[
H_n(m)
\right]_+
\; .
\label{eq:app-dual-W}
\end{align}
Here $\Lambda$ denotes a generic one-observation likelihood ratio
under $P_\infty$. Choose compatible measurable representatives of a survival process
$\mathbf y\in\mathcal Y_N$ and define its continuation kernels by
\begin{align}
\kappa_n(x_{1:n})
:=
\begin{cases}
\displaystyle
\frac{y_n(x_{1:n})}
     {y_{n-1}(x_{1:n-1})},
&
y_{n-1}(x_{1:n-1})>0,
\\[2ex]
0,
&
y_{n-1}(x_{1:n-1})=0.
\end{cases}
\label{eq:app-dual-kappa}
\end{align}
Then
\begin{align}
0
\le
\kappa_n
\le
1
\qquad\text{and}\qquad
y_n
=
y_{n-1}\kappa_n
\; ,
\label{eq:app-dual-kappa-properties}
\end{align}
where $y_{n-1}$ is understood as its lift to the
$n$-observation space. Under $P_\infty$, $\Lambda_{n+1}$ is independent of $X_{1:n}$.
Therefore,
\begin{align}
\Einf
\left[
W_{n+1}(M_{n+1})
\mid
X_{1:n}
\right]
=
\Einf
\left[
W_{n+1}
\left(
\omega_{n+1}+\Lambda M_n
\right)
\right]
\; .
\label{eq:app-dual-Markov}
\end{align}
Using \eqref{eq:app-dual-kappa-properties} and
\eqref{eq:app-dual-Markov},
\begin{align}
&
\Einf
\left[
\left(
1+r\omega_n-M_n
\right)y_n
\right]
+
\Einf
\left[
W_{n+1}(M_{n+1})y_n
\right]
\\
&\qquad=
\Einf
\left[
y_{n-1}\kappa_n H_n(M_n)
\right]
\\
&\qquad\le
\Einf
\left[
y_{n-1}W_n(M_n)
\right]
\; ,
\label{eq:app-dual-Bellman-inequality}
\end{align}
where the final inequality follows from
\begin{align}
\kappa h
\le
[h]_+
\; ,
\qquad
0\le\kappa\le1
\; .
\end{align}
Summing \eqref{eq:app-dual-Bellman-inequality} backward over
$n=1,\ldots,N-1$ gives the telescoping bound
\begin{align}
\mathcal L_N(\mathbf y,\omega)
\le
1+(r-1)\omega_0
+
\Einf
\left[
W_1
\left(
\omega_1+\Lambda_1\omega_0
\right)
\right]
\; .
\label{eq:app-dual-upper-D}
\end{align}
Conversely, equality is achieved by choosing
\begin{align}
\kappa_n
=
\begin{cases}
1,
&
H_n(M_n)>0,
\\
0,
&
H_n(M_n)<0,
\\
\text{any measurable value in }[0,1],
&
H_n(M_n)=0,
\end{cases}
\label{eq:app-dual-Bellman-policy}
\end{align}
and then defining the survival process recursively by
$y_0=1$ and $y_n=y_{n-1}\kappa_n$. This construction belongs to
$\mathcal Y_N$ and makes every inequality in
\eqref{eq:app-dual-Bellman-inequality} an equality. Hence
\begin{align}
d_N(\omega)
=
1+(r-1)\omega_0
+
\Einf
\left[
W_1
\left(
\omega_1+\Lambda_1\omega_0
\right)
\right]
=
\mathfrak D_N(\omega)
\; .
\label{eq:app-dual-D-equality}
\end{align}
Combining \eqref{eq:app-dual-strong} and
\eqref{eq:app-dual-D-equality} gives
\begin{align}
\Gamma_N(r)
=
\min_{\omega\ge0}
\mathfrak D_N(\omega)
\; .
\label{eq:app-dual-final-strong}
\end{align}
\subsection{Step 4: Bellman regularity and the threshold boundaries}

We show by backward induction that each $W_n$ is bounded, continuous,
and nonincreasing on $[0,\infty)$. The claim is immediate for
$W_N\equiv0$. Suppose it holds for $W_{n+1}$. Since
$W_{n+1}$ is bounded and continuous, bounded convergence gives
continuity of
\begin{align}
m
\longmapsto
\Einf
\left[
W_{n+1}
\left(
\omega_{n+1}+\Lambda m
\right)
\right]
\; .
\end{align}
The same mapping is nonincreasing because $W_{n+1}$ is nonincreasing
and $\Lambda\ge0$. Hence, for $m_2>m_1$,
\begin{align}
H_n(m_2)-H_n(m_1)
&=
-(m_2-m_1)
\\
&\quad+
\Einf
\left[
W_{n+1}
\left(
\omega_{n+1}+\Lambda m_2
\right)
-
W_{n+1}
\left(
\omega_{n+1}+\Lambda m_1
\right)
\right]
\\
&\le
-(m_2-m_1)
<
0
\; .
\label{eq:app-dual-H-strict}
\end{align}
Thus $H_n$ is continuous and strictly decreasing. Furthermore,
\begin{align}
H_n(0)
=
1+r\omega_n
+
\Einf
\left[
W_{n+1}(\omega_{n+1})
\right]
>
0
\; ,
\label{eq:app-dual-H-zero}
\end{align}
while
\begin{align}
H_n(m)
&\le
1+r\omega_n
+
\|W_{n+1}\|_\infty
-
m \longrightarrow
-\infty
\; ,
\qquad
m\to\infty
\; .
\label{eq:app-dual-H-infty}
\end{align}
Therefore, there exists a unique
\begin{align}
b_{n,N}(\omega)
\in
(0,\infty)
\end{align}
such that
\begin{align}
H_n\bigl(b_{n,N}(\omega)\bigr)
=
0
\; .
\label{eq:app-dual-boundary-root}
\end{align}
Since $W_n=[H_n]_+$, it follows that $W_n$ is bounded, continuous,
and nonincreasing. This completes the backward induction.

\subsection{Step 5: Saddle compatibility, threshold structure, and complementary slackness}

Let $\omega^\star$ be a dual minimizer and let $\mathbf y^\star$ be
a primal optimizer. Since $\mathbf y^\star$ is feasible,
\begin{align}
J_N(\mathbf y^\star)
&\le
\mathcal L_N(\mathbf y^\star,\omega^\star) \le
d_N(\omega^\star) =
\Gamma_N(r)
=
J_N(\mathbf y^\star)
\; .
\label{eq:app-dual-saddle-chain}
\end{align}
Hence equality holds throughout. In particular,
\begin{align}
0
&=
\mathcal L_N(\mathbf y^\star,\omega^\star)
-
J_N(\mathbf y^\star)
\\
&=
\sum_{\nu=0}^{N-1}
\omega_\nu^\star
\left[
r q_\nu(\mathbf y^\star)
-
A_\nu^{[N]}(\mathbf y^\star)
\right]
\; .
\label{eq:app-dual-complementarity-sum}
\end{align}
Every term in the final sum is nonnegative. Therefore,
\begin{align}
\omega_\nu^\star
\left[
r q_\nu(\mathbf y^\star)
-
A_\nu^{[N]}(\mathbf y^\star)
\right]
=
0
\; ,
\qquad
0\le\nu<N
\; ,
\label{eq:app-dual-complementarity}
\end{align}
which proves complementary slackness. Equality in
\eqref{eq:app-dual-saddle-chain} also shows that
$\mathbf y^\star$ maximizes
$\mathcal L_N(\cdot,\omega^\star)$ over $\mathcal Y_N$.

Choose compatible measurable representatives of $\mathbf y^\star$
and let $\kappa_n^\star$ denote the corresponding continuation
kernels defined as in \eqref{eq:app-dual-kappa}. The Bellman
calculation above yields the exact duality-gap identity
\begin{align}
0
&=
d_N(\omega^\star)
-
\mathcal L_N(\mathbf y^\star,\omega^\star)
=
\sum_{n=1}^{N-1}
\Einf
\left[
y_{n-1}^\star
\left\{
W_n(M_n^\star)
-
\kappa_n^\star H_n(M_n^\star)
\right\}
\right]
\; .
\label{eq:app-dual-gap-identity}
\end{align}
Each integrand in \eqref{eq:app-dual-gap-identity} is nonnegative.
Therefore, $P_\infty^n$-almost surely on histories with
$y_{n-1}^\star>0$,
\begin{align}
H_n(M_n^\star)>0
&\quad\Longrightarrow\quad
\kappa_n^\star=1
\; ,
\label{eq:app-dual-kappa-continue}
\\
H_n(M_n^\star)<0
&\quad\Longrightarrow\quad
\kappa_n^\star=0
\; .
\label{eq:app-dual-kappa-stop}
\end{align}
For the minimizing vector $\omega^\star$, define
\begin{align}
b_{n,N}
:=
b_{n,N}(\omega^\star)
\; .
\end{align}
Since $H_n$ is strictly decreasing and has its unique zero at
$b_{n,N}$,
\begin{align}
M_n^\star
<
b_{n,N}
&\quad\Longrightarrow\quad
\kappa_n^\star
=
1
\qquad
\text{(continue)}
\; ,
\label{eq:app-dual-threshold-continue}
\\
M_n^\star
>
b_{n,N}
&\quad\Longrightarrow\quad
\kappa_n^\star
=
0
\qquad
\text{(stop)}
\; .
\label{eq:app-dual-threshold-stop}
\end{align}
On the tie set
\begin{align}
M_n^\star
=
b_{n,N}
\end{align}
we have $H_n(M_n^\star)=0$, so either action has the same
Lagrangian value. The continuation kernel inherited from the primal
optimizer,
\begin{align}
\kappa_n^\star(x_{1:n})
:=
\begin{cases}
\displaystyle
\frac{
y_n^\star(x_{1:n})
}{
y_{n-1}^\star(x_{1:n-1})
},
&
y_{n-1}^\star(x_{1:n-1})>0,
\\[2ex]
0,
&
y_{n-1}^\star(x_{1:n-1})=0,
\end{cases}
\label{eq:app-dual-tie-kernel}
\end{align}
is measurable, takes values in $[0,1]$, and preserves both primal feasibility and optimality. It may depend on the complete observation
history on the tie set. Histories with zero surviving mass can be modified arbitrarily.

Thus, the single primal optimizer $\mathbf y^\star$ forms a saddle pair with $\omega^\star$ and, simultaneously for every
$1\le n<N$, its off-boundary decision on histories satisfying
$y_{n-1}^\star>0$ is determined by the recursively updated statistic $M_n^\star$ and the time-dependent boundary $b_{n,N}$. On boundary
ties, the rule uses the saddle-compatible continuation kernel
inherited from $\mathbf y^\star$; histories carrying zero surviving mass may be defined arbitrarily. Together with \eqref{eq:app-dual-final-strong} and \eqref{eq:app-dual-complementarity}, this proves the theorem.

\section{Proof of Theorem~\ref{thm:floor-bound}}
\label{app:floor-bound}

The significance of Assumption~\ref{ass:likelihood-floor} appears after
a change at time $\nu$. Over the next $k$ observations,
\begin{align}
L_{\nu,\nu+k}
=
\prod_{j=\nu+1}^{\nu+k}\Lambda_j
\ge
c^k
\qquad
P_\infty\text{-almost surely}
\; .
\label{eq:block-likelihood-floor}
\end{align}
Hence, for every admissible stopping rule $T$,
\begin{align}
\mathbb P_\nu
\left(
T>\nu+k
\mid
T>\nu
\right)
&=
\frac{
\Einf
\left[
L_{\nu,\nu+k}
\mathbf 1_{\{T>\nu+k\}}
\right]
}{
q_\nu(T)
}
\ge
c^k
\frac{q_{\nu+k}(T)}{q_\nu(T)}
\; .
\label{eq:floor-survival-transfer}
\end{align}
Summing these conditional survival probabilities gives the following
reduction.

\begin{lemma}[Likelihood-floor reduction]
\label{lem:floor-reduction}
Under Assumption~\ref{ass:likelihood-floor}, every admissible stopping
rule $T$ satisfies
\begin{align}
D_\nu(T)
\ge
\frac{1}{q_\nu(T)}
\sum_{k=0}^{\infty}
c^kq_{\nu+k}(T) \; ,
\qquad
\nu\ge0
\; .
\label{eq:floor-delay-reduction}
\end{align}
\end{lemma}

\begin{proof}
Fix a change time $\nu\ge0$. Since $T$ is admissible,
$q_\nu(T)>0$. For every $k\ge0$, Assumption~\ref{ass:likelihood-floor}
gives
\begin{align}
L_{\nu,\nu+k}
=
\prod_{j=\nu+1}^{\nu+k}\Lambda_j
\ge
c^k
\qquad
P_\infty\text{-almost surely}
\; .
\label{eq:app-floor-product}
\end{align}
Using the change-of-measure identity,
\begin{align}
\mathbb P_\nu
\left(
T>\nu+k
\mid
T>\nu
\right)
&=
\frac{
\Einf
\left[
L_{\nu,\nu+k}
\mathbf 1_{\{T>\nu+k\}}
\right]
}{
q_\nu(T)
}
\\
&\ge
c^k
\frac{
\Pinf(T>\nu+k)
}{
q_\nu(T)
}
\\
&=
c^k
\frac{
q_{\nu+k}(T)
}{
q_\nu(T)
}
\; .
\label{eq:app-floor-tail-lower}
\end{align}
Conditional on $T>\nu$, the residual delay $T-\nu$ is a positive
integer-valued random variable. Hence its conditional mean can be
written through its tail probabilities as
\begin{align}
D_\nu(T)
&=
\sum_{k=0}^{\infty}
\mathbb P_\nu
\left(
T>\nu+k
\mid
T>\nu
\right) \ge
\frac{1}{q_\nu(T)}
\sum_{k=0}^{\infty}
c^k q_{\nu+k}(T)
\; .
\end{align}
This proves \eqref{eq:floor-delay-reduction}.
\end{proof}
The importance of Lemma~\ref{lem:floor-reduction} is that its right-hand
side depends on the stopping rule only through its no-change survival sequence. To expose the resulting structure, define the discounted future survival
\begin{align}
s_\nu
:=
\sum_{k=0}^{\infty}
c^k q_{\nu+k}(T)
\; .
\label{eq:discounted-survival}
\end{align}
Then Lemma~\ref{lem:floor-reduction} simply states that
$D_\nu(T)\ge s_\nu/q_\nu(T)$. Furthermore, the sequence
$\{s_\nu\}$ satisfies the recursion
\begin{align}
s_\nu
=
q_\nu+c s_{\nu+1} \; ,
\qquad
\nu\ge0
\; ,
\label{eq:discounted-survival-recursion}
\end{align}
or equivalently $q_\nu=s_\nu-cs_{\nu+1}$.
Thus, if all Pollak delays are bounded by some common value $u$, then
the ratios $s_\nu/q_\nu$ must also be bounded by $u$. The next
deterministic result converts this pointwise restriction into an upper
bound on the total ARL.

\begin{lemma}[Extremal survival-sequence inequality]
\label{lem:floor-extremal}
Let $c\in(0,1)$ and let $\{q_\nu:\nu\geq 0\}\subset(0,1]$ be a non-increasing sequence with $q_0=1$. Then define $s_\nu$ as in
\eqref{eq:discounted-survival}. Suppose $\frac{s_\nu}{q_\nu} \le u$ for $\nu\ge 0$
for some
$1\le u<1/(1-c)$. Then
\begin{align}
\sum_{\nu=0}^{\infty}q_\nu
\le
\frac{cu}
{1-(1-c)u}
\; .
\label{eq:extremal-arl-bound}
\end{align}
\end{lemma}

\begin{proof}
Recall that
\begin{align}
s_\nu
=
\sum_{k=0}^{\infty}
c^k q_{\nu+k}
\; .
\end{align}
The definition immediately gives the recursion
\begin{align}
s_\nu
=
q_\nu+c s_{\nu+1} \; ,
\qquad
\nu\ge0
\; ,
\label{eq:app-s-recursion}
\end{align}
and therefore
\begin{align}
q_\nu
=
s_\nu-cs_{\nu+1}
\; .
\label{eq:app-q-from-s}
\end{align}
By assumption, $s_\nu/q_\nu\le u$ for every $\nu$. Since
$s_\nu\ge q_\nu$, necessarily $u\ge1$. Combining the assumed bound
with \eqref{eq:app-q-from-s} gives
\begin{align}
s_\nu
&\le
u q_\nu =
u
\left(
s_\nu-cs_{\nu+1}
\right)
\; ,
\end{align}
and hence
\begin{align}
s_{\nu+1}
\le
\rho(u)s_\nu \; ,
\qquad
\rho(u)
:=
\frac{u-1}{uc}
\; .
\label{eq:app-rho}
\end{align}
Since $u<1/(1-c)$, we have $0\le\rho(u)<1$.
Next, $q_0=1$ and \eqref{eq:app-q-from-s} give
\begin{align}
1
&=
s_0-cs_1 \ge
s_0
\left(
1-c\rho(u)
\right) =
\frac{s_0}{u}
\; ,
\end{align}
where
\begin{align}
1-c\rho(u)
=
\frac{1}{u}
\; .
\end{align}
Therefore,
\begin{align}
s_0
\le
u
\; .
\label{eq:app-s0-bound}
\end{align}
Iterating \eqref{eq:app-rho} gives
\begin{align}
s_\nu
\le
s_0\rho(u)^\nu \; ,
\qquad
\nu\ge0
\; .
\label{eq:app-s-geometric}
\end{align}
For a finite integer $m$, summing
\eqref{eq:app-q-from-s} from $\nu=0$ to $m$ yields
\begin{align}
\sum_{\nu=0}^{m}q_\nu
=
s_0
+
(1-c)
\sum_{\nu=1}^{m}s_\nu
-
c s_{m+1}
\; .
\label{eq:app-telescoping-s}
\end{align}
By \eqref{eq:app-s-geometric}, $s_{m+1}\to0$. Hence
\begin{align}
\sum_{\nu=0}^{\infty}q_\nu
&=
s_0
+
(1-c)
\sum_{\nu=1}^{\infty}s_\nu
\\
&\le
s_0
\left[
1+
(1-c)
\frac{\rho(u)}{1-\rho(u)}
\right]
\\
&\le
u
\left[
1+
(1-c)
\frac{\rho(u)}{1-\rho(u)}
\right]
\; .
\label{eq:app-arl-before-simplify}
\end{align}
Using the definition of $\rho(u)$,
\begin{align}
1+
(1-c)
\frac{\rho(u)}{1-\rho(u)}
=
\frac{c}
{1-(1-c)u}
\; .
\end{align}
Substituting this into \eqref{eq:app-arl-before-simplify} gives
\begin{align}
\sum_{\nu=0}^{\infty}q_\nu
\le
\frac{cu}
{1-(1-c)u}
\; ,
\end{align}
which proves \eqref{eq:extremal-arl-bound}.
\end{proof}

\subsection{Completion of the Proof of Theorem~\ref{thm:floor-bound}}

Consider first $c\in(0,1)$, and let $T$ be any admissible stopping
rule satisfying $\Einf[T]\ge\gamma$. If $J_{\sf P}(T)=\infty$, the result
is immediate. Hence suppose
\begin{align}
R
:=
J_{\sf P}(T)
<
\infty
\; .
\end{align}
For the discounted survival sequence
\begin{align}
s_\nu
=
\sum_{k=0}^{\infty}
c^kq_{\nu+k}(T) \; ,
\end{align}
Lemma~\ref{lem:floor-reduction} gives
\begin{align}
\frac{s_\nu}{q_\nu(T)}
\le
D_\nu(T)
\le
R \; ,
\qquad
\nu\ge0
\; .
\label{eq:app-floor-common-bound}
\end{align}
It also implies that $R\geq 1$ since $s_\nu \geq q_\nu(T)>0$.
If $R\ge1/(1-c)$, then the desired bound already holds, because
\begin{align}
\frac{\gamma}
{c+(1-c)\gamma}
<
\frac{1}{1-c}
\; .
\end{align}

It therefore remains to consider
\begin{align}
R
<
\frac{1}{1-c}
\; .
\end{align}
Applying Lemma~\ref{lem:floor-extremal} with $u=R$ gives
\begin{align}
\Einf[T]
=
\sum_{\nu=0}^{\infty}q_\nu(T)
\le
\frac{cR}
{1-(1-c)R}
\; .
\end{align}
Since $\Einf[T]\ge\gamma$,
\begin{align}
\gamma
\le
\frac{cR}
{1-(1-c)R}
\; .
\end{align}
The denominator is positive, so rearranging gives
\begin{align}
R
\ge
\frac{\gamma}
{c+(1-c)\gamma}
\; .
\end{align}
Since $R=J_{\sf P}(T)$, this proves \eqref{eq:floor-lower-bound}.
Finally, when $c=0$, the right-hand side of~\eqref{eq:floor-lower-bound} equals one. Conditional on survival to
any change time, the residual stopping delay is a positive
integer-valued random variable, and therefore $D_\nu(T)\ge1$.
Hence $J_{\sf P}(T)\ge1$, proving the endpoint case as well.
Taking the infimum over all admissible rules satisfying
$\Einf[T]\ge\gamma$ gives \eqref{eq:floor-value-lower-bound}.

\section{Equality Conditions and Proof of Theorem~\ref{thm:floor-sharpness}}

\label{app:floor-sharpness}

\subsection{Equality Conditions and Sharpness}

We start by providing some relevant preliminaries and the central part of the proof in Section~\ref{app:floor-sharpness}. The lower bound becomes especially informative when we ask what would
be required for equality. For $c\in(0,1)$, the inequalities leading
to Theorem~\ref{thm:floor-bound} are sufficiently rigid that equality
essentially determines the pre-change survival behavior of an optimal rule. For a fixed false-alarm level $\gamma>1$, define $\rho
:= 1-\frac{1}{\gamma}$

\begin{proposition}[Equality structure]
\label{prop:floor-equality}
Suppose $c\in(0,1)$, Assumption~\ref{ass:likelihood-floor} holds,
and an admissible stopping rule $T$ satisfies
\begin{align}
\Einf[T]
=
\gamma
\qquad \mbox{and} \qquad
J_{\sf P}(T)
=
R_c(\gamma)
\; .
\label{eq:floor-equality-assumption}
\end{align}
Then its no-change survival probabilities necessarily satisfy
\begin{align}
q_n(T)
=
\rho^n \; ,
\qquad
n\ge0
\; .
\label{eq:geometric-survival-equality}
\end{align}
Consequently, its aggregate pre-change stopping hazard is constant:
\begin{align}
\Pinf
\left(
T=n+1
\mid
T>n
\right)
=
\frac{1}{\gamma} \; ,
\qquad
n\ge0
\; .
\label{eq:aggregate-geometric-hazard}
\end{align}
Furthermore, let $\mathbf y=\{y_n:n\ge0\}$ denote the survival process
induced by $T$. Then equality requires
\begin{align}
\bigl(
\Lambda(X_n)-c
\bigr)
y_n(X_{1:n})
=
0
\qquad
P_\infty^n\text{-almost surely} \; ,
\qquad
n\ge1
\; .
\label{eq:floor-only-survival}
\end{align}
Thus, whenever the detector has positive probability of continuing
beyond time $n$, the current observation must lie on the
likelihood-ratio floor.
\end{proposition}

\begin{proof}
Define
\begin{align}
R_\star
:=
R_c(\gamma)
=
\frac{\gamma}
{c+(1-c)\gamma}
\; .
\label{eq:app-Rstar}
\end{align}
By assumption,
\begin{align}
\Einf[T]
=
\gamma
\qquad \mbox{and} \qquad
J_{\sf P}(T)
=
R_\star
\; .
\end{align}
For the discounted survival sequence
\begin{align}
s_\nu
=
\sum_{k=0}^{\infty}
c^kq_{\nu+k}(T) \; ,
\end{align}
Lemma~\ref{lem:floor-reduction} gives, for every $\nu$,
\begin{align}
\frac{s_\nu}{q_\nu(T)}
\le
D_\nu(T)
\le
R_\star
\; .
\label{eq:app-equality-common-bound}
\end{align}
Since $T$ is admissible, the sequence
$\{q_\nu(T):\nu\ge0\}$ is non-increasing, takes values in $(0,1]$,
and satisfies $q_0(T)=1$. It remains to verify that $R_\star$ lies
in the range required by Lemma~\ref{lem:floor-extremal}. Since
$c\in(0,1)$ and $\gamma>1$ we have
\begin{align}
R_\star-1 = \frac{c(\gamma-1)}
{c+(1-c)\gamma} > 0 \; , \qquad \mbox{and}
 \qquad 
\frac{1}{1-c}-R_\star
= \frac{c} {(1-c)\bigl(c+(1-c)\gamma\bigr)} > 0
\; ,
\end{align}
indicating
\begin{align}
1 < R_\star < \frac{1}{1-c}
\; .
\end{align}
Together with the preceding bound
$s_\nu/q_\nu(T)\le R_\star$, all the assumptions of
Lemma~\ref{lem:floor-extremal} are satisfied with $u=R_\star$.
Define
\begin{align}
\rho_\star
:=
\frac{R_\star-1}
{R_\star c}
\; .
\label{eq:app-rho-star-def}
\end{align}
Direct substitution of \eqref{eq:app-Rstar} gives
\begin{align}
\rho_\star
=
\frac{\gamma-1}{\gamma}
=
1-\frac1\gamma
=
\rho
\; .
\label{eq:app-rho-star}
\end{align}
Furthermore,
\begin{align}
R_\star
\left(
1-c\rho
\right)
=
1
\; .
\label{eq:app-Rstar-identity}
\end{align}
The proof of Lemma~\ref{lem:floor-extremal} yields
\begin{align}
s_0
&\le
R_\star \; ,
\label{eq:app-equality-s0}\\
s_{\nu+1}
&\le
\rho s_\nu \; ,
\qquad
\nu\ge0 \; ,
\label{eq:app-equality-contraction}
\end{align}
and therefore
\begin{align}
\Einf[T]
&=
s_0+(1-c)\sum_{\nu=1}^{\infty}s_\nu
\le
s_0\left[1+(1-c)\frac{\rho}{1-\rho}\right]
\; ,
\\
&\le
R_\star\left[1+(1-c)\frac{\rho}{1-\rho}\right]
=
\gamma
\; .
\label{eq:app-equality-chain}
\end{align}
The left-hand side is equal to $\gamma$ by assumption. Hence equality
must hold throughout \eqref{eq:app-equality-chain}.
In particular,
\begin{align}
s_0
=
R_\star
\; .
\label{eq:app-s0-equality}
\end{align}
Furthermore, since all coefficients in
\begin{align}
s_0+(1-c)\sum_{\nu=1}^{\infty}s_\nu
\end{align}
are strictly positive and
$s_\nu\le s_0\rho^\nu$ term by term, equality of the total sums
forces
\begin{align}
s_\nu
=
R_\star\rho^\nu \; ,
\qquad
\nu\ge0
\; .
\label{eq:app-s-equality}
\end{align}
Using $q_\nu=s_\nu-cs_{\nu+1}$ together with
\eqref{eq:app-Rstar-identity},
\begin{align}
q_\nu(T)
&=
R_\star\rho^\nu
-
cR_\star\rho^{\nu+1} =
R_\star
\left(
1-c\rho
\right)
\rho^\nu =
\rho^\nu
\; .
\label{eq:app-geometric-q}
\end{align}
This proves \eqref{eq:geometric-survival-equality}. Consequently,
\begin{align}
\Pinf
\left(
T=n+1
\mid
T>n
\right)
&=
1-
\frac{q_{n+1}(T)}{q_n(T)} =
1-\rho
=
\frac1\gamma
\; ,
\end{align}
which gives \eqref{eq:aggregate-geometric-hazard}.
It remains to establish the floor-only continuation property. From
\eqref{eq:app-geometric-q},
\begin{align}
\frac{1}{q_\nu(T)}
\sum_{k=0}^{\infty}
c^kq_{\nu+k}(T)
&=
\sum_{k=0}^{\infty}
(c\rho)^k =
\frac{1}{1-c\rho}
=
R_\star
\; .
\label{eq:app-floor-reduction-equality-value}
\end{align}
Hence, for every $\nu$,
\begin{align}
R_\star
&=
\frac{1}{q_\nu(T)}
\sum_{k=0}^{\infty}
c^kq_{\nu+k}(T) \le
D_\nu(T)
\le
J_{\sf P}(T)
=
R_\star
\; .
\end{align}
Thus every inequality is an equality.
In the proof of Lemma~\ref{lem:floor-reduction}, the difference between
the $k$-th true conditional survival probability and its lower bound
is nonnegative. Since the sum of all these differences is zero, each
difference must vanish. In particular, for $k=1$,
\begin{align}
\Einf
\left[
\Lambda_{\nu+1}
\mathbf 1_{\{T>\nu+1\}}
\right]
=
c\,q_{\nu+1}(T)
\; .
\end{align}
Equivalently,
\begin{align}
\Einf
\left[
\left(
\Lambda_{\nu+1}-c
\right)
\mathbf 1_{\{T>\nu+1\}}
\right]
=
0
\; .
\label{eq:app-floor-only-event}
\end{align}
The integrand is nonnegative. Interpreting expectations involving the randomized rule under the product law specified in Section~\ref{sec:model}, we obtain
\begin{align}
\left(\Lambda_{\nu+1}-c\right)
\mathbf 1_{\{T>\nu+1\}}
=0
\qquad
(\mathbb P_\infty\otimes\mathbb P_U)\text{-almost surely}.
\end{align}
Integrating this identity with respect to the auxiliary randomization conditional on the observation history gives the survival function $y_{\nu+1}$. Hence
\eqref{eq:app-floor-only-event} is equivalently
\begin{align}
\left(
\Lambda(X_n)-c
\right)
y_n(X_{1:n})
=
0
\qquad
P_\infty^n\text{-almost surely} \; ,
\qquad
n\ge1
\; ,
\end{align}
which proves \eqref{eq:floor-only-survival}.
\end{proof}
Proposition~\ref{prop:floor-equality} identifies two distinct requirements for equality. First, the \emph{amount} of pre-change survival is fixed: the aggregate one-step continuation probability must be $\rho$. Second, the \emph{location} of this survival mass is fixed: positive continuation can occur only on observations for which the likelihood ratio is at its floor. In particular, at time one the surviving mass $q_1=\rho$ must be supported on the floor set $B$, whose pre-change probability is $\beta$. Therefore, $\beta\ge\rho$ is not only sufficient for the stationary thinning construction below; for $c>0$ it is necessary for equality in the lower bound. The restriction $c>0$ in Proposition~\ref{prop:floor-equality} is important. At $c=0$, the lower bound remains valid and can be sharp, but equality need not force the geometric structure in \eqref{eq:geometric-survival-equality}.

We next identify a simple condition under which the equality requirements can actually be implemented. Define the floor set
\begin{align}
B
:=
\left\{
x\in\mathsf X:
\Lambda(x)=c
\right\} \; ,
\label{eq:floor-set}
\end{align}
and let $\beta := P_\infty(B)$. Since the likelihood ratio equals $c$ on $B$, we have $P_0(B) = c\beta$. The desired pre-change continuation probability is $\rho$. If
$\beta\ge\rho$, the floor set contains enough pre-change probability
mass to support all of this continuation while obeying the floor-only
condition in \eqref{eq:floor-only-survival}.

\subsection{Proof of Theorem~\ref{thm:floor-sharpness}}
\label{app:floor-sharpness2}

Let
\begin{align}
\rho
=
1-\frac1\gamma
\; .
\end{align}
Since $\rho\le\beta$, the continuation probability
$\rho/\beta$ in \eqref{eq:floor-optimal-rule} belongs to
$[0,1]$, so the proposed randomized rule is well defined.
Under the no-change law, at each time the rule continues precisely when
$X_n\in B$ and $U_n\le\rho/\beta$. Therefore its one-step
continuation probability is
\begin{align}
P_\infty(B)
\frac{\rho}{\beta}
=
\beta\frac{\rho}{\beta}
=
\rho
\; .
\label{eq:app-floor-pre-cont}
\end{align}
Because the observations and auxiliary randomizations are independent
across time, the no-change survival probabilities are
\begin{align}
q_n(T_c^\star)
=
\rho^n \; ,
\qquad
n\ge0
\; .
\end{align}
Consequently,
\begin{align}
\Einf[T_c^\star]
&=
\sum_{n=0}^{\infty}\rho^n
=
\frac{1}{1-\rho}
=
\gamma
\; .
\label{eq:app-floor-sharp-arl}
\end{align}
Next, since $\Lambda=c$ on $B$,
\begin{align}
P_0(B)
&=
\Einf
\left[
\Lambda(X)\mathbf 1_{\{X\in B\}}
\right] =
cP_\infty(B)
=
c\beta
\; .
\label{eq:app-floor-B-post}
\end{align}
Hence, after the change, the one-step continuation probability is
\begin{align}
P_0(B)
\frac{\rho}{\beta}
=
c\rho
\; .
\label{eq:app-floor-post-cont}
\end{align}
Conditional on survival to any change time $\nu$, all future
observations have law $P_0$, all future auxiliary coins are fresh and
independent, and the stopping mechanism is the same at every time. Therefore, the residual delay is geometric on
$\{1,2,\ldots\}$ with stopping probability $1-c\rho$. Thus
\begin{align}
D_\nu(T_c^\star)
&=
\frac{1}{1-c\rho} =
\frac{\gamma}
{c+(1-c)\gamma} \; ,
\qquad
\nu\ge0
\; .
\label{eq:app-floor-sharp-delay}
\end{align}
In particular, $T_c^\star$ is an equalizer and
\begin{align}
J_{\sf P}(T_c^\star)
=
\frac{\gamma}
{c+(1-c)\gamma}
\; .
\end{align}
Theorem~\ref{thm:floor-bound} gives the matching lower bound over all
admissible stopping rules. Hence
\begin{align}
V_{\sf P}(\gamma)
=
\frac{\gamma}
{c+(1-c)\gamma}
\; ,
\end{align}
which proves \eqref{eq:floor-rule-equalizer}.

\subsection{Interpretation of Theorem~\ref{thm:floor-sharpness}}
The construction has a simple interpretation. Under $P_\infty$, the
one-step continuation probability is
$\beta(\rho/\beta)=\rho$, producing the required geometric
pre-change survival and ARL $\gamma$. Under $P_0$, the same
continuation probability becomes
$c\beta(\rho/\beta)=c\rho$. Thus, after a change, the residual
stopping time is geometric with mean $1/(1-c\rho)$, independently
of the change time. The resulting procedure is therefore an exact equalizer.

\section{Proof of Theorem~\ref{thm:bernoulli-exact}}
\label{app:bernoulli-exact}

The Bernoulli model provides the simplest nontrivial setting in which
the floor set and its probability are explicit. Consider
\begin{align}
P_\infty(X_n=1)
=
p
\qquad \mbox{and} \qquad
P_0(X_n=1)
=
q
\qquad \mbox{with} \qquad
0<p<q<1
\; .
\label{eq:bernoulli-model}
\end{align}
The likelihood ratio has only two values:
\begin{align}
\Lambda(0)
=
\frac{1-q}{1-p}
=:c<1
\qquad \mbox{and} \qquad
\Lambda(1)
=
\frac{q}{p}
>1
\; .
\label{eq:bernoulli-likelihoods}
\end{align}
Hence the likelihood-ratio floor is achieved on $B=\{0\}$, with $\beta = P_\infty(B) = 1-p $. The sharpness condition $\rho\le\beta$ becomes $1 < \gamma \le \frac 1 p$. Thus, throughout this range, the general floor construction gives an explicit exact optimizer.
To describe the rule, define
\begin{align}
a_\gamma
:=
\frac{1-1/\gamma}{1-p}
=
\frac{\gamma-1}{\gamma(1-p)}
\; ,
\label{eq:bernoulli-continuation-probability}
\end{align}
which belongs to $[0,1]$.  Equivalently, define the stopping
probability on a zero by
\begin{align}
\eta_\gamma
:=
1-a_\gamma
=
\frac{1-\gamma p}
{\gamma(1-p)}
\; .
\label{eq:bernoulli-randomization}
\end{align}
The resulting rule is particularly simple:
\begin{align}
X_n=1
&\quad\Longrightarrow\quad
\text{stop} \; ,
\\
X_n=0
&\quad\Longrightarrow\quad
\begin{cases}
\text{stop}, & \text{with probability }\eta_\gamma,\\
\text{continue}, & \text{with probability }1-\eta_\gamma.
\end{cases}
\label{eq:bernoulli-rule}
\end{align}
For completeness, under $P_\infty$ the one-step continuation
probability is
\begin{align}
(1-p)a_\gamma
=
\rho
=
1-\frac1\gamma
\; .
\end{align}
Thus
\begin{align}
\Einf[T_\gamma^\star]
=
\frac{1}{1-\rho}
=
\gamma
\; .
\end{align}
Under $P_0$, the one-step continuation probability is
\begin{align}
(1-q)a_\gamma
&=
\frac{1-q}{1-p}
\left(
1-\frac1\gamma
\right) =
c\rho
\; .
\end{align}
Therefore, conditional on survival to any change time, the residual
delay is geometric with mean
\begin{align}
D_\nu(T_\gamma^\star)
&=
\frac{1}{1-c\rho} =
\frac{1}{
1-
\frac{1-q}{1-p}
\left(
1-\frac1\gamma
\right)
} =
\frac{\gamma(1-p)}
{1-q+\gamma(q-p)}
\; ,
\qquad
\nu\ge0
\; .
\label{eq:app-bernoulli-delay}
\end{align}
Hence the rule is an exact equalizer. Finally, Theorem~\ref{thm:floor-bound} gives the all-rules lower bound
\begin{align}
V_{\sf P}(\gamma)
&\ge
\frac{\gamma}
{
\frac{1-q}{1-p}
+
\left(
1-\frac{1-q}{1-p}
\right)\gamma
} =
\frac{\gamma(1-p)}
{1-q+\gamma(q-p)}
\; .
\end{align}
The constructed rule achieves this bound by
\eqref{eq:app-bernoulli-delay}. Therefore,
\begin{align}
V_{\sf P}(\gamma)
=
\frac{\gamma(1-p)}
{1-q+\gamma(q-p)}
\; ,
\end{align}
which proves the theorem.

\section{Proof of  the Canonical-GSR Separation in \eqref{eq:num-gsr-value}}
\label{app:canonical-gsr-separation}

Consider the Bernoulli model \begin{align}
P_\infty(X_n=1)
&=
\frac14
\; ,
&
P_0(X_n=1)
&=
\frac12
\; .
\label{eq:app-gsr-model}
\end{align}
For the canonical GSR class, we use the conventional nonnegative head-start
condition
\begin{align}
Q([0,A))
=
1
\; ,
\label{eq:app-gsr-headstart}
\end{align}
so that $R_0\in[0,A)$ almost surely and is independent of the observations.
The one-observation likelihood ratio takes the two values
\begin{align}
\Lambda(0)
&=
\frac{1-P_0(X_n=1)}
     {1-P_\infty(X_n=1)}
=
\frac23
\; ,
&
\Lambda(1)
&=
\frac{P_0(X_n=1)}
     {P_\infty(X_n=1)}
=
2
\; .
\label{eq:app-gsr-lr-values}
\end{align}
Hence the GSR recursion can be written through the two update maps
\begin{align}
F_0(x)
&:=
\frac23(1+x)
\; ,
&
F_1(x)
&:=
2(1+x)
\; .
\label{eq:app-gsr-update-maps}
\end{align}
The zero-update map has the unique fixed point
\begin{align}
a
=
\frac{2/3}{1-2/3}
=
2
\; .
\label{eq:app-gsr-fixed-point}
\end{align}
We divide the argument according to whether the constant threshold $A$ lies
below or at-or-above this fixed point.

\subsection{Thresholds below the fixed point}
Suppose first that
\begin{align}
A
<
2
\; .
\label{eq:app-gsr-A-below}
\end{align}
For every $x\in[0,2)$, repeated zero updates satisfy
\begin{align}
F_0^{\circ n}(x)
=
2+
\left(\frac23\right)^n(x-2)
\; ,
\qquad
n\ge0
\; .
\label{eq:app-gsr-zero-iterate}
\end{align}
In particular,
\begin{align}
F_0^{\circ n}(x)
\ge
F_0^{\circ n}(0)
=
2
\left[
1-
\left(\frac23\right)^n
\right]
\; ,
\qquad
x\in[0,A)
\; .
\label{eq:app-gsr-zero-lower}
\end{align}
Since the right-hand side increases to $2>A$, there exists a finite
integer $N_A$ such that
\begin{align}
2
\left[
1-
\left(\frac23\right)^{N_A}
\right]
\ge
A
\; .
\label{eq:app-gsr-NA}
\end{align}
Thus, starting from every admissible head start $x\in[0,A)$, the all-zero
trajectory crosses the threshold by time $N_A$.

Furthermore,
\begin{align}
F_1(x)-F_0(x)
=
\frac43(1+x)
>
0
\; ,
\qquad
x\ge0
\; ,
\label{eq:app-gsr-one-dominates}
\end{align}
and both $F_0$ and $F_1$ are increasing. Hence replacing any zero in an
observation path by a one can only increase the subsequent GSR statistic
and therefore cannot delay threshold crossing. Consequently,
\begin{align}
T_A^Q
\le
N_A
\qquad
\text{almost surely}
\; .
\label{eq:app-gsr-bounded-support}
\end{align}
It follows that
\begin{align}
P_\infty(T_A^Q>N_A)
=
0
\; .
\label{eq:app-gsr-zero-survival}
\end{align}
Under the Pollak convention adopted in Definition~2.1, a rule whose
survival probability vanishes at a finite change time has infinite Pollak
risk. Therefore,
\begin{align}
A<2
\qquad\Longrightarrow\qquad
J_{\sf P}(T_A^Q)
=
\infty
\; .
\label{eq:app-gsr-below-infinite}
\end{align}
\subsection{Thresholds at or above the fixed point}

Suppose now that
\begin{align}
A
\ge
2
\; .
\label{eq:app-gsr-A-above}
\end{align}
Whenever the procedure has not yet stopped, $R_{n-1}<A$. Therefore, if
$X_n=0$,
\begin{align}
R_n
=
F_0(R_{n-1})
<
\frac23(1+A)
\le
A
\; ,
\label{eq:app-gsr-zero-no-alarm}
\end{align}
where the final inequality is equivalent to $A\ge2$. Thus an observation
equal to zero can never trigger an alarm.

Under an immediate change, define the first post-change observation equal
to one by
\begin{align}
\tau_1
:=
\inf\{n\ge1:X_n=1\}
\; .
\label{eq:app-gsr-first-one}
\end{align}
Since a zero cannot trigger an alarm,
\begin{align}
T_A^Q
\ge
\tau_1
\qquad
P_0\text{-almost surely}
\; .
\label{eq:app-gsr-lower-first-one}
\end{align}
Under $P_0$, $\tau_1$ is geometric on $\{1,2,\ldots\}$ with success
probability $1/2$. Hence
\begin{align}
\mathbb E_0[\tau_1]
=
2
\; .
\label{eq:app-gsr-first-one-mean}
\end{align}
Since $T_A^Q\ge1$, the fixed-riskconstraint corresponding to an immediate change satisfies
\begin{align}
D_0(T_A^Q)
=
\mathbb E_0[T_A^Q]
\ge
\mathbb E_0[\tau_1]
=
2
\; .
\label{eq:app-gsr-D0-lower}
\end{align}
Consequently,
\begin{align}
A\ge2
\qquad\Longrightarrow\qquad
J_{\sf P}(T_A^Q)
\ge
2
\; .
\label{eq:app-gsr-above-lower}
\end{align}
Together with \eqref{eq:app-gsr-below-infinite}, this proves that every
canonical GSR rule with finite Pollak risk satisfies
\begin{align}
J_{\sf P}(T_A^Q)
\ge
2
\; .
\label{eq:app-gsr-universal-lower}
\end{align}
\subsection{Achieving the lower bound}

It remains to show that the value $2$ is achieved by a canonical GSR rule
whose ARL is at least $\gamma$ for every $1<\gamma\le4$. Take
\begin{align}
A
=
2
\end{align}
and, for example, the conventional head start
\begin{align}
Q
=
\delta_0
\; .
\end{align}
More generally, the following argument works for every
$Q$ supported on $[0,2)$. If $x<2$, then
\begin{align}
F_0(x)
&=
\frac23(1+x)
<
2
\; ,
\\
F_1(x)
&=
2(1+x)
\ge
2
\; .
\label{eq:app-gsr-A2-maps}
\end{align}
Therefore, the threshold-$2$ GSR rule stops exactly at the first
observation equal to one:
\begin{align}
T_2^Q
=
\inf\{n\ge1:X_n=1\}
\; .
\label{eq:app-gsr-first-one-rule}
\end{align}
Under the no-change law, the success probability is $1/4$, and hence
\begin{align}
\Einf[T_2^Q]
=
4
\; .
\label{eq:app-gsr-first-one-arl}
\end{align}
Under a change after time $\nu$, conditional on $T_2^Q>\nu$, all
observations through time $\nu$ are zero and the future observations are
independent Bernoulli$(1/2)$. Thus the residual delay is geometric on
$\{1,2,\ldots\}$ with success probability $1/2$, so
\begin{align}
D_\nu(T_2^Q)
=
2
\; ,
\qquad
\nu\ge0
\; .
\label{eq:app-gsr-first-one-equalizer}
\end{align}
Therefore,
\begin{align}
J_{\sf P}(T_2^Q)
=
2
\; .
\label{eq:app-gsr-first-one-risk}
\end{align}
Since
\begin{align}
\Einf[T_2^Q]
=
4
\ge
\gamma
\; ,
\qquad
1<\gamma\le4
\; ,
\end{align}
the rule is feasible for every false-alarm level in the range considered.
Combining the lower bound in \eqref{eq:app-gsr-universal-lower} with
\eqref{eq:app-gsr-first-one-risk} yields
\begin{align}
\inf_{\substack{
T\in\mathcal G_{\rm can}\\
\Einf[T]\ge\gamma
}}
J_{\sf P}(T)
=
2
\; ,
\qquad
1<\gamma\le4
\; ,
\label{eq:app-gsr-canonical-value}
\end{align}
which proves \eqref{eq:num-gsr-value}. Finally, Theorem~\ref{thm:bernoulli-exact} gives the unrestricted Pollak value
\begin{align}
V_{\sf P}(\gamma)
=
\frac{3\gamma}{\gamma+2}
\; ,
\qquad
1<\gamma\le4
\; .
\label{eq:app-gsr-unrestricted-value}
\end{align}
Hence, for $1<\gamma<4$, the exact separation is
\begin{align}
2-V_{\sf P}(\gamma)
&=
2-
\frac{3\gamma}{\gamma+2} =
\frac{4-\gamma}{\gamma+2}
>
0
\; .
\label{eq:app-gsr-exact-gap}
\end{align}
At $\gamma=4$, the gap closes.

\section{Primal–Dual Guarantee of the PVF Algorithm}
\label{app:certified-pvfa}

The truncated primal and dual formulations can be used from opposite
sides without solving either optimization problem exactly. A
primal-feasible survival process gives a lower bound on $\Gamma_N(r)$,
whereas any nonnegative dual vector gives an upper bound. Combined with
the truncation bound in Theorem~\ref{thm:finite-frontier-bounds}, this
yields the following guarantees.

\subsection{Primal--Dual Guarantees}

Recall that $\operatorname{ARL}_N$, $A_\nu^{[N]}$, $\varepsilon_N(r)$, and $\mathfrak D_N$ are defined in \eqref{eq:finite-arl}, \eqref{eq:finite-delay-numerator}, \eqref{eq:finite-error}, and \eqref{eq:dual-objective}, respectively.
\begin{proposition}[Truncated primal--dual guarantees]
\label{prop:finite-certificates}
Fix $\gamma>1$, $N\ge\lceil\gamma\rceil$, and
$r\in(1,\gamma]$. Let $\widehat{\mathbf y}\in\mathcal Y_N$ satisfy
\begin{align}
A_\nu^{[N]}(\widehat{\mathbf y})
\le
r q_\nu(\widehat{\mathbf y})
\; ,
\qquad
0\le\nu<N
\; ,
\label{eq:certificate-primal-feasibility}
\end{align}
and let $\widehat{\boldsymbol\omega}\ge0$ be arbitrary. Then
\begin{align}
\operatorname{ARL}_N(\widehat{\mathbf y})
\le
\Gamma_N(r)
\le
\mathfrak D_N(\widehat{\boldsymbol\omega})
\; .
\label{eq:certificate-primal-dual-sandwich}
\end{align}
Consequently,
\begin{align}
\operatorname{ARL}_N(\widehat{\mathbf y})
\ge
\gamma
&\quad\Longrightarrow\quad
V_{\sf P}(\gamma)
\le
r
\; ,
\label{eq:certificate-upper-pollak}
\\
\mathfrak D_N(\widehat{\boldsymbol\omega})
+
\varepsilon_N(r)
<
\gamma
&\quad\Longrightarrow\quad
V_{\sf P}(\gamma)
>
r
\; .
\label{eq:certificate-lower-pollak}
\end{align}
In the first case, the primal guarantee can be converted into an
admissible randomized stopping rule $T$ satisfying
\begin{align}
\Einf[T]
= \gamma \qquad \mbox{and} \qquad J_{\sf P}(T)
\le r \; .
\label{eq:certificate-rule}
\end{align}
\end{proposition}

\begin{proof}
Since $\widehat{\mathbf y}$ satisfies
\eqref{eq:certificate-primal-feasibility}, it is feasible for the
truncated frontier at risk level $r$. Therefore,
\begin{align}
\operatorname{ARL}_N(\widehat{\mathbf y})
\le
\Gamma_N(r)
\; .
\label{eq:app-cert-primal-lower}
\end{align}
On the other hand, Theorem~\ref{thm:finite-dual-structure} gives
\begin{align}
\Gamma_N(r)
=
\min_{\boldsymbol\omega\ge0}
\mathfrak D_N(\boldsymbol\omega)
\le
\mathfrak D_N(\widehat{\boldsymbol\omega})
\; ,
\label{eq:app-cert-dual-upper}
\end{align}
which proves \eqref{eq:certificate-primal-dual-sandwich}. Suppose first that
\begin{align}
\operatorname{ARL}_N(\widehat{\mathbf y})
\ge
\gamma
\; .
\end{align}
Extend $\widehat{\mathbf y}$ by zero beyond time $N-1$. The resulting
zero-extended process satisfies all fixed-risk inequalities at level $r$
and has ARL at least $\gamma$. Because $\gamma>1$, its first zero-survival time is strictly larger than one; zero extension guarantees that such a finite first zero exists. Since $r>1$,
Lemma~\ref{lem:tail-completion} converts this zero-extended process into
an admissible randomized rule without changing its ARL or increasing its
Pollak risk. If the resulting ARL is strictly larger than $\gamma$,
Lemma~\ref{lem:arl-calibration} reduces it to exactly $\gamma$, again
without increasing its risk. Thus there exists an admissible $T$ such
that
\begin{align}
\Einf[T]
=
\gamma
\qquad \mbox{and} \qquad
J_{\sf P}(T)
\le
r
\; ,
\end{align}
which proves \eqref{eq:certificate-upper-pollak} and
\eqref{eq:certificate-rule}. Suppose next that
\begin{align}
\mathfrak D_N(\widehat{\boldsymbol\omega})
+
\varepsilon_N(r)
<
\gamma
\; .
\end{align}
Using Theorem~\ref{thm:finite-frontier-bounds} and
\eqref{eq:app-cert-dual-upper},
\begin{align}
\Gamma(r)
&\le
\Gamma_N(r)+\varepsilon_N(r) \le
\mathfrak D_N(\widehat{\boldsymbol\omega})
+
\varepsilon_N(r) <
\gamma
\; .
\end{align}
The exact crossing characterization in Theorem~\ref{thm:main} therefore
gives
\begin{align}
V_{\sf P}(\gamma)
>
r
\; ,
\end{align}
which proves \eqref{eq:certificate-lower-pollak}.
\end{proof}

\subsection{Two-sided Bounds}

\begin{corollary}[Two-sided bounds for the Pollak value]
\label{cor:certified-bracket}
Suppose that $1<r_L<r_U\le\gamma$ and that, for some
$N\ge\lceil\gamma\rceil$, there exist
$\boldsymbol\omega_L\ge0$ and $\mathbf y_U\in\mathcal Y_N$ such that
\begin{align}
\mathfrak D_N(\boldsymbol\omega_L)
+
\varepsilon_N(r_L)
&<
\gamma
\; ,
\label{eq:certified-bracket-lower}
\\
A_\nu^{[N]}(\mathbf y_U)
&\le
r_U q_\nu(\mathbf y_U)
\; ,
\qquad
0\le\nu<N
\; ,
\label{eq:certified-bracket-upper-feas}
\\
\operatorname{ARL}_N(\mathbf y_U)
&\ge
\gamma
\; .
\label{eq:certified-bracket-upper-arl}
\end{align}
Then
\begin{align}
r_L
<
V_{\sf P}(\gamma)
\le
r_U
\; .
\label{eq:certified-pollak-bracket}
\end{align}
Furthermore, $\mathbf y_U$ can be converted into an admissible stopping
rule $T_U$ satisfying
\begin{align}
\Einf[T_U]
=
\gamma
\qquad \mbox{and} \qquad
0
\le
J_{\sf P}(T_U)-V_{\sf P}(\gamma)
<
r_U-r_L
\; .
\label{eq:certified-rule-gap}
\end{align}
For every $\eta>0$, there exists a sufficiently large truncation depth for which
\begin{align}
r_U-r_L
<
\eta
\; .
\label{eq:certified-arbitrary-width}
\end{align}
\end{corollary}

\begin{proof}
For the corollary,
\eqref{eq:certified-bracket-lower} and
\eqref{eq:certificate-lower-pollak} give
$r_L<V_{\sf P}(\gamma)$, whereas
\eqref{eq:certified-bracket-upper-feas}--%
\eqref{eq:certified-bracket-upper-arl} and
\eqref{eq:certificate-upper-pollak} give
$V_{\sf P}(\gamma)\le r_U$. The same primal construction produces
$T_U$ with $\Einf[T_U]=\gamma$ and $J_{\sf P}(T_U)\le r_U$. Hence
\begin{align}
0
\le
J_{\sf P}(T_U)-V_{\sf P}(\gamma)
<
r_U-r_L
\; ,
\end{align}
which proves \eqref{eq:certified-rule-gap}. It remains to prove that the bracket can be made arbitrarily narrow.
Fix $\eta>0$. Choose
\begin{align}
1
<
r_L
<
V_{\sf P}(\gamma)
\qquad\text{with}\qquad
V_{\sf P}(\gamma)-r_L
<
\frac{\eta}{2}
\; .
\end{align}
Then $\Gamma(r_L)<\gamma$. By
Theorem~\ref{thm:finite-frontier-bounds} and
$\varepsilon_N(r_L)\to0$, for all sufficiently large $N$,
\begin{align}
\Gamma_N(r_L)+\varepsilon_N(r_L)
<
\gamma
\; .
\end{align}
Taking an exact dual minimizer at $r_L$ gives the required lower
bound. If $V_{\sf P}(\gamma)<\gamma$, choose
\begin{align}
V_{\sf P}(\gamma)
<
r_U
\le
\gamma
\qquad\text{with}\qquad
r_U-V_{\sf P}(\gamma)
<
\frac{\eta}{2}
\; .
\end{align}
Theorem~\ref{thm:finite-convergence} gives
$r_{\gamma,N}\downarrow V_{\sf P}(\gamma)$, so for all sufficiently
large $N$, $\Gamma_N(r_U)\ge\gamma$, and a primal optimizer of the truncated problem gives the upper bound. If $V_{\sf P}(\gamma)=\gamma$,
take $r_U=\gamma$ and use Lemma~\ref{lem:finite-crossing}. Enlarging
$N$ if necessary so that the same truncation depth works for both bounds gives
\begin{align}
r_U-r_L
<
\eta
\; ,
\end{align}
which proves \eqref{eq:certified-arbitrary-width}.
\end{proof}

\subsection{Validated Numerical Evaluation}
\label{app:M3}
\begin{remark}[Validated numerical evaluation]
\label{rem:validated-numerics}
For finite discrete models, the quantities in
Proposition~\ref{prop:finite-certificates} reduce to finite sums. For
continuous models, ordinary floating-point quadrature gives numerical
approximations. A rigorous
implementation may use validated quadrature, interval arithmetic, or any
other method that returns one-sided enclosures. In particular, if
\begin{align}
\underline q_n
&\le
q_n(\widehat{\mathbf y})
\; ,
&
A_\nu^{[N]}(\widehat{\mathbf y})
&\le
\overline A_\nu
\; ,
&
\mathfrak D_N(\widehat{\boldsymbol\omega})
&\le
\overline{\mathfrak D}_N
\; ,
\end{align}
then
\begin{align}
\overline A_\nu
&\le
r\underline q_\nu
\; ,
\qquad
0\le\nu<N
\; ,
&
\sum_{n=0}^{N-1}\underline q_n
&\ge
\gamma
\label{eq:validated-primal-tests}
\end{align}
ensure $V_{\sf P}(\gamma)\le r$, whereas
\begin{align}
\overline{\mathfrak D}_N
+
\varepsilon_N(r)
<
\gamma
\label{eq:validated-dual-test}
\end{align}
ensures $V_{\sf P}(\gamma)>r$.
\end{remark}

\subsection{PVFA Procedure}
\label{app:M4}

The one-sided tests in Proposition~\ref{prop:finite-certificates} are deliberately strict. Consequently, a search that repeatedly tests only the midpoint of a bracket is not guaranteed to terminate. If that midpoint happens to equal $V_{\sf P}(\gamma)$, a finite truncation need not provide either a strict lower certificate or a feasible truncated upper certificate. The following safeguarded search avoids this degeneracy.

Maintain the invariant
\begin{align}
r_L<V_{\sf P}(\gamma)\le r_U,
\label{eq:safeguarded-invariant}
\end{align}
initialized by $r_L=1$ and $r_U=\gamma$ using Theorem~\ref{thm:main} and Lemma~\ref{lem:r-gamma-positive}. At each stage, set
\begin{align}
r_1
&:=
r_L+\frac{r_U-r_L}{3},
&
r_2
&:=
r_U-\frac{r_U-r_L}{3}.
\label{eq:safeguarded-probes}
\end{align}
For each probe $r_i$, attempt the two one-sided tests:
\begin{enumerate}
\item a primal-feasible truncated process with $\operatorname{ARL}_N\ge\gamma$ certifies $V_{\sf P}(\gamma)\le r_i$;
\item a nonnegative dual vector satisfying $\mathfrak D_N(\widehat{\boldsymbol\omega})+\varepsilon_N(r_i)<\gamma$ certifies $V_{\sf P}(\gamma)>r_i$.
\end{enumerate}
We increase $N$ until at least one probe is classified by exact primal/dual values. In a validated numerical implementation, refine $N$ and the one-sided enclosures as needed; the argument below establishes structural termination for exact one-sided tests, while numerical termination additionally depends on the enclosures becoming decisive. Update every certified endpoint that preserves \eqref{eq:safeguarded-invariant}: any certified lower probe may replace $r_L$, and any certified upper probe may replace $r_U$. If both probes are classified and bracket the crossing, use both updates.

\begin{proposition}[Safeguarded outer search under exact one-sided tests]
\label{prop:safeguarded-search}
For every target width $\eta>0$, the safeguarded procedure above produces, after finitely many successful bracket contractions,
\begin{align}
r_L<V_{\sf P}(\gamma)\le r_U,
\qquad
r_U-r_L\le\eta.
\label{eq:safeguarded-output}
\end{align}
At each successful stage the bracket width is reduced by at least the factor $2/3$.
\end{proposition}

\begin{proof}
Let $V:=V_{\sf P}(\gamma)$ and suppose \eqref{eq:safeguarded-invariant} holds. The two probes in \eqref{eq:safeguarded-probes} are distinct, so at most one can equal $V$. For every fixed $r<V$, the proof of Corollary~\ref{cor:certified-bracket} shows that, for sufficiently large $N$, an exact dual minimizer satisfies
\begin{align}
\mathfrak D_N(\boldsymbol\omega_N^\star)+\varepsilon_N(r)<\gamma,
\end{align}
Hence a point strictly below $V$ is eventually certified as a lower point by the exact dual test. Similarly, for every fixed $r>V$, Theorem~\ref{thm:finite-convergence} implies $\Gamma_N(r)\ge\gamma$ for all sufficiently large $N$, so a truncated primal optimizer eventually certifies $V\le r$. Therefore at least one of the distinct probes $r_1,r_2$ is eventually classifiable, even if the other probe equals $V$ and remains undecided.

If $r_1$ is certified as an upper point, replacing $r_U$ by $r_1$ leaves width $(r_U-r_L)/3$. If $r_2$ is certified as a lower point, replacing $r_L$ by $r_2$ likewise leaves one third of the old width. In the remaining useful cases, a lower certificate at $r_1$ or an upper certificate at $r_2$ reduces the width to at most $2(r_U-r_L)/3$; when both are available the new bracket $[r_1,r_2]$ has width $(r_U-r_L)/3$. Thus each successful stage preserves the invariant and contracts the width by at least a factor $2/3$. Repetition therefore reaches any prescribed $\eta>0$ after finitely many successful stages.
\end{proof}

Once $r_U-r_L\le\eta$, retain a primal rule certifying the upper endpoint. Proposition~\ref{prop:finite-certificates}, followed by terminal completion and ARL calibration, yields an admissible rule $T_U$ satisfying
\begin{align}
\Einf[T_U]
&=
\gamma,
&
J_{\sf P}(T_U)
&\le
r_U,
\end{align}
and hence
\begin{align}
0
\le
J_{\sf P}(T_U)-V_{\sf P}(\gamma)
<
\eta.
\label{eq:certified-pvfa-output}
\end{align}
Positive-probability boundary ties use the saddle-compatible continuation kernel of Theorem~\ref{thm:finite-dual-structure}. For continuous models, Remark~\ref{rem:validated-numerics} requires one-sided evaluations when the numerical output is to be interpreted as a rigorous bound rather than a floating-point approximation.

\end{document}